\documentclass[10pt,leqno]{article}

\usepackage{amsmath,amssymb,amsfonts,graphicx,bm,amsthm}
\usepackage[usenames]{xcolor}
\usepackage{soul}
\usepackage[colorlinks=true]{hyperref}
\usepackage{cleveref}
\usepackage{appendix}
\usepackage{pstricks}

\usepackage{enumitem}
\usepackage{latexsym}
\usepackage{stmaryrd}
\usepackage{pstricks}
\usepackage{aliascnt}

\numberwithin{equation}{section} 
\newtheorem{theorem}{Theorem}[section]
\newaliascnt{lemma}{theorem}  
\newtheorem{lemma}[lemma]{Lemma}  
\aliascntresetthe{lemma}  
  
\newaliascnt{definition}{theorem}  
\newtheorem{definition}[definition]{Definition}  
\aliascntresetthe{definition}  
 
\newaliascnt{corollary}{theorem}  
\newtheorem{corollary}[corollary]{Corollary}  
\aliascntresetthe{corollary}  
 
\newaliascnt{proposition}{theorem}  
\newtheorem{proposition}[proposition]{Proposition}  
\aliascntresetthe{proposition}  
 
\newaliascnt{remark}{theorem}  
\newtheorem{remark}[remark]{Remark}  
\aliascntresetthe{remark}  
 
\newaliascnt{notation}{theorem}  
\newtheorem{notation}[notation]{Notation}  
\aliascntresetthe{notation}  
 
\newaliascnt{example}{theorem}  
\newtheorem{example}[example]{Example}  
\aliascntresetthe{example}  
 
\newaliascnt{conjecture}{theorem}  
  
\aliascntresetthe{conjecture}  
 
\newaliascnt{question}{theorem}  
  
\aliascntresetthe{question}  
 
\newaliascnt{fact}{theorem}  
  
\aliascntresetthe{fact}  
 
\newaliascnt{claim}{theorem}  
  
\aliascntresetthe{claim}  
 
\newcommand{\Boxdot}{\,
	\setlength{\unitlength}{1ex}
	\begin{picture}(2,2)
	\put(0,0){$\Box$}
	\put(.50,.65){.}
	\end{picture}
	\!}

\newcommand{\NR}{\,
	{\psset{xunit=10pt,yunit=10pt,runit=.3pt}
		\begin{pspicture}(1,1)
		{
			\newrgbcolor{curcolor}{0 0 0}
			\pscustom[linewidth=1,linecolor=curcolor]
			{\newpath
				\moveto(.8,.8)
				\lineto(-.1,-.1)}
		}
		\put(0,0){$\mathcal{R}$}
		\end{pspicture}}
	\!}

\def\emptycommand{}
\def\kcal{\mathcal{K}}
\def\q{\mathrel{\sf q}}

\def\IQC{\hbox{\sf IQC}{} }
\def\CQC{\hbox{\sf CQC}{} }
\def\IPC{\hbox{\sf IPC} }
\def\HA{\hbox{\sf HA}{} }

\def\PA{\hbox{\sf PA}{} }
\def\LC{\hbox{\sf LC}{}}

\def\CP{\hbox{\sf CP}{} }

\def\NNIL{\hbox{\sf NNIL}{} }
\def\TNNIL{\hbox{\sf TNNIL}{} }

\def\K4{\hbox{\sf K4}{} }
\def\GL{\hbox{\sf GL}{} }
\def\L{\hbox{\sf L}{} }

\def\R{\mathcal{R}}
\def\Rbar{\overline{\mathcal{R}}}
\def\PL{\mathcal{PL}}
\def\PLS{\mathcal{PL}_{\sigma}}

\def\alphabar{\overline{\alpha}}
\def\betabar{\overline{\beta}}
\def\lle{\hbox{\sf{LLe}}{}^+}

\def\lles{\sf{iH}_{\sigma}}

\def\mrefute{\nVdash_\tinysub{\sf max}}

\newcommand{\brt}{\mathrel{\mbox{\textcolor{gray}{$\blacktriangleright$}}}}
\newcommand{\blt}{\mathrel{\mbox{\textcolor{gray}{$\blacktriangleleft$}}}}

\newcommand{\ra }{\rightarrow }

\newcommand{\lr }{\leftrightarrow}

\newcommand{\bo }{\Box }
\newcommand{\blrt }{\brt\!\blt }

\newcommand{\Prf}{{\sf Proof}}
\newcommand{\Prv}{{\sf Prov}}

\newcommand{\gnumber}[1]{\ulcorner#1\urcorner}
\newcommand{\tinysub}[1]{{^{_{#1}}}}

\begin{document}
\setstcolor{red}
\title{The $\Sigma_1$-Provability Logic of $\HA$}
\author{\begin{tabular}{c c}
		Mohammad Ardeshir\thanks{mardeshir@sharif.ir}  
		\quad &	Mojtaba Mojtahedi\thanks{mojtaba.mojtahedi@ut.ac.ir}\\
		Department of Mathematical Sciences  
		\quad &Department of Mathematics, \\
		Sharif University of Technology
		\quad &  Statistics and Computer Science, \\ 
		\quad &College of Sciences,  University of Tehran
\end{tabular}}

\maketitle

\begin{abstract}
In this paper we introduce a modal theory $\lles$ which is sound and complete for 
arithmetical $\Sigma_1$-interpretations in $\HA$, in other words, we will show that $\lles$ is the 
$\Sigma_1$-provability logic of $\HA$. Moreover we will show that $\lles$ is decidable.
As a by-product of these results, we show that $\HA+\Box \bot$ has de Jongh property. 
\end{abstract}

\tableofcontents

\section{Introduction}\label{sec-introduction}
As far as we know, there are at least two updated reliable sources \cite{ArtBekProv,VisBek} for current situation, historical background and motivations for {\em provability logic}.
 To be self-contained, in this introduction, we extract a brief backgrounds of provability logic from the mentioned sources for readers not much familiar with the subject. 

Provability Logic is a  modal logic in which the modal operator
$\Box$ has intended meaning of provability in some formal system. Unlike the other realms of modal logic, 
e.g. temporal logic, epistemic logic and  deontic logic, 
here in provability logic, we have a rational meaning for $\Box A$:
 \begin{center}
  ``$A$   is provable in the system $T$"
  \end{center}

The notion of  provability logic goes back essentially to K.~G\"{o}del 
\cite{Godel33} in 1933, where he intended to provide a semantics for 
 Heyting's formalization of {\em intuitionistic logic} $\IPC$.
He defined a
{\em translation}, or {\em interpretation} $\tau$ from the
propositional language to the modal language such that
\begin{center}
${\sf IPC}\vdash A\quad\quad\Longleftrightarrow \quad\quad{\sf S4}\vdash \tau(A)$.
\end{center}
The translation $\tau(A)$ adds a $\Box$ before each sub-formula of $A$.
The idea behind this translation is hidden  in the intuitionistic meaning of truth (the BHK interpretation): 
``The truth of a proposition coincides with its provability". 
Hence if one assumes  $\Box A$ as ``provability of $A$", then it is  reasonable to add a $\Box$ behind each sub-formula 
and expect to have a correspondence between the intuitionistic 
propositional calculus $\IPC$ and some classical modal logic.

On the other hand, by works of G\"{o}del in \cite{Godel}, for each arithmetical formula $A$ and
recursively axiomatizable theory ${T}$ (like ${\sf PA}$), we can
formalize the statement ``there exists a proof in ${ T}$ for
$A$"  by a sentence of the language of arithmetic, i.e.
${\sf Prov}_\tinysub{T}(\gnumber{A}):=\exists{x}\,{\sf Proof}_\tinysub{T}(x,\gnumber{A})$, 
where $\gnumber{A}$ is the code
of $A$. Now the question is whether we can find  some modal propositional theory 
such  that the $\square$ operator captures  the notion of {\em
provability} in classical mathematics. Let us restrict our attention to
the part of mathematics known as Peano Arithmetic ${\sf PA}$. Hence 
the question is to find some propositional modal theory $T_\Box$ such that:
$$T_\Box\vdash A\quad\quad \Longleftrightarrow \quad\quad \forall{*} \ \PA\vdash A^*$$
By  $(\ )^*$, we mean a mapping from the modal language to the first-order language of arithmetic, such that
\begin{itemize}
\item $p^*$ is an arithmetical first-order sentence, for any atomic 
variable $p$, and $\bot^*=\bot$,
\item $(A\circ B)^*=A^*\circ B^*$, for $\circ\in\{\vee, \wedge, \rightarrow\}$,
\item $(\square A)^*:=\exists{x}\,{\sf Proof}_\tinysub{\sf PA}(x,\gnumber{A^*})$.
\end{itemize}

 It turned out that ${\sf S4}$ is {\em not} a right candidate for
interpreting  the notion of {\em provability}, since
$\neg\square\bot$ is a theorem of ${\sf S4}$, contradicting
G\"{o}del's second incompleteness theorem (Peano Arthmetic ${\sf PA}$, does not prove its own consistency).

In 1976, R. Solovay \cite{Solovay} proved that  the right modal logic, in
which the $\square$ operator interprets the notion of {\em
provability in {\sf PA}}, is $\GL$. This modal logic is well-known as the
G\"{o}del-L\"{o}b logic, and has the following axioms and
rules:
\begin{itemize}
\item all tautologies of classical propositional logic,
\item $\Box (A\rightarrow B)\rightarrow(\Box A\rightarrow\Box B)$,
\item $\Box A\rightarrow\Box\Box A$,
\item L\"{o}b's axiom \textup{(}{\sf L}\textup{)}: $\Box(\Box A\rightarrow A)\rightarrow\Box A$,
\item Necessitation Rule: $A/\Box A$,
\item Modus ponens: $(A,A\rightarrow B)/B$.
\end{itemize}
\noindent {\bf Theorem.} {(Solovay)} 
{\em For any sentence $A$ in the language of modal logic, ${\sf
GL}\vdash A$ if and only if for all interpretations $(\ )^*$,
${\sf PA}\vdash A^*$. }
\vspace{.15cm}

There are many open problems which could be assumed as a generalization of the above theorem. 
A list of such problems could be found in \cite{VisBek}. Also a live list of open problems could be found in the
homepage of Lev Beklemishev\footnote{\url{http://www.mi.ras.ru/~bekl}}.

The question of generalizing Solovay's result from classical
theories to intuitionistic ones, such as the intuitionistic
counterpart of ${\sf PA}$, well-known as Heyting Arithmetic
${\sf HA}$, proved to be remarkably difficult \cite{ArtBekProv}. This
problem was taken up by A. Visser, D. de Jongh and their
students. The problem of axiomatizing the provability logic of
${\sf HA}$ remains a major open problem since the end of 70s
\cite{ArtBekProv}. Precisely speaking, the problem of the provability logic of $\HA$
is as follows: 
 $$\text{Find a modal theory ${\sf iH}$ such that:}\quad
 {\sf iH}\vdash A\quad\quad \Longleftrightarrow \quad\quad \forall{*} \ \HA\vdash A^*$$
 Note that in the above statement of the provability logic of $\HA$, we have $(\Box A)^*:={\sf Prov}_{_{\sf HA}}(\gnumber{A^*})$.
The following list contains important  results about the provability logic of $\HA$ with arithmetical nature:
\begin{itemize}
\item Myhill 1973 and Friedman 1975. $ {\sf iH}\nvdash \Box (A\vee B)\to(\Box A\vee \Box B)$, \cite{Myhill,Friedman75}
\item Leivant 1975. ${\sf iH}\vdash\Box(A\vee B)\to\Box(\Boxdot A\vee\Boxdot B)$, in which $\Boxdot A$ is a shorthand for
$A\wedge\Box A$, \cite{Leivant-Thesis} 
\item Visser 1981. ${\sf iH}\vdash\Box\neg\neg\Box A\to\Box\Box A$ and 
${\sf iH}\vdash \Box(\neg\neg\Box A\to\Box A)\to\Box(\Box A\vee \neg\Box A)$, \cite{VisserThes,Visser82}
\item Iemhoff 2001. Introduced a uniform axiomatization  of all known axiom schemas of  ${\sf iH}$ in an extended language 
with a bimodal operator $\rhd$. In her Ph.D. dissertation \cite{IemhoffT}, Iemhoff raised a conjecture that implies directly that her axiom system, ${\sf iPH}$,  restricted to the normal modal language, is equal to ${\sf iH}$, \cite{IemhoffT} 
\item Visser 2002. Introduced a decision algorithm for ${\sf iH}\vdash A$, for all $A$ not containing any atomic variable.
\cite{Visser02} 
\end{itemize}

In this paper, we introduce an axiomatization of a modal logic $\lles$ 
and prove the following result which partially
answers the question. 

We first show that  any $\TNNIL$-proposition\footnote{
We say that $A$ is $\TNNIL$, if any  two nested occurrences of $\to$ 
in the left are separated  by a $\Box$. For example $(p\to q)\to r$ and $\neg (p\to q)$ are not $\TNNIL$, while $p\to q$ and 
$\Box(p\to q)\to r$ are $\TNNIL$. Precise definition of $\TNNIL$-propositions, a modal variant of  $\NNIL$-propositions
\cite{Visser-Benthem-NNIL}, is in \Cref{subsubsec-TNNIL-algorithm}.
}
 $A$
is in the $\Sigma_1$-provability logic of ${\sf HA}$,
iff ${\sf i GL}+{\sf  CP}\vdash A$, where ${\sf iGL}$ is the intuitionistic
G\"odel-L\"ob's logic and ${\sf CP}$ is the {\em completeness principle} $B\to\Box B$ (we call this theory as $\LC$). 
This fact in combination with the conservativity result of
 \Cref{Theorem-TNNIL Conservativity of LC over LLe+} and also some variant of 
 Visser's $\NNIL$-algorithm in  \cite{Visser02}, implies that the $\Sigma_1$-provability logic
of ${\sf HA}$, is a decidable modal theory, that is called $\lles$ here. More precisely, we find a system $\lles$ such that
$$\lles\vdash A\quad\quad \Longleftrightarrow \quad\quad \forall{*} \ \HA\vdash A^*,$$
in which, $*$ range over all of the interpretations that $p^*$ is a $\Sigma_1$-sentence for atomic variables $p$.
The complete axiomatization of $\lles$ is in \Cref{Sec-Aximatizing-TNNIL}.
It is worth mentioning that a non-modal variant for all of the axioms of $\lles$, were  already discovered by Visser in \cite{VisserThes,Visser82,Visser02}.
He also showed in \cite{Visser02} that those variant of axioms of $\lles$ are sound for 
$\Sigma_1$ arithmetical interpretations in $\HA$.

\subsection{Inspiring examples}

In the following four examples,  we roughly explain the main roads in the paper. 
Before we continue with examples, let us review what we 
are going to do in this paper.  Our main results are soundness and completeness theorems of  $\lles$ for arithmetical $\Sigma_1$-interpretations in $\HA$. As usual, the difficult part is the completeness theorem.  The soundness part is not problematic: 
some major  part of soundness is already done by  Visser \cite{Visser02} and the rest (extended Leivant's principle) is done in \autoref{Theorem-Soundness of HA for lle+}. We are not going to talk about soundness in these examples. 
We explain how to refute some
modal proposition $A$ from the $\Sigma_1$-provability logic (and a fortiori from the provability logic) of $\HA$, i.e. we will 
find some $\Sigma_1$-interpretation $\sigma_{_{\sf HA}}$ such that $\HA\nvdash \sigma_{_{\sf HA}}(A)$. 
 The propositions which we treat here are $(p\to q) \vee (q\to p)$, 
$\Box(p\vee q)\to(\Box p\vee\Box q)$, ${\neg\neg\Box(\neg\neg p\to  p)}\to{\Box(\neg\neg p\to p)}$ and finally
$A:=\Box(p\vee q)\to [ (\Box p\to (p\vee q\vee \Box q))\vee (\Box q\to (p\vee q\vee \Box p))]$.  
 
\begin{example}\label{example00}\em 
In this example we will show that how to refute the Dummet formula $A:={(p\to q)} \vee {(q\to p)}$
 from the $\Sigma_1$-provability logic of $\HA$. 
Since $A$ is non-modal, we are actually faced with a special case of proving 
 de Jongh property for $\HA$ with $\Sigma_1$-substitutions.
C. Smory\'nsk first discovered this result  \cite{Smorynski-Thesis}. For a survey on the de Jongh property see \cite{vi1}. 
Here we explain how to use  Solovey's method \cite{Solovay} in combination with Smor\'ynski's general method for defining first-order Kripke models of $\HA$ \cite[page~372]{Smorynski-Troelstra} to deduce the de Jongh property (with $\Sigma_1$-substitutions) for $\HA$.  We are not going to 
provide all details here,  instead we explain the idea which motivated us to our main result \autoref{Theorem-Main tool}.
 First we find a Kripke model $\kcal_0\nVdash A$:

\vspace{.2in}

\begin{tabular}{r l}
\quad \quad \quad \quad  \parbox{3cm}{\includegraphics[scale=.6]{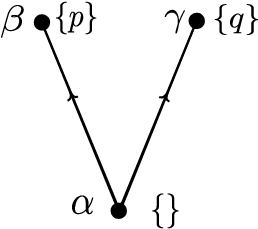}} \quad \quad \quad \quad 
& \parbox{7cm}{$\beta\Vdash p$\quad,\quad $\beta\nVdash q$\quad,\quad $\gamma\Vdash q$\quad, \quad
$\gamma\nVdash p$ \quad  

\vspace{.3in} 

$ \alpha\leq \beta,\gamma$ \quad , \quad $\alpha \nVdash p,q$
}
\end{tabular}

\vspace{.2in}

\noindent In the left and right hand side of each node we wrote the name of that node and the set of atomic variables which are 
forced at that node, respectively. 
The precise definition of  Kripke models for intuitionistic propositional logic $\IPC$, 
  came in \Cref{Sec-PropModKripke}.

Next we will find some arithmetical $\Sigma_1$-sentences 
$B$ and $C$ and also a first-order Kripke model $\kcal_1\Vdash \HA$
such that $\kcal_1$ simulates $\kcal_0$, with $B$ and $C$ playing the role of $p$ and $q$, respectively: 
\begin{center}
\includegraphics[scale=.6]{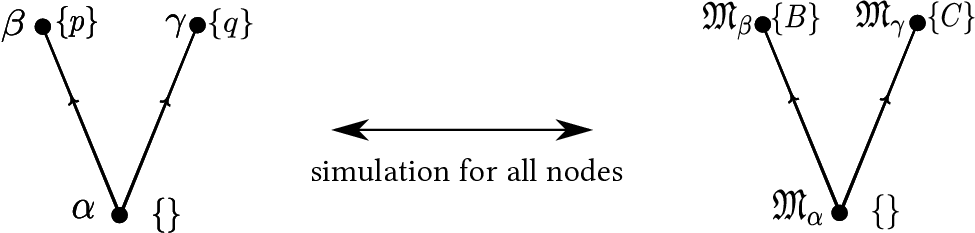}
\end{center}
In the above picture, $\mathfrak{M}_\alpha$, $\mathfrak{M}_\beta$ and $\mathfrak{M}_\gamma$ are classical 
structures assigned to the corresponding nodes. For definition of intuitionistic first-order Kripke models, see 
\Cref{sec-KripkeModelFirstOrder}.

 To explain what are these classical structures and also what are the 
 sentences $B$ and $C$, we first define a recursive function $F$ with the domain of natural numbers and with the range in 
the nodes of the  Kripke model. Let us define:
$$B:=\exists{x}(F(x)=\beta) \quad \text{and} \quad C:=\exists{x}(F(x)=\gamma)$$ 
Since $F$ is a recursive function, $B$ and $C$ are $\Sigma_1$ sentences. Moreover, 
for any $\delta\in\{\alpha,\beta,\gamma\}$, we assume the classical structures $\mathfrak{M}_\delta$ such that 
$$\mathfrak{M}_\delta \models T_\delta\quad , \quad T_\delta:=\PA+(\lim_{x\to\infty}\!\!F(x)=\delta)$$
In which $\lim_{x\to\infty}\!\!F(x)=\delta$ is defined as $\exists{x}\forall{y\geq x} F(y)=\delta$.
The function $F$ is defined  as follows: $F(0):=\alpha_0$ and $F(n+1)$ is defined to be some node $\delta>F(n)$, if there exists 
some proof (in $\PA$) with the G\"odel number less than $n+1$  for the statement 
$$  \lim_{x\to\infty}F(x)\neq\delta $$  
In other words, $F$ climbs at stage $n+1$ to the node $\delta$ if there is a witness for the inconsistency of 
$T_\delta$. This function is the same as the Solovey's function in \cite{Solovay}. The recursive definition of $F$ is such that 
although it is true that $F$ is a constant function,  $\PA$ can't prove it. Moreover, $F$ is such that for any  pair of nodes
$\delta\lneqq\delta'$, we have $T_\delta\vdash {\sf Con}(T_{\delta'})$, i.e. 
$$\PA+\lim_{x\to \infty}F(x)=\delta\vdash \neg {\sf Prov}_{_{\sf PA}}(\gnumber{\lim_{x\to \infty}F(x)\neq\delta'})$$
This guaranties the existence of the classical structures $\mathfrak{M}_\delta\models T_\delta$, such that $\kcal_1$ is a first-order 
Kripke model of $\HA$ (see \cite{Smorynski-Thesis,Smorynski-Troelstra}).   From $\mathfrak{M}_\beta\models T_\beta$,
we can  deduce that $\kcal_1,\beta\Vdash B$  and $\kcal_1,\beta\nVdash C$. From $\mathfrak{M}_\gamma\models T_\gamma$,
we can deduce $\kcal_1,\gamma\Vdash C$ and $\kcal_1,\gamma \nVdash B$. These would imply that 
$\kcal_1,\alpha\nVdash (B\to C)\vee(C\to B)$, as desired.
\end{example}

\begin{example}\label{example01}\em 
Let $A=\Box (p\vee q)\to(\Box p\vee\Box q)$. 
 J. Myhill \cite{Myhill} and H. Friedman   \cite{Friedman75}  have already shown
that there exist  some first-order arithmetical formulas $B$ and $C$ such that 
$\HA\nvdash \Box(B\vee C)\to(\Box B\vee\Box C)$, in other words,  there exist  some arithmetical substitution
$\sigma$ such that $\HA\nvdash\sigma_{_{\sf HA}}(A)$, i.e. $A$ 
does not belong to the provability logic of $\HA$. However, their proof does not provide an explicit $B$ and $C$. 
It only guarantees the existence of such arithmetical propositions.
With the methods of this paper, we will find some explicit sentences $B$ and $C$ such that 
$\HA\nvdash \Box(B\vee C)\to(\Box B\vee\Box C)$. 
\\

As usual, we  first  find some Kripke model which refutes the proposition $A$.  The
Kripke models of intuitionistic modal logic have two relations: one for intuitionistic  logic ($\leq$)
which we illustrated it in the  pictures with one  arrow in the middle line,  
and another relation for modal connective  (${\R}$)
which is illustrated with two arrows in  middle line. 
All  Kripke models in this paper have the following property: $\alpha\leq \beta\,\R\,\gamma$ implies $\alpha\,\R\,\gamma$.
 Also, in this example and other examples (\Cref{example00,example01,example02,example03}), the relations 
$\R$  and  $\leq$ are transitive, $\R\;\subseteq\;\leq$ and moreover, $\R$ is irreflexive and $\leq$ is reflexive.
In the pictures, we do not draw all the relations and 
always we assume the colsure of relations under the mentioned properties for the relations. For precise definition of 
Kripke semantics for intuitionistic modal logics, see \Cref{Sec-PropModKripke}.
The Kripke counter-model for $A$ is $\kcal_0$:

\vspace{.2in}

\begin{tabular}{r l}
\quad \quad \quad \quad  \parbox{3cm}{\includegraphics[scale=.6]{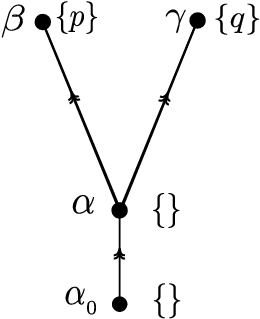}} \quad \quad 
& \parbox{9cm}{ \quad\quad$\beta\Vdash p$\quad,\quad $\beta\nVdash q$\quad,\quad $\gamma\Vdash q$\quad, \quad
$\gamma\nVdash p$ \quad  

\vspace{.25in} 

\quad\quad$\alpha\,\R\,\beta,\!\gamma$ \quad , \quad $\alpha\leq\beta,\gamma$

\vspace{.25in} 

\quad\quad$\alpha_0\,\R\,\alpha,\beta,\!\gamma$\quad , \quad  $\alpha_0\leq \alpha,\beta,\gamma$ \quad , \quad  $\alpha,\alpha_0\nVdash p,q$
}
\end{tabular}

\vspace{.2in} 

As it may be observed, the node $\alpha_0$ is not necessary. We add  this 
extra (root) node,  whenever we are not able to simulate the behaviour of the existing root of the tree.
\Cref{Theorem-Propositional Completeness LC} ensures us that always for invalid propositions, such Kripke models exist.
Next we will find some arithmetical sentences $B$ and $C$ and also a first-order Kripke model $\kcal_1\Vdash \HA$
such that $\kcal_1$ simulates $\kcal_0$ with $B$ and $C$ playing the role of $p$ and $q$, respectively: 
\begin{center}
\includegraphics[scale=.6]{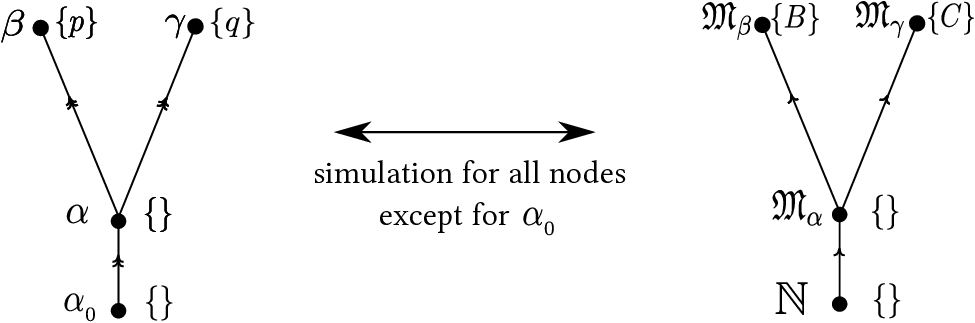}
\end{center}
In the above picture, $\mathfrak{M}_\alpha$, $\mathfrak{M}_\beta$ and $\mathfrak{M}_\gamma$ are classical 
structures assigned to the corresponding nodes and $\mathbb{N}$ indicates the standard model of arithmetic. To explain 
 these classical structures and also   the 
sentences $B$ and $C$, we first define a recursive function $F$ with the domain of natural numbers and 
with the range in the nodes of the  Kripke model. Let us define: 
$$B:=\exists{x}(F(x)=\beta) \quad \text{and} \quad C:=\exists{x}(F(x)=\gamma)$$ 
Moreover, for any $\delta\in\{\alpha,\beta,\gamma\}$, we assume the classical structures $\mathfrak{M}_\delta$ such that 
\begin{equation}\label{EQ765}
\mathfrak{M}_\delta \models T_\delta , \quad 
T_\delta:=\PA+(\lim_{x\to\infty}\!\!F(x)=\delta)+{\sf Prov}_{_{\sf HA}}(\gnumber{\varphi_{_\delta}}) ,  \quad
 \varphi_{_\alpha}:=B\vee C,\quad \varphi_{_\beta}:=\varphi_{_\gamma}:=B\wedge C
 \end{equation}

Note that $\mathfrak{M}_\delta\models {\sf Prov}_{_{\sf HA}}(\gnumber{\varphi_{_\delta}})$, and this implies
$\kcal_1,\delta\Vdash {\sf Prov}_{_{\sf HA}}(\gnumber{\varphi_{_\delta}})$. This means that the node $\delta$ in the first-order
Kripke model forces the interpretation of those boxed propositions which are forced at $\delta$ in 
the propositional Kripke model $\kcal_0$.  As we will see in \Cref{corollar-4st&2st}:  
 $$\PA+(\lim_{x\to\infty}\!\!F(x)=\delta)\vdash {\sf Prov}_{_{\sf HA}}(\gnumber{\varphi_{_\delta}})$$
 Hence we may define $T_\delta$ simply as $\PA+(\lim_{x\to\infty}\!\!F(x)=\delta)$, instead of our previous definition in
 \cref{EQ765}.

The function $F$ is defined as follows: $F(0):=\alpha_0$ and $F(n+1)$ is defined to be 
some node $\delta$  such that 
$F(n)\,\R\,\delta$, if there exists 
some proof (in $\PA$) with the G\"odel number less than $n+1$  for the statement 
$$\neg[\lim_{x\to\infty}F(x)=\delta\wedge {\sf Prov}_{_{\sf HA}}(\gnumber{\varphi_{_\delta}})]$$  
 Otherwise define $F(n+1):=F(n)$. 
In other words, $F$ climbs at stage $n+1$ to the node $\delta$ if there is a witness for the inconsistency of 
$T_\delta$ with ${\sf Prov}_{_{\sf HA}}(\gnumber{\varphi_{_\delta}})$.
What is $\varphi_{_\delta}$? The proposition $ \varphi_{_\delta}$ is the conjunction of all propositions $E$ such that $\Box E$ is forced at 
$\delta$. Since the number of such propositions are infinite, we only take care of those $E$ which are important to us, i.e. those 
which are a sub-formula of $A$. Without $\varphi_{_\delta}$, the function $F$ becomes exactly what Solovay used to prove 
his  completeness theorems for $\GL$ \cite{Solovay}, as we used in \autoref{example00}. This is not enough for our aim. 
We need to have $T_\alpha\vdash {\sf Prov}_{_{\sf HA}}(\gnumber{B\vee C})$ and generally, 
$T_\delta\vdash {\sf Prov}_{_{\sf HA}}(\gnumber{\varphi_{_\delta}})$, which is not the case without 
the clause of  $\varphi_{_\delta}$ in the recursive definition of $F$.  

By arguments in \Cref{sec-transforming},
there exists some classical structures $\mathfrak{M}_\delta\models T_\delta$ such that $\kcal_1$ is a first-order 
Kripke model of $\HA$. This implies that $\mathfrak{M}_\beta\models B$, $\mathfrak{M}_\beta\not\models C$,
$\mathfrak{M}_\gamma\models C$ and $\mathfrak{M}_\gamma\not \models B$. Since 
$T_\alpha\vdash {\sf Prov}_{_{\sf HA}}(\gnumber{B\vee C})$, we can deduce that 
$\mathfrak{M}_\alpha\models {\sf Prov}_{_{\sf HA}}(\gnumber{B\vee C})$. Also it is easy to show that 
for any first-order Kripke model $\kcal\Vdash \HA$, any node $\delta$  and arbitrary $\Sigma_1$-sentence $E$, we have 
$$\kcal,\delta\Vdash E\quad \quad \Longleftrightarrow \quad \quad \mathfrak{M}_\delta\models E$$
Since ${\sf Prov}(x)$ is a $\Sigma_1$-predicate, we can deduce 
$\kcal_1,\alpha\Vdash {\sf Prov}_{_{\sf HA}}(\gnumber{B\vee C})$ and 
$\kcal_1,\alpha\nVdash {\sf Prov}_{_{\sf HA}}(\gnumber{B})\vee {\sf Prov}_{_{\sf HA}}(\gnumber{ C})$.
Hence $\kcal_1 \nVdash {\sf Prov}_{_{\sf HA}}(\gnumber{B\vee C})\to 
({\sf Prov}_{_{\sf HA}}(\gnumber{B})\vee {\sf Prov}_{_{\sf HA}}(\gnumber{ C}))$, as desired.  We will consider 
this proposition ($A$) again in \Cref{example1} and refute it from $\Sigma_1$-provability logic of $\HA$
 with the direct use of our main theorem in \Cref{sec-transforming}. 

\end{example}

\begin{example}\label{example02}\em 
In this example, we show that how to refute $A=\neg\neg\Box(\neg\neg p\to  p)\to\Box(\neg\neg p\to p)$ from 
the provability logic of $\HA$  and also  from the $\Sigma_1$-provability logic of $\HA$. 
In  this example, the  $\TNNIL$ algorithm is involved.

The first thing is that we cannot directly refute $A$ from the provability logic of $\HA$ as we did in \Cref{example00,example01}.
The difficulty comes from the nested implications in the left hand side which are not separated by a $\Box$. Note that $\neg p$
is a shorthand for $p\to \bot$. To overcome this difficulty, we iteratively use  Visser's $\NNIL$ 
({\sf N}o {\sf N}ested {\sf I}mplication to the {\sf L}eft) approximation \cite{Visser02}
 in the modal language, i.e. inside any $\Box$ we compute the best $\NNIL$ approximation from below and replace it for the proposition. 
The approximated proposition for any modal proposition  $E$ is denoted in this paper by $E^+$. Some of 
Visser's $\NNIL$ approximations are \cite{Visser02}:
$$(\neg\neg p)^+=p\quad , \quad (\neg\neg p\to p)^+=p\vee\neg p \quad , \quad ((p\to q)\to r)^+=r\vee(p\wedge(q\to r))$$
The process of computing the approximation $(.)^+$ is complicated and we do not precisely define it in this example. 
It is explained in details in \Cref{sec-nnil}. We may briefly describe it in the following way.

\vspace{.1in}

Let $A$ be a non-modal proposition. Its $\NNIL$ approximation $A^+$, is some proposition with no nested implications to the left
such that $\IPC\vdash A^+\to A$, and for any  other $\NNIL$ proposition $B$ such that $\IPC\vdash B\to A$, we have 
$\IPC\vdash B\to A^+$. It is clear that, up to  $\IPC$-deductive equivalency,  such an approximation is unique.

\vspace{.1in}

\noindent We have the following approximation for $A$:
 $$A^+=\Box(p\vee\neg p)\vee\neg\Box(p\vee \neg p)$$
Now we can handle this simplified proposition $A^+$ as we did in \Cref{example00,example01}.
The following Kripke model  $\kcal_0$ is a counter-model for $A^+$:

\vspace{.2in} 
 
\begin{tabular}{r l}
\quad \quad \quad \quad \quad \quad  \parbox{3cm}{\includegraphics[scale=.6]{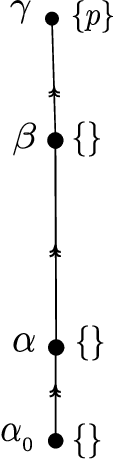}} 
& \parbox{9cm}{$\alpha_0,\alpha,\beta\nVdash p$\quad \quad $\gamma\Vdash p$\quad \quad  and 

\vspace{.3in} 

$\alpha\,\R\,\beta,\gamma$ \quad   \quad $\alpha\leq\beta,\gamma$ \quad   \quad $\beta\leq\gamma$  \quad  \quad  
$\beta\,\R\,\gamma$

\vspace{.3in} 

$\alpha_0\,\R\,\alpha,\beta,\gamma$ \quad  \quad $\alpha_0\leq\alpha,\beta,\gamma$ 
}
\end{tabular}

\vspace{.2in}

One can define the recursive function $F$ exactly the same as  \Cref{example01} with new definitions for $\varphi_{_\delta}$
and $B$:
$$ \varphi_{_\alpha}:= \top \quad , \quad \varphi_{_\beta}=\varphi_{_\gamma}:=B\vee \neg B 
 \quad , \quad   B:=(\exists{x}F(x)=\gamma)$$
Then we can define the first-order Kripke model $\kcal_1\Vdash \HA$ which simulates $\kcal_0$ in a same way as \Cref{example01}. 
Then we can  deduce that 
$\kcal_1\nVdash  {\sf Prov}_{_{\sf HA}}(\gnumber{B\vee \neg B})\vee \neg {\sf Prov}_{_{\sf HA}}(\gnumber{B\vee \neg B})$.
Although we have refuted $\Box(p\vee\neg p)\to \neg\Box(p\vee\neg p)$ from the $\Sigma_1$-provability logic of $\HA$, 
the proposition $A$ is not refuted yet. But the key point here is that by Visser's Rule,  we have 
$\HA\vdash {\neg\neg {\sf Prov}_{_{\sf HA}}(\gnumber{\neg\neg B\to B})}\to 
{{\sf Prov}_{_{\sf HA}}(\gnumber{\neg\neg B\to B})}$
if and only if 
$\HA\vdash {\sf Prov}_{_{\sf HA}}(\gnumber{\neg\neg B\to B})\vee \neg {\sf Prov}_{_{\sf HA}}(\gnumber{\neg\neg B\to B})$. 
And also by formalized Visser's Rule, we have: 
$$\HA\vdash  {\sf Prov}_{_{\sf HA}}(\gnumber{B\vee \neg B}) \leftrightarrow {\sf Prov}_{_{\sf HA}}(\gnumber{\neg\neg B\to B})$$
This will finish the refutation process, i.e. 
$\HA\nvdash \neg\neg {\sf Prov}_{_{\sf HA}}(\gnumber{\neg\neg B\to B})\to {\sf Prov}_{_{\sf HA}}(\gnumber{\neg\neg B\to B})$. 
The Visser's Rule says that for any $\Sigma_1$-sentence $B$, we have $\HA\vdash \neg\neg B\to B$ iff $\HA\vdash B\vee\neg B$. The 
proof of this rule first appeared in \cite{VisserThes} (see \Cref{corollar-NNIL properties} \cref{1corollar-NNIL properties}).

We will consider 
this proposition ($A$) again in \Cref{example2} and refute it from $\Sigma_1$-provability logic of $\HA$ with the direct use of our main theorem in \Cref{sec-transforming}. 
 %
%
\end{example}

In all of the \Cref{example00,example01,example02}, the relation for modal operator did not play an independent role, i.e. 
$\R$ and $\leq$ either where  equal (\Cref{example01,,example02}) or could be defined as equal relations (\Cref{example00}).
This made too much simplifications in the definition of the recursive function $F$. In the following example, 
there exist  some $\alpha$ and $\beta$ such that $\alpha\leq \beta$ but it is not the case that $\alpha\,\R\, \beta$.

\begin{example}\label{example03}\em 
In this example we refute the modal proposition 
$$A:=\Box(p\vee q)\to [ (\Box p\to (p\vee q\vee \Box q))\vee (\Box q\to (p\vee q\vee \Box p))]$$
Like \Cref{example00,example01},
we first find a Kripke counter-model $\kcal_0\nVdash A$:
\vspace{.2in} 
 
\begin{tabular}{r l}
\quad \quad  \parbox{3cm}{\includegraphics[scale=.6]{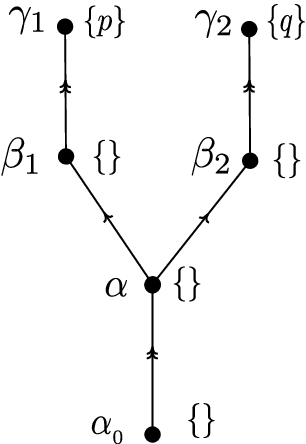}} \quad \quad \quad \quad 
& \parbox{9cm}{ $\gamma_1\Vdash p$ \quad \quad $\gamma_2\Vdash q$  \quad \quad

\vspace{.2in} 

$\alpha_0,\alpha,\beta_1,\beta_2,\gamma_2\nVdash p$ \quad \quad  $\alpha_0,\alpha,\beta_1,\beta_2,\gamma_1\nVdash q$ 

\vspace{.2in} 

$\alpha_0\leq \alpha,\beta_1,\beta_2,\gamma_1,\gamma_2$ \quad \quad $\alpha_0\,\R\, \alpha,\beta_1,\beta_2,\gamma_1,\gamma_2$ 

\vspace{.2in} 

$\alpha\leq \beta_1,\beta_2,\gamma_1,\gamma_2$ \quad \quad $\alpha\,\R\, \gamma_1,\gamma_2$ \quad \quad $\beta_i\leq\gamma_i$
\quad \quad $\beta_i\,\R\, \gamma_i$

\vspace{.2in} 

$\alpha \NR \beta_1$ \quad \quad $\alpha\NR \beta_2$
}
\end{tabular}

\vspace{.2in}

We simulate this Kripke model with a first-order Kripke model $\kcal_1\Vdash\HA$:

\begin{center}
\includegraphics[scale=.6]{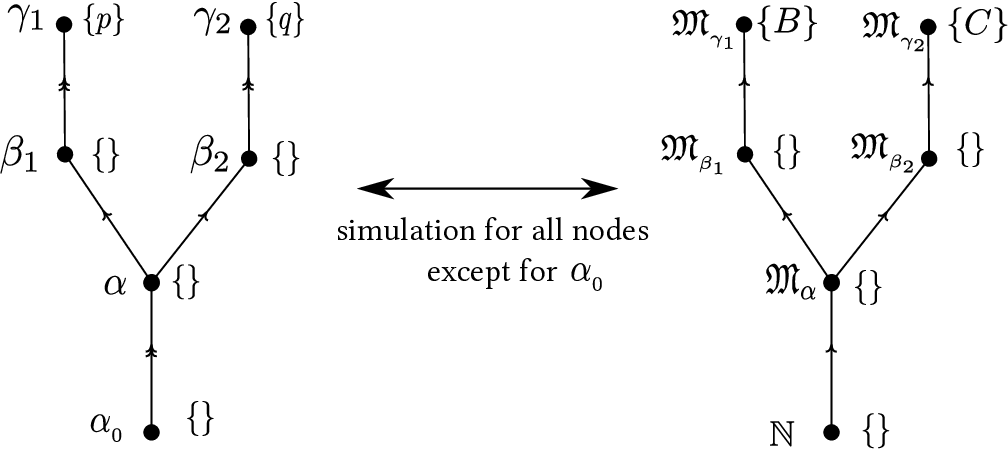}
\end{center}

We define the $\Sigma_1$-sentences $B$ and  $C$ and also the sentences $\varphi_{_\delta}$ for any $\delta\neq \alpha_0$  like before:
$$B:=(\exists{x}F(x)=\gamma_1)\quad \quad C:=(\exists{x}F(x)=\gamma_2)\quad \quad \varphi_{_\alpha}:=B\vee C \quad \quad
\varphi_{_{\beta_1}}:=B \quad \quad \varphi_{_{\beta_2}}:= C 
$$
\begin{equation}\label{eq-prop-of-firs-order-kripke-model}
 \varphi_{_{\gamma_1}}:=\varphi_{_{\gamma_2}}:=B\wedge C
\quad \quad \mathfrak{M}_\delta\models T_\delta \quad \quad 
T_\delta:=\PA+\lim_{x\to \infty}F(x)=\delta+{\sf Prov}_{_{\sf HA}}(\gnumber{\varphi_{_{\delta}}})
\end{equation}
 
 The recursive definition of $F$ is more complicated than previous examples. This is because we have really two different relations:
 $\leq$ and $\R$. The clause in recursive definition of $F$ for the treatment of $\R$ is as before. For 
 $\leq$ we use a variant of Berarducci's   primitive recursive   function in \cite{Berarducci} which he used
  for characterizing the interpretability logic of $\PA$.
\\ 
 We define $F(0):=\alpha_0$. Assume that we have defined $F(n):=\delta$, and we will define $F(n+1):=\delta'$ 
 if one of the following cases occurs, otherwise we define $F(n+1):=F(n)=\delta$.
 \begin{itemize}
  \item  $\delta\,\R\,\delta'$ and there exists some witness ( which is less than or equal to $n+1$)  for the inconsistency of $T_{\delta'}$,
  or in other words, 
    there exists some proof (in $\PA$) with the G\"odel number  $\leq n+1$  for the statement 
$$\neg[\lim_{x\to\infty}F(x)=\delta'\wedge {\sf Prov}_{_{\sf HA}}(\gnumber{\varphi_{_{\delta'}}})]$$  
  \item  All of the following conditions hold:
  \begin{itemize}
   \item $\delta\NR\delta'$ and $\delta\leq\delta'$,
   \item There exists some witness (which is less than or equal to $n+1$)  for the inconsistency of $T_{\delta'}$,
   \item The inconsistency rank of $T_{\delta'}$ (we call it  $r(\delta',n+1)$)
   is less than the inconsistency   rank of $T_\delta$ (we call it  $r(\delta,n+1)$), 
   \item $F(r(\delta',n+1))\,\R\, \delta$.
\end{itemize}  
The inconsistency rank of $T_\delta$ is defined to be the minimum $k$ such that there exists a witness (less than or equal to $n+1$)
 for the 
inconsistency of 
$$\PA_k+\lim_{x\to\infty}F(x)=\delta+ {\sf Prov}_{_{\sf HA}}(\gnumber{\varphi_{_\delta}})$$
In above definition, $\PA_k$ is the theory $I\Sigma_1$ plus induction axiom for those formulas with G\"odel number less than $k$.
 \end{itemize}
 The crucial fact about the function $F$  is that   $F$ would not climb over tree 
  (see \Cref{Lemma-limit is root}).  This fact is crucial for proving that the first-order Kripke model $\kcal_1\Vdash \HA$ exists 
  such that it fulfils the conditions in \cref{eq-prop-of-firs-order-kripke-model}.  By \Cref{corollar-4st&2st}, we  have 
  $$T_\delta\vdash {\sf Prov}_{_{\sf HA}}(\gnumber{\varphi_{_\delta}})\quad \text{for any }\delta\neq\alpha_0$$
  To simplify notations, we use $\Box A$ instead of ${\sf Prov}_{_{\sf HA}}(\gnumber{A})$ for arithmetical formula $A$.  Hence we have
  $$\mathfrak{M}_\alpha\models \Box (B\vee C) \quad  \mathfrak{M}_{\beta_1}\models \Box B \quad 
  \mathfrak{M}_{\beta_2}\models \Box C \quad  \mathfrak{M}_{\gamma_1} \models B\quad \mathfrak{M}_{\gamma_2}\models C$$
  Moreover, we have
  $$\mathfrak{M}_{\beta_1}\not \models B,C \quad 
     \quad 	\quad \mathfrak{M}_{\beta_2}\not \models B,C  $$
We need two more conditions to deduce that
 $$\kcal_1\nVdash \Box(B\vee C)\to [ (\Box B\to (B\vee C\vee \Box C))\vee (\Box C\to (B\vee C\vee \Box B))]$$
 These two conditions are $\mathfrak{M}_{\beta_1}\not\models \Box C$ and   $\mathfrak{M}_{\beta_2}\not\models \Box B$. 
 We will show in \Cref{Lemma-2st Properties of Solovay's Function}  that these conditions hold as well.  
 The proof of this fact  
 take up all \Cref{subsection-proofofmaintheorem} and there, we use  \Cref{Lemma-sigma_l translation} which is 
 the essential result in  \Cref{Sec-Arithmetic}. 
 
 Why  $\leq$ is not treated like $\R$ in recursive definition of $F$? 
 Because if we do so, we are not able to prove that 
the function $F$ is constant (\Cref{Lemma-limit is root}) and even 
 the consistency of $L=\alpha_0$ and consequently the consistency of all the theories $T_\delta$ will be lost.  
\end{example}

\subsection{What happens in classical case}
The main result of this paper in classical case, i.e. the $\Sigma_1$-provability logic of $\PA$ is already characterized 
by  A. Visser \cite{VisserThes} and is remarkably simpler than the intuitionistic case. A.~Visser showed:
$${\sf GLV}\vdash A\quad\quad \Longleftrightarrow \quad\quad \forall{*} \ \PA\vdash A^*,$$
in which, $*$ ranges over all of the  interpretations  that $p^*$ is a $\Sigma_1$-sentence for atomic variables $p$ and 
${\sf GLV}$ is $\GL$ plus the completeness axiom for atomic variables: $p\to \Box p$. For a proof of this fact see \cite{Boolos} page 135.
It is shown \cite{reduction} that 
the provability logic of $\PA$ could be reduced to its $\Sigma_1$-provability logic.

\subsection{Map of sections}
Let us explain the content of sections and their interrelationship. 
 All of the contents of this paper are minimally chosen for one major goal: 
  {\em soundness and  completeness of $\lles$ for arithmetical $\Sigma_1$-interpretations}, i.e.
 \Cref{Theorem HA-Completeness,Theorem-Soundness}. In \Cref{sec-definitions}, we give  
 definitions   of some elementary notions and also make some conventions. 
 In \Cref{Sec-Arithmetic}, we gather all the required statements with arithmetical nature.
 Most of the lemmas and definitions are for proving a 
 refinement of Leivant's principle in \Cref{Lemma-sigma_l translation} (or its simplified form in
 \Cref{Theorem-Leivant}). This will be used in \Cref{sec-transforming}. 
 In \Cref{sec-propositional}, we collect all required notions with propositional nature.
 The most crucial fact we will show in this section is that in $\lles$ (precise axiomatization of 
 $\lles$ will come in \Cref{Sec-Aximatizing-TNNIL}), 
 one could transform any modal proposition $A$
 to another proposition $A^+$ with simpler form, which is called $\TNNIL$ in this paper. Roughly speaking, in a $\TNNIL$-formula, every two nested implications in the left hand side are separated by a $\Box$.   This is done in 
\Cref{Theorem HA-NNIL approximation is propositionally equivalent}  and
\Cref{corollar HA-NNIL approximation is propositionally equivalent}.   
 Then we show that the theory $\LC$ (intuitionistic version of $\GL$ plus the axiom schema $A\to\Box A$)
 is $\TNNIL$-conservative over $\lles$ (\Cref{Theorem-TNNIL Conservativity of LC over LLe+}). It turns out that  
 $\LC$ and $\lles$ actually prove same $\TNNIL$ modal propositions (\Cref{Corollary-TNNIL-equaipotency}).
 Moreover, it is shown in \Cref{Theorem-Propositional Completeness LC} that $\LC$ is sound and complete for a special class of 
 finite  Kripke models (perfect Kripke models). In \Cref{sec-transforming}, we show that one could transform 
 a finite Kripke model of $\LC$ (with tree-frame) to a first-order Kripke model of $\HA$ (\Cref{Theorem-Main tool}). 
 This transformation is such that there is a natural correspondence between these two Kripke models.
 Finally in \Cref{Section-Sigma Provability}, we use the results of \Cref{Sec-Arithmetic,,sec-propositional,,sec-transforming},
to prove the soundness and  completeness of $\lles$ for arithmetical
 $\Sigma_1$-interpretations.
%
\section{Definitions and conventions}\label{sec-definitions}
The propositional non-modal language $\mathcal{L}_0$ contains atomic variables,
$\vee, \wedge, \ra, \bot$ and  the propositional modal language, $\mathcal{L}_\Box$ has an additional operator $\Box$. 
In this paper, the atomic propositions (in modal or non-modal language) includes 
atomic variables and $\bot$. 
For an arbitrary proposition $A$, ${\sf Sub}(A)$ is defined to be the set
of all sub-formulae of $A$, including $A$ itself. We take
${\sf Sub}(X):=\bigcup_{A\in X}{\sf Sub}(A)$ for a set of propositions $X$.
We use 
$\Boxdot A$ as a shorthand for $A\wedge\Box A$. 
The logic \IPC is  intuitionistic propositional
non-modal logic over usual propositional non-modal language.
The theory $\IPC_\Box$ is the same theory \IPC in the extended language
of propositional modal language, i.e. its language is
propositional modal language and its axioms and rules are 
same as \IPC. Because we have no axioms for $\Box$
in $\IPC_\Box$, it is obvious that $\Box A$ for each $A$,
behaves exactly like an atomic variable inside $\IPC_\Box$.
First-order intuitionistic logic is denoted
 $\IQC$ and the logic $\CQC$ is its classical closure, i.e. $\IQC$ plus
the principle of excluded middle. 
 For a set of sentences and rules
$\Gamma\cup\{A\}$ in the  propositional non-modal, propositional modal
or first-order language, $\Gamma\vdash A$ means that $A$ is
derivable from $\Gamma$ in the system $\IPC, \IPC_\Box,\IQC$,
respectively. For an arithmetical formula, $\ulcorner
A\urcorner$ represents the G\"{o}del number of $A$. For an
arbitrary  arithmetical theory $T $ with a set of $\Delta_0$-
axioms, we have the $\Delta_0$-predicate $\Prf_\tinysub{T}(x,\ulcorner A\urcorner)$, 
that is a formalization of ``$x$ is  the code of a proof
for $A$ in $T$". We also have the provability predicate
$\Prv_\tinysub{T}(\ulcorner A \urcorner):=\exists{x}\ \Prf_\tinysub{T}(x,\ulcorner A
\urcorner)$.  The set of natural numbers is denoted by
$\omega:=\{0,1,2,\ldots\}$.

\begin{definition}\label{Definition-Arithmetical substitutions}
Suppose $T$ is a recursively enumerable (r.e.) arithmetical theory and $\sigma$ is a  substitution i.e. a
function from atomic variables to arithmetical sentences. We define the interpretation $\sigma_{_T}$ which 
extend the substitution $\sigma$ to all modal propositions $A$, inductively:
\begin{itemize}
\item $\sigma_{_T}(A):=\sigma(A)$ for atomic $A$,
\item $\sigma_{_T}$ distributes over $\wedge, \vee, \ra$,
\item $\sigma_{_T}(\Box A):=\Prv_{_T}(\ulcorner\sigma_{_T}(A)\urcorner)$.
\end{itemize}
We call $\sigma$  a $\Sigma_1$-substitution, if for every
atomic $A$, $\sigma(A)$ is a $\Sigma_1$-sentence.  We also say that $\sigma_{_T}$ is a $\Sigma_1$-interpretation
if $\sigma$ is a $\Sigma_1$-substitution.
\end{definition}

\begin{definition}\label{Definition-Provability Logic}
The provability logic of a sufficiently strong theory $T$, is defined
to be a modal propositional theory $\PL(T)$ such that
$\PL(T)\vdash A$  iff for all arithmetical substitutions
$\sigma$, \ \ $T\vdash\sigma_{_T}(A)$. If we restrict the
substitutions to $\Sigma_1$-substitutions, then the new modal
theory is $\PLS(T)$.
\end{definition}

\begin{lemma}\label{Lemma-boxed as atomic}
Let $A(p_1, \ldots, p_n)$ be a non-modal proposition with
$p_i\neq p_j$ for all $0<i<j\leq n$. Then for all modal
sentences $B_1, \ldots, B_n$ 
we have:
$$\IPC\vdash A \text{\ \ iff\ \ } \IPC_\Box\vdash A[p_1|\Box B_1,\ldots,p_n|\Box B_n]$$
\end{lemma}
\begin{proof}
By simple inductions on complexity of proofs in \IPC and
$\IPC_\Box$.
\end{proof}

\noindent
We define  {\sf NOI} (No Outside Implication) as the set of modal
propositions $A$, such that any occurrence of $\ra$ is in the scope of
some $\Box$. To be able to state an extension of Leivant's
Principle (that is adequate to axiomatize $\Sigma_1$-provability
logic of $\HA$) we need a translation on the modal language which we
call  \emph{Leivant's translation}. We define it recursively as
follows:
\begin{itemize}
\item $A^l:=A$ for atomic or boxed $A$, 
\item $(A\wedge B)^l:=A^l\wedge B^l$.
\item $(A\vee B)^l:=\Boxdot A^l\vee\Boxdot B^l$.
\item $(A\ra B)^l$ is defined by cases: If $A\in {\sf NOI}$, we define
$(A\ra B)^l:=A\ra B^l$, otherwise we define $(A\ra B)^l:=A\ra B$.
\end{itemize}

\begin{definition} \label{Def-Axiom schema and modal theories}
Minimal provability logic {\sf iGL}, is the same as G\"{o}del-L\"{o}b
provability logic \GL, with all tautologies of intuitionistic logic (in modal language) instead of tautologies of classical logic.
\ ${\sf iK4}$ is ${\sf iGL}$ without  L\"{o}b's axiom. Note
that we can get rid of the necessitation rule by adding $\Box
A$ to the axioms, for each axiom $A$ in the above list. We will
use this fact later in this paper. 
We list the following axiom schemas:
\begin{itemize}
\item The Completeness Principle: $\CP:=A\ra\Box A$.
\item Restriction of Completeness Principle to atomic formulae: $\CP_{\sf a}:=p\ra\Box p$, for atomic $p$.
\item Leivant's Principle: ${\sf Le}:=\Box(B\vee C)\ra\Box (\Box B\vee C)$. \cite{Leivant-Thesis}
\item Extended Leivant's Principle: ${\sf Le}^+:=\Box A\ra\Box A^l$.
\end{itemize}
We define theories $\LC:={\sf iGL}+\CP$ and  $\lle:={\sf iGL}+{\sf Le}^++\CP_{\sf a}$. Note
that in the presence of \CP and modus ponens, the necessitation
rule is superfluous. 
\end{definition}

\section{Arithmetic}\label{Sec-Arithmetic}
In this section,
we gather some preliminaries from intuitionistic arithmetic. Mostly we will prove some refinements of 
well-known theorems such as: $\Pi_2$-conservativity of $\PA$ over $\HA$,
 G\"odel's diagnolization lemma and $\Sigma_1$-completeness of $\HA$.
Most of these preliminaries will be used to
prove a refinement of Leivant's principle 
$\Box (A\vee B)\to\Box(\Boxdot A\vee\Boxdot B)$
in  the technical \Cref{Lemma-qtranslation2}. 
 \Cref{Theorem-Leivant} states a simplified version of \Cref{Lemma-qtranslation2}.

\subsection{Some arithmetical preliminaries}
The first-order language of arithmetic contains three functions
(successor, addition and multiplication), one predicate symbol
and a constant: $(S,+,.,\leq,0)$. First-order intuitionistic
arithmetic ($\HA$) is the theory over $\IQC$ with the axioms:
 \begin{enumerate}
  \item[Q1] $S(x)\neq 0$,
  \item[Q2] $S(x)=S(y)\ra x=y$,
  \item[Q3] $y=0\vee\exists{x}\;S(x)=y$,
  \item[Q4] $x+0=x$,
  \item[Q5] $x+S(y)=S(x+y)$,
  \item[Q6] $x.0=0$,
  \item[Q7] $x.S(y)=(x.y)+x$,
  \item[Q8] $x\leq y\lr\exists{z}(z+x=y)$,
  \item[Ind:] For each formula $A(x)$:
    $${\sf Ind}(A,x):=\mathcal{UC}[A(0)\wedge\forall{x}(A(x)\ra A(S(x)))]\ra\forall{x}A(x)]$$
    In which $\mathcal{UC}(B)$ is the universal closure  of $B$.
 \end{enumerate}
Peano Arithmetic \PA\!\!, has the same axioms of \HA over
\CQC\!\!. We also define $x<y$ as $x\leq y\wedge x\neq y$.  Let $T$ be an r.e.~theory
with the set of axioms $A_1,A_2,\ldots$. It is known in the literature (see e.g.~\cite[section~2.3]{Berarducci} or
\cite[section~8.1]{Visser02})
that $T_n$ indicates the theory with the first $n$ axioms of $T$, i.e. $A_1,\ldots A_n$. In the following notation, 
 we order the axioms of $\HA$ and $\PA$ in a way which best fit the relevant lemmas and theorems in this paper.

\begin{notation}
From now on, when we are working in first-order language of
arithmetic, for a first-order sentence $A$, $\Box A$ and
$\Box^+A$ are  shorthand for 
$\Prv_\tinysub{\sf HA}(\ulcorner A\urcorner)$ and  $\Prv_\tinysub{\sf PA}(\ulcorner A\urcorner)$ , respectively. 
 Let $i{\Sigma}_1$ be the
theory $\HA$, where the induction principle is restricted to
$\Sigma_1$-formulae. We also define the theories $\HA_x$ to be
the theory with axioms of $\HA$, in which the induction principle
is restricted to  formulas satisfying at least one of the following conditions:
\begin{itemize}
\item formulas of the form $(A\ra B)\ra B$ in which $A$ and $B$ are $\Sigma_1$.
\item formulas with G\"{o}del number less than $x$.
\end{itemize}
We can define similar concept for $\PA_x$. Note that classically,
a formula of the form $(A\ra B)\ra B$ in which $A$ and $B$ are
$\Sigma_1$, is equivalent to the $\Sigma_1$-formula $A\vee B$
and hence $\PA_0$ is the well-known theory $I\Sigma_1$. We also
define $\Box_x A$ and  $\Box^+_x A$ to be $\Prv_{_{{\sf HA}_x}}(\ulcorner A\urcorner) $ 
and $\Prv_{_{{\sf PA}_x}}(\ulcorner A\urcorner)$, respectively.
\end{notation}
\noindent

\noindent
We recall that a function $f$ on $\omega:=\{0,1,2,\ldots\}$ is
recursive iff there exists some $\Sigma_1$-formula
$A_f(\bar{x},y)$ such that $\mathbb{N}\models A_f(\bar{x},y)$ iff
$f(\bar{x})=y$. It is called  provably total in $T$, iff
$T\vdash\forall\bar{x}\exists{y}A_f(\bar{x},y)$.

It is well known that all primitive recursive functions are
provably total in $I\Sigma_1$ with a $\Delta_0$-formula as
defining formula. So we may use primitive recursive function
symbols in the language of arithmetic with their defining axioms
(as far as we work in $I\Sigma_1$).

\begin{lemma}\label{Lemma-Conservativity of HA}
Let $A$, $B$ be  $\Sigma_1$-formulae such that $\PA\vdash A\ra
B$. Then $\HA\vdash A\ra B$.
\end{lemma}
\begin{proof}
Let $\PA\vdash A\to B$.  Then as it is well known in classical logic, we have $\PA\vdash \neg A\vee B$. Since $A$
and $B$ are  $\Sigma_1$,
there are some $\Delta_0$ formulas $A'(x)$ and $B'(y)$ such that $A=\exists x A'(x)$ and $B=\exists y B'(y)$. 
We may assume that 
$x$ is not free in $B'$ and $y$ is not free in $A'$.  Hence we may deduce that 
$\PA\vdash  \exists y (\neg A'(x)\vee B'(y))$.
By $\Pi_2$-conservativity of
$\PA$ over $\HA$ \cite{TD}(3.3.4), we can deduce that 
$\HA\vdash  \exists y (\neg A'(x)\vee B'(y))$. Then
we may deduce that $\HA\vdash   \exists y (A'(x)\to B'(y))$ and hence  (since $y$ is not free in $A'$)
$\HA\vdash A'(x)\to B$ and by generalization rule $\HA\vdash \forall{x}(A'(x)\to B)$. This implies that $\HA\vdash A\to B$ 
(since $x$ is not free $B$). 
\end{proof}

\begin{lemma}\label{Lemma-decidability of delta formulae}
For any $\Delta_0$-formula $A(\bar{x})$, we have
$\HA_0\vdash\forall\bar{x}(A(\bar{x})\vee\neg A(\bar{x}))$.
\end{lemma}
\begin{proof}
This is well-known in the literature.
\end{proof}

\noindent
The G\"{o}del-Gentzen translation associates a formula $A^g$ for
any formula $A$ in a first-order language, and is defined
inductively by the following items:
\begin{itemize}
\item $A^g:=A$, for atomic  $A$,
\item $(A\wedge B)^g:=A^g\wedge B^g$,
\item $(A\vee B)^g:=\neg(\neg A^g\wedge\neg B^g)$,
\item $(A\ra B)^g:=A^g\ra B^g$,
\item $(\forall{x}A)^g:=\forall{x}A^g$ ,
\item $(\exists{x}A)^g:=\neg\neg\,\exists{x}\, A^g$.
\end{itemize}

The Friedman translation associates a formula $A^C$, for an
arbitrary formula $C$ and   $A$ in  first-order
language. Roughly speaking, $A^C$ is the result of adding $C$ as a disjunct to all atomic sub-formulas of $A$.
To define $A^C$, we assume that free variables of $C$
do not appear as bound variables of $A$. It is obvious that
we can always take care of this detail by renaming bound variables of
$A$ to  fresh variables.
\begin{itemize}
\item $A^C:=A\vee C$, for atomic $A$,
\item $(A\wedge B)^C:=A^C\wedge B^C$,
\item $ (A\vee B)^C:=A^C\vee B^C$,
\item $(A\ra B)^C:=A^C\ra B^C$,
\item $(\forall{x} A)^C:=\forall{x}A^C$,
\item $(\exists{x} A)^C:=\exists{x}A^C$.
\end{itemize}

As shown in \cite{TD}, we have the following properties for
G\"{o}del-Gentzen and Friedman translations:
\begin{itemize}
\item For each $\Sigma_1$-formula $A$ in the language of arithmetic,
$\HA\vdash A^g\lr\neg\neg A$ and $\HA\vdash A^C\lr (A\vee
C)$.
\item For any $A$ in the language of arithmetic, $ \CQC\vdash A$
implies $\IQC\vdash A^g$.
\item $\HA_0$ is closed under Friedman's translation with respect to
$\Sigma_1$-formulas. i.e. for any $\Sigma_1$-formula $B$ and any
$A$, $\HA_0\vdash A$ implies $\HA_0\vdash A^B$. Actually in
\cite{TD}, this property is proved for $\HA$ instead of $\HA_0$,
but this case is very similar to that one.
\end{itemize}

\noindent
We have the following variant of \Cref{Lemma-Conservativity of
HA}.

\begin{lemma}\label{Lemma-Godel-Gentzen Translation}
For any $\Sigma_1$-formula $A$, $\PA_0\vdash A$ implies
$\HA_0\vdash A$. Hence for any $\Pi_2$-sentence $A$, $\PA_0\vdash
A$ implies $\HA_0\vdash A$.
\end{lemma}
\begin{proof}
First observe that $\PA_0\vdash B$ implies $\HA_0\vdash B^g$, by
induction on proof of $B$ in $\PA_0$. We refer the reader to
\cite{TD} for a detailed proof of this fact for $\PA$ and $\HA$
instead of $\PA_0$ and $\HA_0$. It should only be noted that for
any instance $B$ of induction over $\Sigma_1$ formulae in $\PA_0$, by
definition   of G\"{o}del-Gentzen translation, $B^g$
belongs to the axioms of $\HA_0$. 
Hence, we have $\HA_0\vdash
\neg\neg A$, and thus  $\HA_0\vdash (\neg\neg A)^A$. 
This implies
$\HA_0\vdash A$, as desired.
\end{proof}

\noindent
Consider the mapping: 
\[ F: n \mapsto A({ S}^n(0)):=A(\overbrace{{ S\ldots S}}^{n\text{ times}}(0))\]
Let $G$ be the primitive recursive function that assigns to $n$ the G\"odel number of $F(n)$.
Instead of  $G(x)$, we use  the notation $\ulcorner A(\dot{x})\urcorner$ which is common in the literature. 
We may omit the dot over variables when no
confusion is likely. 

\begin{lemma}\label{Lemma-diagonalization lemma}
For every formula $A(x,x_1\ldots,x_n)$ with free variables 
exactly as shown, there exists a formula $B(x_1,\ldots,x_n)$ such
that
$$\HA_0\vdash B(x_1,\ldots,x_n)\lr A(\ulcorner B(\dot{x}_1,\ldots,\dot{x}_n)\urcorner,x_1,\ldots,x_n)$$
Moreover, if the formula $A$ is $\Delta_0$, then $B$ is also
$\Delta_0$.
\end{lemma}
\begin{proof}
It is easy to see that the usual proof of the fixed point lemma holds in this setting.
\end{proof}
The following lemma states the $\Sigma_1$-completeness of $\HA_0$.
\begin{lemma}\label{Lemma-bounded Sigma completeness}
$\HA_0$  proves all true $\Sigma_1$ sentences. Moreover 
this argument is formalizable and provable in $\HA_0$, i.e. 
for every $\Sigma_1$-formula $A(x_1,\ldots,x_k)$ we have
$\HA_0\vdash{
A(x_1,\ldots,x_k)\ra\Box_0A(\dot{x}_1,\ldots,\dot{x}_k)}$.
\end{lemma}
\begin{proof}
It is a well-known fact that any true (in the standard model
$\mathbb{N}$) $\Sigma_1$-sentence is provable in $i{\Sigma}_1$.
Moreover this argument is constructive and formalizable in
$i{\Sigma}_1$.
\end{proof}

\begin{lemma}\label{Lemma-Reflection}
For every formula $A$, we have $\PA\vdash \forall{x}\
\Box^+(\Box^+_xA\ra A)$ and  $\HA\vdash\forall{x}\
\Box(\Box_xA\ra A)$.
\end{lemma}
\begin{proof}
The case of $\PA$ is well known. For the case $\HA$, see
\cite{Smorynski-Troelstra} or  \cite[Theorem 8.1]{Visser02}.
\end{proof}

\subsubsection{Coding of finite sequences}
We use  some fixed 
method for encoding of finite sequences and use $\langle
x_1,\ldots,x_n\rangle$ as the code of the finite sequence
$(x_1,\ldots,x_n)$. We assume here that the encoding is a one-one correspondence between 
natural numbers and the associated finite sequences.  
 For details on coding of finite sequences, we
refer the reader to \cite{Smorynski-Book}, Chapter 0. 

\vspace{.1in}

\noindent Let $x=\langle x_0, x_1,\ldots,x_n\rangle$ and $y=\langle y_0,y_1,\ldots,y_m\rangle$. 
The following notations are used in this paper:
\begin{itemize}
\item ${\sf lth}(x)$ is defined as the length of the sequence with the code $x$, i.e. here ${\sf lth}(x):=n+1$,
\item $x*y:=\langle x_0,\ldots,x_n,y_0,\ldots,y_m\rangle$,
\item $(x)_i$ is defined  (if $i<{\sf lth}(x)$) as the $i$-th element in the sequence with the 
code $x$, i.e. here $(x)_i:=x_{i}$. If also $i\geq{\sf lth}(x)$, we define $(x)_i:=0$,
\item   $\hat{x}$ is  defined as the final element of the sequence with the code $x$, i.e. here $\hat{x}:=(x)_{{\sf lth}(x)\dot{-}1}$,
\item $x$ is an initial segment of $y$ ($x\subseteq_{\sf i}y$) if ${\sf lth}(x)\leq{\sf lth}(y)$ and for all $ j< {\sf lth}(x)$, we have 
$(x)_j=(y)_j$.
\end{itemize}

\subsubsection{Kripke models of \HA}\label{sec-KripkeModelFirstOrder}
A first-order Kripke model for $\HA$ is a triple $\kcal=(K,<,\mathfrak{M})$ such that:
\begin{itemize}
\item  The frame of $\kcal$, i.e. $(K,<)$,  is a non-empty partially ordered set,
\item $\mathfrak{M}$ is  a function from $K$ to the first-order classical structures for the language 
of the arithmetic, i.e. $\mathfrak{M}(\alpha)$ is a first-order classical structure, for each $\alpha\in K$,
\item For any $\alpha\leq\beta\in K$, $\mathfrak{M}(\alpha)$ is a weak  substructure of
$\mathfrak{M}(\beta)$.   
\end{itemize}
For any  $\alpha\in K$ and  first-order formula $A\in\mathcal{L}_\alpha$ (the language of arithmetic augmented with constant symbols $\bar{a}$ for each $a\in|\mathfrak{M}(\alpha)|$),
 we define $\kcal,\alpha\Vdash A$
 (or simply $\alpha\Vdash A$, if no confusion is likely) inductively as follows:
 \begin{itemize}
 \item For atomic $A$, $\alpha\Vdash A$ iff $\mathfrak{M}(\alpha)\models A$. 
 Note that in the structure $\mathfrak{M}(\alpha)$, $\bar{a}$ is 
 interpreted as $a$,
 \item $\kcal,\alpha\Vdash A\vee B$ iff $\kcal,\alpha\Vdash A$ or $\kcal,\alpha\Vdash B$,
 \item $\kcal,\alpha\Vdash A\wedge B$ iff $\kcal,\alpha\Vdash A$ and $\kcal,\alpha\Vdash B$,
 \item $\kcal,\alpha\Vdash A\ra B$ iff for all $\beta\geq\alpha$, $\kcal,\beta\Vdash A$ implies $\kcal,\beta\Vdash B$,
 \item If $A=\forall{x}B$,  $\alpha\Vdash A$ iff for all $\beta\geq\alpha$ and each 
 $b\in|\mathfrak{M}(\beta)|$, we have $\beta\Vdash B[x:\bar{b}]$.
 \end{itemize}
 It is well-known in the literature that $\HA$ is complete for first-order Kripke models. 
 \begin{lemma}\label{Lemma-Sigma-local-global}
 Let $\kcal=(K,<,\mathfrak{M})$ be a Kripke model of $\HA$ and $A$ be an arbitrary $\Sigma_1$-formula.
Then for each $\alpha\in K$, we have $\alpha\Vdash A$ iff $\mathfrak{M}(\alpha)\models A$.
 \end{lemma}
 \begin{proof}
 Use induction on the complexity of $A$ to show that for each $\alpha\in K$, we have 
 $\alpha\Vdash A$ iff $\mathfrak{M}(\alpha)\models A$. In the inductive step for $\to$ and $\forall$,
 use \Cref{Lemma-decidability of delta formulae}.
 \end{proof}
 
\subsection{q-Realizability and Leivant's principle}
A variant of realizability introduced by Kleene, is
$\q$-realizability (see \cite{TD}) which is defined inductively
for arithmetical formula $A$ as follows:
\begin{itemize}
 \item $x\q A:=A$ for atomic $A$.
 \item $x\q (A_1\wedge A_2):= {\sf j}_1(x)\q A_1\wedge {\sf j}_2(x)\q A_2$,
 \item $x\q (A_1\vee A_2):=({\sf j}_1(x)=0\ra {\sf j}_2(x)\q A_1)\wedge({\sf j}_1(x)\neq0\ra {\sf j}_2(x)\q A_2)$,
 \item $x\q (A_1\ra A_2):=\forall{y}\, (y\q A_1\ra \exists{u}\, ({\sf T}xyu\wedge {\sf U}(u)\q A_2))\wedge (A_1\ra A_2) $,
 \item $x\q \exists{y}A(y):=j_1(x)\q A(j_2(x))$,
 \item $x\q\forall{y}A(y):=\forall{y}\,\exists{u}\, ({\sf T}xyu\wedge {\sf U}(u)\q A(y))$
\end{itemize}
In the y
above definition ${\sf j}_1$ and ${\sf j}_2$ are inverses for a one-to-one onto,
pairing function, {\sf j}, such that  $x={\sf j}({\sf  j}_1(x),{\sf j}_2(x) )$. Also
${\sf T}xyu$ is  Kleene's predicate formalizing ``$u$ is a computation
for the Turing Machine with code $x$ with input $y$", and {\sf U} is
the result extractor function, i.e. if $u$ is a computation for a
Turing Machine, then ${\sf U}(u)$ is its output.

\begin{lemma}\label{Lemma-qrealizability}
 For any formula $A$ we have $\HA_0\vdash x\!\q\! A\ra A$.
\end{lemma}
\begin{proof}
See \cite{TD}.
\end{proof}

\noindent
In the following,  $\{x\}$ is  partial recursive function of the
Turing Machine with the code $x$. The notation $\{x\}y\!\!\downarrow$
means that ``the function $\{x\}$ is defined on input $y$", or
equivalently ``the Turing machine with the code $x$ halts on the input
$y$". It is well known that $\{x\}y\!\!\downarrow$ is a $\Sigma_1$
sentence. We use terms which contain some Kleene's bracket
notation. In that case, we use $t\!\!\downarrow$ to mean that all
the brackets in $t$ are defined (terminate). 

One immediate consequence of $\q$-realizability, is Church's
Rule for $\HA$:
\begin{lemma}\label{Lemma-Chirch rule}
For every formula $A(x,y)$, if 
$\HA\vdash\forall{x}\,\exists{y}\,A(x,y)$, then there exists some
$n\in \omega$ such that
$\HA\vdash\forall{x}\, ({\{n\}(x)\!\!\downarrow}\wedge A(x,\{n\}(x)))$.
\end{lemma}
\begin{proof}
See \cite{TD}.
\end{proof}

\noindent
It is easy to observe that ``$\HA\vdash A$" implies ``there
exists some $n$ such that $\HA\vdash n\q A$"(\cite{TD}). The
point of the following lemma is that we can 
refine the above statement in the following way. There exists some recursive function $f$ such that 
``$\HA_m\vdash A(k_1,\ldots,k_l)$" implies 
``there exists some recursive function 
$g$ such that $\HA_{f(m)}\vdash g(k_1,\ldots,k_l)\q A(k_1,\ldots,k_l)$".
Moreover, we can formalize 
this statement in $\HA$.

\begin{lemma}\label{Lemma-upper bound for realizability}
Suppose that $A(x_1,\ldots,x_m)$ is an arithmetical formula with
free variables as shown. Then, there exists a provably \textup{(}in $\HA$\textup{)}
total recursive function $f$ such that:
$$\HA\vdash\Box_x A(\dot{x}_1,\ldots,\dot{x}_m)\ra\exists{z}\;
\Box_{f(x)}({\{\dot{z}\}\langle \dot{x}_1,\ldots,\dot{x}_m
\rangle\!\!\downarrow}\wedge\{\dot{z}\}\langle \dot{x}_1,\ldots,
\dot{x}_m\rangle\q A(\dot{x}_1,\ldots,\dot{x}_m))$$
\end{lemma}
\begin{proof}
The proof is very similar to the proof of the soundness part of \cite[Theorem~4.10]{TD}. 
First define $f(n)$ in this way: 
$$f(n):=\text{max}(\{\gnumber{B^{\q,x}}\ |\ \gnumber{B}< n,\text{ $x$ is a free variable of $B$}\}\cup\{n\})$$
in which, $ B^{\q,x}:=\{ t(u)\}\langle x\rangle\!\!\downarrow \wedge\; \{ t(u)\}\langle x\rangle \q B$, 
 $u\!\neq\! x$ and $t(u)$ is a primitive recursive function that will be defined later in the proof.
 Let's fix some sequence of numbers $\boldsymbol{m}$. 
 With induction on the complexity of the proof 
 $\HA_n\vdash A(\boldsymbol{m})$, we show that  (by $A(\boldsymbol{m})$, we mean $A[\boldsymbol{x}:\boldsymbol{m}]$) by
\begin{center} 
 $\HA\vdash$ ``$\HA_n\vdash A(\boldsymbol{m})$'' $\rightarrow$  $\exists z$
  ``$\HA_{f(n)}\vdash \{z\}\langle\boldsymbol{m}\rangle\!\!\downarrow \wedge \{z\}\langle\boldsymbol{m}\rangle \q A(\boldsymbol{m})$''
 \end{center}
 We only treat the case where $A$ is an instance of induction schema. All the other cases are trivial.
 Assume that $\gnumber{B}<n$ and 
 $$A(\boldsymbol{m})=(B[x:0]\wedge\forall{x}(B\to B[x:S(x)]))\to \forall{x}B$$
 We should find some number $\{z\}\langle\boldsymbol{m}\rangle= k$ such that 
 $$\HA_{f(n)}\vdash k\q [(B(0)\wedge\forall{x}(B(x)\to B(x+1))\to \forall{x}B]$$
  By definition  of $\q$-realizability, we have:
\begin{equation*}
 k\q A(\boldsymbol{m})=\overbrace{\forall{u}[u\q (B(0)\wedge\forall{x}(B(x)\to B(x+1))\to (\{k\}(u)\!\!\downarrow\wedge \{k\}(u)\q \forall{x}B)]}^{C}\wedge A(\boldsymbol{m})
 \end{equation*}
 Since $f(n)\geq n$, we have $\HA_{f(n)}\vdash A(\boldsymbol{m})$. Hence it remains only to show that 
 $\HA_{f(n)}\vdash C$. 
Define the primitive recursive function $t(u)$ in the following way. For any given $u$, $t(u)$ is the code of the Turing Machine that fulfills the following conditions: 
$$
\begin{cases}
\{t(u)\}\langle 0\rangle=j_1(u)\\
\{t(u)\}(x+1)=\{\{j_2(u)\}\langle x\rangle\}\langle\{t(u)\}\langle x\rangle\rangle
\end{cases}
$$
Finally, let $k$ be the code of the Turing Machine that computes  the primitive recursive function $t$. Now it is not difficult to observe 
that, by induction on $B^{\q,x}$, one could deduce $C$ in $\HA_0$, and hence $\HA_{f(n)}\vdash C$. This implies $\HA_{f(n)}\vdash A(\boldsymbol{m})$, as desired. 
\end{proof}

\begin{lemma}\label{Lemma-Reflection refinement}
For every sentence $A$, there exists some provably (in $\HA$)
total recursive function $h_A$ such that
$\HA\vdash\forall{x}\, \Box_{h_{_{\!A}}\!(x)}(\Box_{\dot{x}}A\ra A)$.
\end{lemma}
\begin{proof}
By  \Cref{Lemma-Reflection} we have
$\HA\vdash\forall{x}\,\exists{y}\, \Box_y(\Box_{\dot{x}}A \ra A)$. 
Now we have the desired result by use of 
\Cref{Lemma-Chirch rule}.
\end{proof}

\begin{lemma}\label{Lemma-Sigma sentences are autoq}
Suppose that $A(x_1,\ldots,x_m)$ is a $\Sigma_1$-formula with
variables as shown. Then there exists some $n_\tinysub{A}\in\mathbb{N}$,
such that
$$\HA\vdash A(x_1,\ldots,x_m)\ra ({\{n_\tinysub{A}\}\langle x_1,\ldots,x_m
\rangle\!\!\downarrow}\wedge\{n_\tinysub{A}\}\langle x_1,\ldots,x_m\rangle\q A(x_1,\ldots,x_m))$$
\end{lemma}
\begin{proof}
This theorem for $\textbf{r}$-realizability instead of
$\q$-realizability is proved in \cite{TD}(Proposition 4.4.5). The
proof for $\textbf{q}$-realizability is quite similar and we
leave it to the reader.
\end{proof}

It  is well-known that the disjunction property holds for \IPC and
$\HA$, however it is also shown that in case of \HA, the proof is
 not formalizable in $\HA$,
i.e. $\HA\nvdash\Box(A\vee B)\ra(\Box A\vee\Box B)$. But
this is not the end of story! Daniel Leivant in his PhD dissertation
\cite{Leivant-Thesis} showed that $\HA\vdash\Box(A\vee B)\ra\Box(
A\vee\Box B)$. Albert Visser in an unpublished paper showed
that we can extend Leivant's principle to the following version.
For every $\Sigma_1$-sentence $A$,  $\HA\vdash\Box (A\ra
(B\vee C))\ra\Box(A\ra(\Box B\vee C))$. In the following
lemma,  we will show that we can find (constructively) from the
code $x$ of the  proof of $A\ra(B\vee C)$, some $f(x)$ such that
$\Box(A\ra(\Box_{f(x)}B\vee C))$ holds. Although the statement of this theorem
would not be used later in this paper, we bring it here for better understanding of its generalization
in a more technical lemma, i.e. \Cref{Lemma-qtranslation2}.

\begin{theorem}\label{Theorem-Leivant}
For arbitrary sentences $A,B,C$ such that $A\in\Sigma_1$, there
exists a provably (in $\HA$) total recursive function  $f$ such
that
$$\HA\vdash \Box_x (A\ra(B\vee C))\ra \Box_{f(x)}(A\ra(\Box_{f(x)}B\vee C))$$
\end{theorem}
\begin{proof}
First observe that, by  \Cref{Lemma-Sigma sentences are
autoq}, there exists some finite number $n_\tinysub{A}\in \mathbb{N}$ such
that $\HA\vdash A\ra
({\{n_\tinysub{A}\}\langle\rangle\!\!\downarrow}\wedge\{n_\tinysub{A}\}\langle\rangle\q
A)$. We set $t_0:=\{n_\tinysub{A}\}\langle\rangle$. Hence there exists some
$n_0\in\mathbb{N}$ such that
\begin{equation}\label{Eq Lei1}
\HA\vdash\Box_{n_0}( A\ra ({t_0\!\!\downarrow}\wedge t_0\q A))
\end{equation}

We work inside $\HA$. Assume $\Box_x (A\ra(B\vee C) )$. By
 \Cref{Lemma-upper bound for realizability}, there exists
some $z$ such that
$\Box_{g_{_0}(x)}({\{\dot{z}\}\langle\rangle\!\!\downarrow}\wedge
\{\dot{z}\}\langle\rangle\q (A\ra (B\vee C)))$, in which $g_{_0}$ is
the recursive function  provided by  \Cref{Lemma-upper bound
for realizability}. We define $t_1:=\{\dot{z}\}\langle\rangle$
and hence we have $\Box_{g_{_0}(x)}t_1\!\!\downarrow$. If we set
$g_{_1}(y):=g_{_0}(y)+n_{_0}$, by use of \cref{Eq Lei1}, we can
deduce $\Box_{g_{_1}(x)}(A\ra ({t_0\!\!\downarrow} \wedge
\{t_1\}(t_0)\q (B\vee C)))$. We set $t_2:=\{t_1\}(t_0)$. Then, by
definition of $\q$-realizability, we have:
 \[ \Box_{g_{_1}(x)}(A\ra
({t_2\!\!\downarrow}\wedge (j_1(t_2)=0\ra j_2(t_2)\q B)\wedge
(j_1(t_2)\neq 0\ra j_2(t_2)\q C))).\]
 Let $B':=(j_1(t_2)=0)\ra
j_2(t_2)\q B$ and $C':=(j_1(t_2)\neq 0)\ra j_2(t_2)\q C$. Then we
have ${\Box_{g_{_1}(x)}(A\ra B')}$
 and, hence, by
$\Sigma_1$-completeness (\Cref{Lemma-bounded Sigma
completeness}), we can deduce $\Box_{0}\Box_{g_{_1}(x)}(A\ra
B')$, that again by use of  \Cref{Lemma-bounded Sigma
completeness}, implies $\Box_0(A\ra\Box_{g_{_1}(x)}B')$. Thus
we have 
$$\Box_{g_{_1}(x)}(A\ra ({t_2\!\!\downarrow}\wedge\Box_{g_{_1}(x)}B'\wedge C'))$$
 Again by
 \Cref{Lemma-bounded Sigma completeness} and
\Cref{Lemma-qrealizability}, $\Box_{g_1(x)}(A\ra
({t_2\!\!\downarrow}\wedge(j_1(t_2)=0\ra
\Box_{g_{_1}(x)}B)\wedge(j_1(t_2)\neq0\ra C)))$. Since atomic
formulae are decidable in $\HA$, so for any atomic formulae $D$,
there exists some finite $n_2$ such that in $\HA_{n_2}$ we have
decidability of $D$. Let $\HA_{n_2}+t_2\!\!\downarrow$ decide
$j_1(t_2)=0$. If we set $f(x):=g_{_1}(x)+n_2$, we can deduce
$\Box_{f(x)}(A\ra(\Box_{f(x)}B\vee C))$, as desired.
\end{proof}


\subsection{The extended Leivant's Principle}\label{Sec-ExLePr}
In this section, we study properties of the extended Leivant's
principle, ${\sf Le}^+$. We prove that for any $\Sigma_1$-
substitution $\sigma$, $\HA\vdash\sigma_{_{\sf HA}}({\sf Le}^+)$.

Define a translation $q_\sigma(A,x)$ recursively for a modal
proposition $A$ and a $\Sigma_1$-substitution $\sigma$, as
follows:
\begin{itemize}
\item $q_\sigma(A,x):=\sigma_{_{\sf HA}}(A)$, if $A$ is atomic or boxed,
\item $q_\sigma(A\wedge B,x):=q_\sigma(A,j_1(x))\wedge q_\sigma(B,j_2(x))$,
\item $q_\sigma(A\vee B,x):=(j_1(x)=0\ra
q_\sigma(A,j_2(x)))\wedge(j_1(x)\neq 0\ra q_\sigma(B,j_2(x)))$,
\item if $A=B\ra C$ and $B\in {\sf NOI}$, we define $q_\sigma(B\ra
C,x):=\sigma_{_{\sf HA}}(B)\ra( \{x\}(n_\tinysub{B})\!\!\downarrow \wedge
q_\sigma(C,\{x\}(n_\tinysub{B})))$, in which $n_\tinysub{B}$ is as in 
\Cref{Lemma-Sigma sentences are autoq}. If $B\not\in {\sf NOI}$,
then define $q_\sigma(A,x):=\sigma_{_{\sf HA}}(A)$.
\end{itemize}

\begin{lemma}\label{Lemma-qtranslation1}
Let $A$ be a  modal proposition  and
 $t$ be a term in first-order language of arithmetic which possibly contain  Kleene's brackets. Then
\begin{itemize}
\item $\HA_0\vdash x\q\sigma_{_{\sf HA}}(A)\ra q_\sigma(A,x)$,
\item $\HA_0\vdash (t\!\!\downarrow\wedge q_\sigma(A,t))\ra \sigma_{_{\sf HA}}(A)$,
\end{itemize}
\end{lemma}
\begin{proof}
Proof of both parts are by  induction on the 
complexity of $A$.
\end{proof}

For the next lemma, we need some auxiliary notation $\sigma_{_l}(A,x)$.  Informally speaking, 
$\sigma_{_l}(A,x)$ is going to be $\sigma_{_{\sf HA}}(A^l)$ with one difference. The new added boxes in $A^l$
 should be interpreted as provability in $\HA_x$. More precisely, we define it inductively as the following.
\begin{itemize}
\item $A$ is atomic or boxed. $\sigma_{_l}(A,x):=\sigma_{_{\sf HA}}(A)$,
\item $A= B\wedge C$. then $\sigma_{_l}(A,x):=\sigma_{_l}(B,x)\wedge\sigma_{_l}(C,x)$,
\item $A= B \vee C$. then $\sigma_{_l}(A,x):=\Boxdot_x\sigma_{_l}(B,x)\vee\Boxdot_x\sigma_{_l}(C,x)$,
 in which $\Boxdot_x D$ is defined as $D\wedge\Box_x D$,
 \item $A=B\to C$. Like the definition of $A^l$, we define $\sigma_{_l}(A,x)$ by cases. If $B\in {\sf NOI}$, then we define
 $\sigma_{_l}(A,x):=\sigma_{_{\sf HA}}(B)\to\sigma_{_l}(C,x)$, otherwise we define $\sigma_{_l}(A,x):=\sigma_{_{\sf HA}}(A)$.
\end{itemize}

\begin{lemma}\label{Lemma-prop-sigma_l}
Let $A$ be a modal proposition. Then 
\begin{enumerate}
\item \label{1Lemma-prop-sigma_l}$\HA_0\vdash (x\leq y \wedge \sigma_{_l}(A,x))\to\sigma_{_l}(A,y)$,
\item \label{2Lemma-prop-sigma_l} $\HA_0\vdash \sigma_{_l}(A,x) \to \sigma_{_{\sf HA}}(A^l)$,
\item \label{3Lemma-prop-sigma_l}$\HA_0\vdash \sigma_{_l}(A,x) \to \sigma_{_{\sf HA}}(A)$.
\end{enumerate}
\end{lemma}
\begin{proof}
Use induction on $A$.
\end{proof}

\begin{lemma}\label{Lemma-qtranslation2}
Let $A$ be a  modal proposition, $D$ be any $\Sigma_1$-sentence  and
 $t$ be a term in first-order language of arithmetic which possibly contain  Kleene's brackets. Then there exists 
 a provably total recursive function $f$ such that
 $$\HA\vdash \Box_x(D\ra  (t\!\!\downarrow\wedge q_\sigma(A,t))\ra
\Box_{f(x)}(D\ra\sigma_{_{l}}(A,f(x)))$$
\end{lemma}
\begin{proof}
We use induction on $A$.
For simplicity of notations, we assume here that $t$ is a normal
term. One can easily build the general case.

 \vspace{.05in}

\noindent{\em Atomic, Boxed or conjunction.}  Trivial.

 \vspace{.05in}

\noindent{\em  Disjunction.} Let $A=B\vee C$. Then by definition
of $q_\sigma$, we have
$$\HA\vdash\Box_x(D\ra q_\sigma(B\vee
C,t))\ra [\Box_x((D\wedge j_1(t)=0)\ra
q_\sigma(B,j_2(t)))\wedge\Box_x((D\wedge j_1(t)\neq 0)\ra
q_\sigma(C,j_2(t))]$$
Hence by  the induction hypothesis, there exists functions $g$ and $h$ such that 
\begin{align*}
\HA\vdash & \Box_x(D\ra q_\sigma(B\vee C,t))\ra 
\\
&\Box_{g(x)}((D\wedge
j_1(t)=0)\ra \sigma_{_l}(B,g(x)))\wedge\Box_{h(x)}((D\wedge j_1(t)\neq 0)\ra \sigma_{_l}(C,h(x)))
\end{align*}
Let $f(x)$ be the maximum of $g(x)$ and $h(x)$. 
One can use the  $\Sigma_1$-completeness of $\HA_0$ (\Cref{Lemma-bounded Sigma completeness}) and
\Cref{Lemma-prop-sigma_l}
 to derive
$$\HA\vdash \Box_x(D\ra q_\sigma(B\vee C,t))\ra
\Box_{f(x)}(D\ra(\Boxdot_{f(x)} \sigma_{_l}(B,f(x))\vee\Boxdot_{f(x)} \sigma_{_l}(C,f(x))))$$

 \vspace{.05in}

\noindent{\em  Implication.} Assume that $A=B\ra C$. If $B\not\in
{\sf NOI}$, by  \Cref{Lemma-qtranslation1}, we are done. So assume that $B\in {\sf NOI}$. By
definition of $q_\sigma$, there exists some term $t_1$ such that 
$$\HA\vdash \Box_x[D\ra q_\sigma(B\ra
C,t)]\ra\Box_x[(D\wedge \sigma_{_{\sf HA}}(B))\ra(t_1\!\!\downarrow\wedge
q_\sigma(C,t_1))]$$
Since $B\in{\sf NOI}$, $\sigma_{_{\sf HA}}(B)$ is a $\Sigma_1$-formula.  Hence by the induction hypothesis,  
there exists some function $f$ such that 
$$\HA\vdash \Box_x(D\ra q_\sigma(A,t))\ra\Box_{f(x)}((D\wedge
\sigma_{_{\sf HA}}(B))\ra \sigma_{_l}(C,f(x)))$$
This by definition of $\sigma_{_l}(B\to C,f(x))$, implies the desired result.
\end{proof}

\begin{lemma}\label{Lemma-sigma_l translation}
For any $\Sigma_1$-substitution $\sigma$ and modal proposition $A$, there exists some provably total recursive function $g$
such that 
$ \HA\vdash \Box_x \sigma_{_{\sf HA}}(A)\to\Box_{g(x)}\sigma_{_l}(A,g(x))$.
\end{lemma}
\begin{proof}
Work inside $\HA$. Assume $\Box_x\sigma_{_{\sf
HA}}(A)$. By  \Cref{Lemma-upper bound for realizability},
there exists some $y$ such that 
$${\Box_{f_0(x)} (t\!\!\downarrow\wedge\  t\q\sigma_{_{\sf HA}}( A))}$$
 in which $t:=\{y\}\langle\rangle$ and $f_0$ is a provably total recursive function 
 as stated in \Cref{Lemma-upper bound for realizability}.
Hence by the first item of \Cref{Lemma-qtranslation1}, 
$\Box_{f_0(x)}(t\!\!\downarrow\wedge q_\sigma(A,t))$. Hence by 
\Cref{Lemma-qtranslation2}, we have the function $f$ such that
 $\Box_{f(f_0(x))}\sigma_{_l}(A,f(f_0(x))$.
\end{proof}

\begin{theorem}\label{Theorem-Soundness of HA for lle+}
For any $\Sigma_1$-substitution $\sigma$, we have $\HA\vdash
\sigma_{_{\sf HA}}({\sf Le}^+)$.
\end{theorem}
\begin{proof}
Let $A$ be a modal proposition. We must show
$\HA\vdash\Box\sigma_{_{\sf HA}}(A)\ra\Box\sigma_{_{\sf
HA}}(A^l)$. Now the desired result may be deduced by 
\Cref{Lemma-sigma_l translation} and  the second item of \Cref{Lemma-prop-sigma_l}.
\end{proof}
Although there are  other ways of  proving the above theorem  (see
\cite{Visser02} or \cite{IemhoffT}), we need its major preliminary lemma (i.e. \Cref{Lemma-sigma_l translation})
in the proof of the completeness theorem. Specially, we use  \Cref{Lemma-sigma_l translation}
in the proof of  \Cref{Lemma-1.7st Properties of Solovay Function}.

\subsection{Interpretability}
Let $T$ and $S$ be two first-order theories. Informally speaking,
we say that $T$ interprets $S$ ($T\rhd S$) if there exists a
translation from the language of $S$ to the language of $T$ such
that $T$ proves the translation of all of the theorems of $S$.
For a formal definition see \cite{VisserInterpretability}. It is
well-known that for recursive theories $T$ and $S$ containing $\PA$,
the assertion $T\rhd S$ is formalizable  in first-order language
of arithmetic. For two arithmetical sentences $A$ and $B$, we use the
notation $A\rhd B$ to mean that $\PA+A$ interprets $\PA+B$. The
following theorem due to Orey, first appeared in \cite{Feferman}.

\begin{theorem}\label{Theorem-Orey}
For recursive theories $T$ and $S$ containing $\PA$, we have: 
\[ \PA\vdash (T\rhd S) \lr \forall{x}\, \Box_T {\sf Con}(S^x),\]
 in which $S^x$ is the
restriction of the theory $S$  to axioms with G\"{o}del number
$\leq x$ and ${\sf Con}(U):=\neg\,\Box_U\bot$.
\end{theorem}
\begin{proof}
See \cite{Feferman}. p.80 or \cite{Berarducci}.
\end{proof}

\noindent\textbf{Convention.}
From \Cref{Theorem-Orey}, one can easily observe that $\PA\vdash {(A\rhd B)}\lr{\forall{x}\,\Box^+(A\ra\neg\Box^+_x\neg B)}$.
So from now on, 
$A\rhd B$  means its $\Pi_2$-equivalent
$\forall{x}\,\Box^+(A\ra\neg\Box^+_x\neg B)$, even when we
are working in weaker theories like $\HA$. We remind the reader
that $\Box^+$ stands for provability in $\PA$.

\section{Propositional modal logics}\label{sec-propositional}
In this section, we collect all the required notions with propositional flavour. 
This section is mostly
devoted to provide an axiomatic system for the $\Sigma_1$-provability logic of $\HA$, i.e. $\lles$,
 and stating some of its essential properties that we need them later in the proof of soundness  
 (\Cref{Theorem-Soundness})
 or completeness (\Cref{Theorem HA-Completeness}) of $\lles$ for arithmetical $\Sigma_1$-interpretations.
 The following are 
some of  important results that will be used in the proof of completeness theorem.
\begin{itemize}[leftmargin=*]
\item   In \Cref{Sec-Aximatizing-TNNIL},  it is shown  
that the axiomatic system $\lles$ is capable of simplifying any modal proposition   to 
an equivalent $\TNNIL^-$ proposition (\Cref{corollar HA-NNIL approximation is propositionally equivalent}).  
This fact is useful for  proof of the completeness theorem (\Cref{Theorem HA-Completeness}).
\item In \Cref{sec.tnnil.conservativity}, the $\TNNIL$-conservativity of the theory 
$\LC$ 
over $\lles$ (\Cref{Theorem-TNNIL Conservativity of LC over LLe+}) is proved. 
This conservativity plays an important role in the proof of completeness theorem.  
As far as working with $\TNNIL$-formulas, we get rid of all those complicated axioms of 
$\lles$ and just use the more handful theory $\LC$.
\item In \Cref{Sec-PropModKripke}, we will prove the 
finite model property for the theory $\LC$ (\Cref{Theorem-Propositional Completeness LC}). 
With the aid of our main theorem in next section (\Cref{Theorem-Main tool}), 
such finite counter-models 
are used to be transformed to a first-order counter-models of $\HA$.
\end{itemize}
\subsection{The NNIL formulae and related topics}\label{sec-nnil}
The class of {\em No Nested Implications to the Left}, \NNIL
formulae in a propositional language was introduced in
\cite{Visser-Benthem-NNIL}, and more explored in \cite{Visser02}. 
The crucial
result of \cite{Visser02} is providing an algorithm that as
input, receives a non-modal proposition $A$ and returns its best \NNIL
approximation $A^*$ from below, i.e., $\IPC\vdash A^*\ra A$ and
for all \NNIL formula $B$ such that $\IPC\vdash B\ra A$, we have
$\IPC\vdash B\ra A^*$. Also for all $\Sigma_1$-substitutions $\sigma$, 
we have $\HA\vdash \sigma_{_{\sf HA}}(\Box A\lr \Box A^*)$ \cite{Visser02}.
\begin{itemize}[leftmargin=*]
\item In \Cref{subsubsec-NNIL-algorithm}, we state Visser's
$\NNIL$-algorithm for computing $A^*$,  and some of its useful properties.  
\item In \Cref{subsubsec-TNNIL-algorithm}, we explain  the extension of this algorithm 
to the modal language (the $\TNNIL$-algorithm), which computes $A^+$ 
 and  is essentially  the same as the $\NNIL$-algorithm with this extra rule: 
treat inside $\Box$ as a fresh proposition, i.o.w. in the inductive definition of the algorithm $(\Box A)^+:=\Box A^+$.   
Then we prove some useful properties of the $\TNNIL$-algorithm: \Cref{Lemma-TNNIL algorithm preserve theorems of iGL} and \Cref{corollar-NNIL properties}.  The best feature of 
$\TNNIL$-algorithm is that for all $\Sigma_1$-substitutions $\sigma$, 
we have $\HA\vdash \sigma_{_{\sf HA}}(\Box A\lr \Box A^+)$ (first part of \Cref{corollar-NNIL properties}).
\item In \Cref{subsubsec-TNNIL^-algorith},
we define another algorithm $\TNNIL^-$ for computing $A^-$, which is essentially the same as the 
$\TNNIL$-algorithm, with this minor difference: Only treat those sub-formulae which are boxed and leave the others. 
With this minor change, we even have a better feature for $A^-$, i.e., for all $\Sigma_1$-substitutions $\sigma$, 
we have $\HA\vdash \sigma_{_{\sf HA}}( A\lr  A^-)$ (\Cref{Lemma HA-NNIL properties minus}).
\end{itemize}
Now we define the class $\NNIL$ of modal propositions precisely by  $\NNIL:= \{A\mid \rho A\leq 1\}$, in which 
the complexity measure $\rho$, is defined inductively as follows:
\begin{itemize}[leftmargin=*]
\item $\rho(\Box A)=\rho(p)= \rho(\bot)=\rho(\top) = 0$, for an arbitrary atomic variables $p$ and modal proposition $A$,
\item $\rho(A\wedge B) = \rho(A\vee B) = \text{max} (\rho A, \rho B)$,
\item $\rho(A\ra B) = \text{max} (\rho A +1, \rho B)$,
\end{itemize}
In the following, we define another complexity measure $\mathfrak{o}(.)$ on modal propositions. 
We need this measure for  termination of the $\NNIL$-algorithm.
\begin{definition}\label{Definition-non-modal complexity}
Let $D$ be a modal proposition. Let  
\begin{itemize}
\item $I(D):=\{E\in {\sf Sub}(D) \mid E \text{ is an implication that is not in the scope of a } \Box\}$.
\item $\mathfrak{i}(D):=\text{\em max}\{|I(E)|\mid E\in I(D)\}$, where $|X|$
is the number of elements of $X$.
\item $\mathfrak{c}D:=$ the number of occurrences of logical connectives which are not in the scope of a $\Box$.
\item $\mathfrak{d}D:=$ the maximum number of nested boxes. To be more precise,
\begin{itemize}
\item $\mathfrak{d}D:=0$ for atomic $D$,
\item $\mathfrak{d}D:=\text{\em max}\{\mathfrak{d}D_1,\mathfrak{d}D_2\}$, where $D = D_1\circ D_2$
and $\circ\in\{\wedge, \vee, \ra\}$,
\item $\mathfrak{d}\Box D:=\mathfrak{d}D+1$,
\end{itemize}
\item $\mathfrak{o}D:=(\mathfrak{d}D,\mathfrak{i}D,\mathfrak{c}D)$.
\end{itemize}
We order the measures  $\mathfrak{o}D$ lexicographically, i.e., 
$(d,i,c)<(d',i',c')$ iff $d<d'$ or $d=d'
, i<i'$ or $d=d', i=i' , c<c'$.
\end{definition}
For definition of $\NNIL$-algorithm, we use the bracket notation   $[A]B$ from \cite{Visser02}.  
We also use a variant of this notation, $[A]'B$:
\begin{definition}\label{definition-braket}
For any two modal propositions $A$ and $B$, we define $[A]B$ and  $[A]'B$ 
 by induction on the complexity of $B$:
\begin{itemize}
\item $[A]B = [A]'B =B$, for atomic or boxed $B$,
\item $[A](B_1\circ B_2) = [A](B_1)\circ [A](B_2)$, $[A]'(B_1\circ B_2) =[A]'(B_1)\circ [A]'(B_2)$
for $\circ\in\{\vee,\wedge\}$,
\item $[A](B_1\ra B_2) = A\ra (B_1\ra B_2)$, $[A]'(B_1\ra B_2)=A'\to( B_1\ra B_2)$,
in which $A' = {A[B_1\ra B_2\mid B_2]}$, i.e., replace each outer
occurrence of $B_1 \ra B_2$ (by outer occurrence we mean that it is not in the scope of any $\Box$)  in $A$ by $B_2$,
\end{itemize}
For a set $X$ of modal propositions, we also define $[A]X:=\bigvee_{B\in X}A[B]$ and $[A]'X:=\bigvee_{B\in X}[A]'B$.
\end{definition}
\begin{remark}\label{Remark0}
It is easy to observe that $[A]B$ and $[A]'B$ are equivalent in $\IPC_\Box$.
\end{remark}

\subsubsection{The $\NNIL$-algorithm}\label{subsubsec-NNIL-algorithm}
For each modal proposition $A$, the proposition 
$A^*$ is defined by induction on $\mathfrak{o}A$ as follows \cite{Visser02}:
\begin{enumerate}[leftmargin=*]
\item $A$ is atomic or boxed, take $A^*:=A$.
\item $ A=B \wedge C$, take $A^*:=B^*\wedge C^*$.
\item $ A=B\vee C$, take $A^*:=B^*\vee C^*$.
\item $ A=B \ra C $, we have several sub-cases. In the following, 
an occurrence of $E$ in $D$ is called an {\em outer
occurrence}, if $E$ is neither in the scope of an implication nor
in the scope of a boxed formula.
\begin{enumerate}[leftmargin=*]
\item $C$ contains an outer occurrence of a conjunction. In this
case, there is some formula $J(q)$ such that
\begin{itemize}
\item $q$ is a propositional variable not occurring in $A$.
\item $q$ is outer in $J$ and occurs exactly once.
\item $C=J[q|(D\wedge E)]$.
\end{itemize}
Now set $C_1:=J[q|D], C_2:=J[q|E]$ and
$A_1:=B\ra C_1, A_2:=B\ra C_2$ and finally, define
$A^*:=A_1^*\wedge A_2^*$.
\item $B$ contains an outer occurrence of a disjunction. In this
case, there is some formula $J(q)$ such that
\begin{itemize}
\item $q$ is a propositional variable not occurring in $A$.
\item $q$ is outer in $J$ and occurs exactly once.
\item $B=J[q|(D\vee E)]$.
\end{itemize}
 Now set $B_1:=J[q|D], B_2:=J[q|E]$
and ${A_1:=B_1\ra C}, {A_2:=B_2\ra C}$ and finally, define
$A^*:=A_1^*\wedge A_2^*$.
\item  $B=\bigwedge X$ and $C=\bigvee Y$ and $X,Y$ are sets of
implications or atoms. We have several sub-cases:
\begin{enumerate}[leftmargin=*]
\item $X$ contains atomic variables or boxed formula $E$. We
set $D:=\bigwedge(X\setminus\{E\})$ and take
 ${A^*:=E^*\ra(D\ra C)^*}$.
\item $X$ contains $\top$. Define
$D:=\bigwedge(X\setminus\{\top\})$ and take $A^*:=(D\ra C)^*$.
\item $X$ contains $\bot$. Take $A^*:=\top$.
\item $X$ contains only implications. For any $D=E\ra F\in X$,
define
$$B\!\downarrow\! D:=\bigwedge((X\setminus\{D\})\cup\{F\}).$$
Let $Z:=\{E\mid E\ra F\in X\}\cup\{C\}$ and define: 
\begin{align*}
A^*:=\bigwedge\{((B\!\downarrow\! D)\ra C)^*|D\in X\}\wedge \bigvee \{([B]'E)^*\mid E\in Z\}
\end{align*}
We should show $\mathfrak{o}([B]'E)<\mathfrak{o}A$. 
For a proof of this fact see \cite{Visser02}. 
\end{enumerate}
\end{enumerate}
\end{enumerate}

\begin{remark}\label{rem1}
{\em In fact in \cite{Visser02}, the $\NNIL$-algorithm  is only for non-modal propositions. 
One may also compute the best $\NNIL$-approximation for modal propositions,
in the following way. 
Let $A$ be a given modal proposition. 
Let $B_1,\ldots,B_n$ be all
boxed sub-formulae of $A$ which are not in the scope of any
other boxes. Let $A'(p_1,\ldots,p_n)$ be the unique non-modal
proposition such that $\{p_i\}_{1\leq i\leq n}$ are fresh atomic
variables not occurring in $A$ and $A=A'[p_1|B_1,\ldots,p_n|B_n]$.
Let $\gamma(A):=(A')^*[p_1|B_1,\ldots,p_n|B_n]$. Then it is easy to
observe that $\IPC_{\Box}\vdash\gamma(A)\lr A^*$.}
\end{remark}

The above defined algorithm is not deterministic, however from the
following theorem we know that $A^*$ is unique up to
$\IPC_{\Box}$ equivalence. Notation
$A\vartriangleright_{_{{\sf IPC}_{\Box},{\sf NNIL}}}B$ 
($A$, $\NNIL $-preserves $B$) from \cite{Visser02}, means that
for each $\NNIL $ modal proposition $C$, if
$\IPC_{\Box}\vdash C\ra A$, then $\IPC_{\Box}\vdash C\ra B$, in
which $A, B$ are modal propositions.

\begin{theorem}\label{Theorem-NNIL Crucial Properties}
For each modal proposition $A$,
\begin{enumerate}
\item  \label{1Theorem-NNIL Crucial Properties}The $\NNIL $ algorithm with  input $A$ terminates and
the output formula $A^*$, is an $\NNIL $ proposition such
that $\IPC_{\Box}\vdash A^*\ra A$.
\item \label{2Theorem-NNIL Crucial Properties} $\IPC_{\Box}\vdash A^*\ra B$ iff
$A\vartriangleright_{_{{\sf IPC}_{\Box},{\sf NNIL}}}B$.
\item \label{3Theorem-NNIL Crucial Properties} $A^*$ is the best $\NNIL $ approximation of $A$ from below i.e.
$\IPC_{\Box}\vdash A^*\ra A$ and for each $\NNIL $
proposition $B$, with $\IPC_{\Box}\vdash B\ra A$, we have
$\IPC_{\Box}\vdash B\ra A^*$.
\item \label{4Theorem-NNIL Crucial Properties}$\IPC_{\Box}\vdash A_1\ra A_2$ implies $\IPC_{\Box}\vdash A_1^*\ra A_2^*$.
\item \label{5Theorem-NNIL Crucial Properties}$\IPC_{\Box}\vdash A\lr B$ implies $\IPC_{\Box}\vdash A^*\lr B^*$.
\item \label{6Theorem-NNIL Crucial Properties}For each $\Sigma_1$-substitution $\sigma$,
$\HA\vdash\Box\sigma_{_{\sf HA}}(A)\lr\Box\sigma_{_{\sf HA}}(A^*)$.
\end{enumerate}
\end{theorem}
\begin{proof}
\begin{enumerate}
\item Direct consequence of \cite[Theorem~7.1]{Visser02}.
First assume $A'[p_1,\ldots,p_n]$ be as in   \Cref{rem1}. By
\cite[Theorem~7.1]{Visser02}, we have $\IPC\vdash (A')^*\ra
A'$, and hence by  \Cref{Lemma-boxed as atomic},
$\IPC_{\Box}\vdash (A')^*[p_1|B_1,\ldots,p_n|B_n]\ra A$.
\item Direct consequence of \cite[Theorem~7.2]{Visser02}.
First suppose that $\IPC_{\Box}\vdash A^*\ra B$. Let $A', B'$ be
non-modal propositions as defined in   \Cref{rem1}, i.e,
$A=A'[p_1|C_1,\ldots,p_n|C_n], B=B'[p_1|C_1,\ldots,p_n|C_n]$.
Then by  \Cref{Lemma-boxed as atomic}, $\IPC\vdash(A')^*\ra
B'$. Now by \cite[Theorem~7.2]{Visser02}, we have
$A'\vartriangleright_{\IPC,\NNIL}B'$, and then by 
\Cref{Lemma-boxed as atomic},
$A\vartriangleright_{\IPC_{\Box},\NNIL }B$. For the proof of the other way around, 
note that all of the previous deductions are reversible.
\item Suppose  $\IPC_{\Box}\vdash B\ra A$ and $B$ is  $\NNIL $.
Since $\IPC_{\Box}\vdash A^*\ra A^*$, from \cref{2Theorem-NNIL Crucial Properties} above, we get
$A\vartriangleright_{\IPC_{\Box},\NNIL }A^*$. By
$\IPC_{\Box}\vdash B\ra A$ and $B\in\NNIL $, we have
$\IPC_{\Box}\vdash B\ra A^*$.
\item Suppose that $\IPC_{\Box}\vdash A_1\ra A_2$.
By part 1, $\IPC_{\Box}\vdash A_1^*\ra A_2$ and hence by \cref{3Theorem-NNIL Crucial Properties},
${\IPC_{\Box}\vdash A_1^*\ra A_2^*}$.
\item Direct consequence of \cref{4Theorem-NNIL Crucial Properties}.
\item First suppose that $A$ is a non-modal proposition.
Combining Theorem 10.2 and Corollary 7.2 from \cite{Visser02},
implies that $\IPC\vdash A^*\ra B$ iff $A\mid\!\sim^{\sf
HA}_{{\sf HA},\Sigma}B$, in which $A\mid\!\sim^{{\sf HA}}_{{\sf
HA},\Sigma}B$ means that for each $\Sigma_1$-substitution
$\sigma$, we have $\HA\vdash\Box\sigma_{_{\sf HA}}(A)\ra\Box\sigma_{_{\sf HA}}(B)$.
This implies that
${\HA\vdash\Box\sigma_{_{\sf HA}}(A)\lr\Box\sigma_{_{\sf HA}}(A^*)}$. Now for a
modal proposition $A$, suppose that $A'(p_1,\ldots,p_n)$ and
$B_1,\ldots,B_n$ be such that $A=A'[p_1|B_1,\ldots,p_n|B_n]$, in
which $A'$ is a non-modal proposition and $p_1,\ldots,p_n $ are
fresh atomic variables (not occurred in $A$). Let $\sigma'$ be
the substitution defined by $\sigma'(p_i):=\sigma_{_{\sf HA}}(B_i)$, for each
$1\leq i\leq n$, and for any other atomic variable $q$,
$\sigma'(q) = \sigma(q)$. Clearly, $\sigma'$ is again a
$\Sigma_1$-substitution and hence we have
$\HA\vdash\Box\sigma'_{_{\sf HA}}(A')\lr\Box\sigma'_{_{\sf HA}}((A')^*)$. This
implies  ${\HA\vdash\Box\sigma_{_{\sf HA}}(A)\lr\Box\sigma_{_{\sf HA}}(A^*)}$.
\end{enumerate}
\end{proof}
\subsubsection{The $\TNNIL$-algorithm}\label{subsubsec-TNNIL-algorithm}

\begin{definition}\label{Def-TNNIL-Propositions}
$\TNNIL$ (Thoroughly $\NNIL$) is the
smallest class of propositions such that
\begin{itemize}
\item $\TNNIL$ contains all atomic propositions,
\item if $A, B\in\TNNIL$, then $A\vee B, A\wedge B,\Box A\in\TNNIL$,
\item if all $\ra$ occurring in $A$ are contained in the scope of a $\Box$ (or equivalently
$A\in {\sf NOI}$) and $A,
B\in\TNNIL$, then $A\ra B\in\TNNIL$.
\end{itemize}
Let  $\TNNIL^-$  indicates the set of all the propositions
like $  A(\Box B_1,\ldots,\Box B_n)$, such that
$A(p_1,\ldots,p_n)$ is an arbitrary non-modal proposition and
$B_1,\ldots,B_n\in\TNNIL$.
\end{definition}

Here we define $A^+$  to be the   $\TNNIL$-formula approximating $A$. The major difference between 
$A^+$ and $A^*$ is that $\IPC_\Box\vdash A^+\to A$ may not hold any more.
Informally speaking, to find $A^+$, we first compute $A^*$ and
then replace all outer boxed formula $\Box B$ in $A$ by
$\Box B^+$. To be more accurate, we  define $A^+$ by induction
on $\mathfrak{d}A$. Suppose that for all $B$ with
$\mathfrak{d}B<\mathfrak{d}A$, we have defined $B^+$. Now suppose
that $A'(p_1,\ldots,p_n)$ and $\Box B_1,\ldots,\Box B_n$ are
such that $A=A'[p_1|\Box B_1,\ldots,p_n|\Box B_n]$, where
$A'$ is a non-modal proposition and $p_1,\ldots,p_n $ are fresh
atomic variables (not occurred in $A$). It is clear that
$\mathfrak{d}B_i<\mathfrak{d}A$ and then we can define
$A^+:=(A')^*[p_1|\Box B_1^+,\ldots,p_n|\Box B_n^+]$.

\begin{lemma}\label{Lemma-TNNIL algorithm preserve theorems of iGL}
For every modal proposition $A$,
\begin{enumerate}
\item \label{item1-Lemma-TNNIL algorithm preserve theorems of iGL} 
If ${\sf iGL}\vdash A$ then ${\sf iGL}\vdash A^+$.
\item \label{item2-Lemma-TNNIL algorithm preserve theorems of iGL}
If ${\sf iK4}\vdash A$ then ${\sf iK4}\vdash A^+$.
\end{enumerate}
\end{lemma}
\begin{proof}
We prove the first part by induction on the complexity of proof
${\sf iGL}\vdash A$. Proof of the second part is similar to the first
one.
\begin{itemize}[leftmargin=*]
\item $A$ is an axiom.
\begin{itemize}[leftmargin=*]
\item $A$ is L\"{o}b's axiom, i.e., $A=\Box(\Box B\ra B)\ra\Box B$.
Then $A^+=\Box(\Box B^+\ra B^+)\ra\Box B^+$, that is
valid also in ${\sf iGL}$.
\item $A=\Box B\ra\Box\Box B$. Then $A^+=\Box B^+\ra\Box\Box
B^+$, that is valid in ${\sf iGL}$.
\item $A=(\Box(B\ra C)\wedge\Box B)\ra\Box C$. Then
$ A^+=(\Box (B\ra C)^+\wedge\Box B^+)\ra\Box C^+$. On
the other hand, $\IPC_{\Box}\vdash (B\wedge(B\ra C))\ra C$ and
hence $\IPC_{\Box}\vdash (B\wedge(B\ra C))^*\ra C^*$, by 
\Cref{Theorem-NNIL Crucial Properties} \cref{4Theorem-NNIL Crucial Properties}. Now we can infer
$\IPC_{\Box}\vdash( B^+\wedge (B\ra C)^+)\ra C^+$, by definition of
$\TNNIL$-algorithm and  \Cref{Lemma-boxed as atomic}.
Finally, by the necessitation rule in ${\sf iGL}$, we have
${{\sf iGL}\vdash(\Box B^+\wedge\Box (B\ra C)^+)\ra \Box C^+}$.
\end{itemize}
\item $A$ is a theorem of $\IPC_{\Box}$. Then $\IPC_{\Box}\vdash
A^+$, by  \Cref{Theorem-NNIL Crucial Properties} \cref{5Theorem-NNIL Crucial Properties} and
 \Cref{Lemma-boxed as atomic}.
\item $A = \Box B$ and $A$ is derived by applying the necessitation rule.
Let ${\sf iGL}\vdash B$. By induction hypothesis, ${\sf iGL}\vdash B^+$ and
then ${\sf iGL}\vdash\Box B^+$.
\item $A$ is derived by modus ponens. Let ${\sf iGL}\vdash B$ and ${\sf iGL}\vdash B\ra A$.
From these, we have ${\sf iGL}\vdash B^+\wedge(B\ra A)^+$ and then 
${\sf iGL}\vdash(B\wedge(B\ra A))^+$. Since $\IPC_{\Box}\vdash
(B\wedge(B\ra A))\ra A$, then by  
 \Cref{Theorem-NNIL Crucial Properties} \cref{4Theorem-NNIL Crucial Properties} we have $\IPC_{\Box}\vdash (B\wedge(B\ra
A))^*\ra A^*$. Then by   \Cref{Lemma-boxed as atomic},
$\IPC_{\Box}\vdash (B\wedge(B\ra A))^+\ra A^+$ and hence
${\sf iGL}\vdash A^+$ as desired.
\end{itemize}
\end{proof}

\begin{corollary}\label{corollar-NNIL properties}
For any modal proposition $A$,
\begin{enumerate}
\item \label{1corollar-NNIL properties} For all $\Sigma_1$-substitution $\sigma$ we have
$\HA\vdash\Box\sigma_{_{\sf HA}}(A)\lr\Box\sigma_{_{\sf HA}}(A^+)$ and  hence $\HA\vdash\sigma_{_{\sf HA}}(A)$
 iff $\HA\vdash\sigma_{_{\sf HA}}(A^+) $,
\item \label{2corollar-NNIL properties}${\sf iGL}\vdash A_1\ra A_2$ implies ${\sf iGL}\vdash A_1^+\ra A_2^+$, and
${\sf iK4}\vdash A_1\ra A_2$ implies ${\sf iK4}\vdash A_1^+\ra A_2^+$,
\item \label{3corollar-NNIL properties}${\sf iGL}\vdash A_1\lr A_2$ implies ${\sf iGL}\vdash A_1^+\lr A_2^+$, and
${\sf iK4}\vdash A_1\lr A_2$ implies ${\sf iK4}\vdash A_1^+\lr A_2^+$.
\end{enumerate}
\end{corollary}
\begin{proof}
The first assertion can be deduced simply by induction on
$\mathfrak{d}A$ and using  \Cref{Theorem-NNIL Crucial Properties} \cref{6Theorem-NNIL Crucial Properties}.

To prove the second part, first note that by 
\Cref{Theorem-NNIL Crucial Properties} \cref{4Theorem-NNIL Crucial Properties}, if $\IPC_{\Box}\vdash
A_1\ra A_2$, then $\IPC_{\Box}\vdash A_1^*\ra A_2^*$. By 
\Cref{Lemma-boxed as atomic}, we can replace each outer occurrence
of boxed formulae by arbitrary propositions, in particular, by
their $\TNNIL$ approximations. We should take care of these replacements to be such that  equal 
propositions be substituted by equal approximations and unequal propositions substituted by unequal ones.  
 Then by definition of $A_i^+$, we
have $\IPC_{\Box}\vdash A_1^+\ra A_2^+$. 

Now suppose that
${\sf iGL}\vdash A_1\ra A_2$ (${\sf iK4}\vdash A_1\ra A_2$). 
Let $A=A_1\to A_2$. This implies
$\IPC_{\Box}\vdash (A\wedge A_1 )\ra A_2$. Then $\IPC_{\Box}\vdash (A\wedge A_1)^+\ra  A_2^+$, 
and hence by $\TNNIL$-algorithm,
$\IPC_{\Box}\vdash (A^+\wedge A_1^+)\ra A_2^+$. This
implies  $\IPC_{\Box}+A^+\vdash A_1^+\ra A_2^+ $ and by 
\Cref{Lemma-TNNIL algorithm preserve theorems of iGL}, ${\sf iGL}\vdash
A_1^+\ra A_2^+$ (${\sf iK4}\vdash A_1^+\ra A_2^+$).

Proof of the third part is a direct consequence of the second part.
\end{proof}

\subsubsection{The $\TNNIL^-$-algorithm}\label{subsubsec-TNNIL^-algorith}
\begin{corollary}\label{corollar HA-NNIL properties}
There exists a $\TNNIL^-$-algorithm such that for any modal
proposition $A$, it halts and produces a proposition
$A^-\in\TNNIL^-$ such that $\IPC_\Box\vdash A^+\ra A^-$.
\end{corollary}
\begin{proof}
Let $A:=B(\Box C_1,\ldots,\Box C_n)$, and $B(p_1,\ldots,p_n)$ is non-modal.
Clearly  such $B$ exists. Then define $A^-:=B(\Box C_1^+,\ldots,\Box C_n^+)$.
Now definition of $A^+$ implies  $A^+=(A^-)^*$ and
 hence   \Cref{Theorem-NNIL Crucial Properties}  \cref{1Theorem-NNIL Crucial Properties}
 implies that $A^-$ has desired property.
\end{proof}

\begin{lemma}\label{Lemma HA-NNIL properties minus}
For each modal proposition $A$ and  $\Sigma_1$-substitution
$\sigma$,  $\HA\vdash \sigma_{_{\sf HA}}A\lr\sigma_{_{\sf HA}}A^-$.
\end{lemma}
\begin{proof}
Use definition of $(.)^-$ and  \Cref{corollar-NNIL properties} \cref{1corollar-NNIL properties}.
\end{proof}

\begin{remark}
{\em Note that $\LC\vdash A\lr B$ does not imply $\LC\vdash
A^+\lr B^+$. A counterexample is $A:=\neg\neg p$ and
$B:=\neg\Boxdot(\neg p)$. We have $A^+=A^*=p$ and 
$ B^+=(\Box\neg p\ra p)$. Now one can use Kripke models to
show $\LC\nvdash (\Box\neg p\ra p)\to p$.}
\end{remark}

\begin{remark}\label{Remark-New def for TNNIL algorithm}{\em
In the algorithm produced for $\NNIL$, let's change the step (1) in this way (and use new symbol $(.)^\dag$
instead of $(.)^*$)
\begin{enumerate}
\item  $A^\dag:= A$ for atomic  $A$,  and $(\Box B)^\dag:=\Box B^\dag$,
\end{enumerate}
Then the new algorithm also halts, and for any modal
proposition $A$, we have ${\sf iK4}\vdash A^\dag\lr A^+$.}
\end{remark}

\subsection{The Box Translation}\label{sec.box.translation}
 The following definition of the box-translation, is
essentially from \cite[Definition~4.1]{Visser82}. The
box-translation extends the well-known G\"odel-McKinsey-Tarski translation.
In this subsection, we prove that ${\sf iGL}$ is closed under box-translation (\Cref{Proposition-propositional properties of Box translation}).
\begin{definition}\label{Definition-Box translation}
For every proposition $A$ in the modal propositional language, we
associate a proposition $A^\Box$, called  the box-translation
of $A$, in the following way:
\begin{itemize}
\item $A^\Box:= A\wedge\Box A$, for atomic $A$, 
\item $(A\circ B)^\Box:=A^\Box\circ B^\Box$, for $\circ\in\{\vee,\wedge\}$,
\item $(A\ra B)^\Box:=(A^\Box\ra B^\Box)\wedge\Box(A^\Box\ra B^\Box)$,
\item $(\Box A)^\Box:=\Box(A^\Box)$.
\end{itemize}
\end{definition}

\begin{lemma}\label{Lemma-Box-translation-prop-0}
For any modal proposition $A$, we have ${\sf iK4}\vdash A^\Box\to\bo A^\bo$.
\end{lemma}
\begin{proof}
Easy induction over the complexity of $A$.
\end{proof}
\noindent In the following lemma we state some properties of $\Boxdot$.
\begin{lemma}\label{Lemma-Properties of boxdot}
For any modal proposition $A$, the following
propositions are provable in ${\sf iK4}$:
\begin{enumerate}
\item \label{1Lemma-Properties of boxdot}$\Box\Boxdot A\lr\Box A\lr \Boxdot \Box A$,
\item \label{2Lemma-Properties of boxdot}$\Boxdot A^\Box\lr A^\Box$.
\end{enumerate}
\end{lemma}
\begin{proof}
The first part is easily deduced in ${\sf iK4}$.  For the second part use \Cref{Lemma-Box-translation-prop-0}.
\end{proof}


We say that a modal theory $T$ is  {\em closed under
box-translation} if for every proposition $A$, $T\vdash A$ implies
$T\vdash A^\Box$.

\begin{proposition}\label{Proposition-propositional properties of Box translation}
The theory ${\sf iGL}$ is
closed under  the box-translation.
\end{proposition}
\begin{proof}
The proof can be carried out in three steps:
\begin{enumerate}[leftmargin=*]
\item For any proposition $A$ first we show that
$\IPC_\Box\vdash A$ implies ${\sf iK4}\vdash A^\Box$. This can
be done by a routine induction on the length of the proof in
$\IPC$. Note that for any axiom $A$ of $\IPC$, we have
${\sf iK4}\vdash A^\Box$. As for the rule of modus ponens, suppose
that $\IPC_\Box\vdash A$ and $\IPC_\Box\vdash A\ra B$. By
induction hypothesis, then ${\sf iK4}\vdash A^\Box$ and
${\sf iK4}\vdash(A^\Box\ra B^\Box)\wedge\Box(A^\Box\ra
B^\Box)$ and so ${\sf iK4}\vdash B^\Box$.
\item Next observe that
$$(\Box A\ra\Box\Box A)^\Box=\Box A^\Box\ra\Box\Box A^\Box$$ and also
$${\sf iK4}\vdash[(\Box(A\ra B)\wedge \Box A)\ra\Box B]^\Box\lr
[(\Box(A^\Box\ra B^\Box)\wedge \Box A^\Box)\ra\Box B^\Box]$$
\item Observe that the box translation of an instance of L\"ob's axiom $\L$, is 
also an instance of $\L$.
\end{enumerate}
\end{proof}

\subsection{Axiomatizing the $\TNNIL$-algorithm}\label{Sec-Aximatizing-TNNIL}

In this subsection we present axioms which we need for the $\TNNIL^-$-algorithm
$(.)^-$. More precisely, we will find some axiom set $X$  
such that ${X\vdash A^-\lr A}$. 

To do that,  we use some relation $\brt$ on modal propositions. A variant of this relation for non-modal case, first 
appeared in \cite{Visser02}.
 The relation $\brt$ is defined to be the
smallest relation on modal propositions satisfying the following conditions:
\begin{itemize}
 \item[A1.] If ${\sf iK4}\vdash A\ra B$, then $A\brt B$,
 \item[A2.] If $A\brt B$ and $B\brt C$, then $A\brt C$,
 \item[A3.] If $C\brt A$ and $C\brt B$, then $C\brt A\wedge B$,
 \item[A4.] If $A\brt B$, then $\bo A\brt \bo B$,
 \item[B1.] If $A\brt C$ and $B\brt C$, then $A\vee B\brt C$,
  \item[B2.]   Let $X$ be a set of implications, $B:=\bigwedge X$ and $A:=B\ra C$.
 Also assume that $Z:={\{E | E\ra F\in X\}}\cup \{C\}$. Then $A\brt [B]Z $,
 \item[B3.] If $A\brt B$, then for any  atomic or boxed $C$ we have $C\ra A\brt C\ra B$.
\end{itemize}
\begin{remark}\label{Remark1}
Let $A$, $B$ and $Z$ be as in {\em B2}. Then the relation $\brt$ has the following additional property: $$A\brt [B]'Z$$
The reason goes as follows. 
First observe, by induction on $E$ and using A1-A3, that $[B]E\brt [B]'E$. That by use of A1-A3 and B1, implies that 
$[B]Z\brt[B]'Z$. Hence by A2 and B2, we have $A\brt [B]'Z$.
\end{remark}
The notation $A\blrt B$ means $A\brt B$ and $B\brt A$. Let us define the theory
 $$\lles:=\lle+\CP_{a}+\{\Box A\ra\Box B| A\brt B\}$$ 
Note that by A1, the relation $\brt$ contains all the pairs $(A,B)$ such that ${\sf iK4}\vdash A\to B$.  But 
 it worth mentioning that the inclusion is strict. The axiom which makes $\brt$ strictly superset of 
 $\{(A,B):{\sf iK4}\vdash A\to B\}$ is B2, i.e. in the absence of B2, the relation $\brt$ is the same as provable implications 
 in ${\sf iK4}$. 
 However, with B2 the story is different,  e.g. one can observe that $\neg\neg p\brt p$, for any atomic $p$, holds while
 ${\sf iK4}\nvdash \neg\neg p\to p$.

\vspace{.1in}

\noindent\textbf{Notation.} In the rest of the paper, we use $A\equiv B$
as a shorthand for ${\sf iK4}\vdash A\lr B$.

\vspace{.1in}

\noindent The following theorem, shows that A1-A4 and B1-B3, axiomatize  the $\TNNIL$ algorithm: 
\begin{theorem}\label{Theorem HA-NNIL approximation is propositionally equivalent}
For any modal proposition $A$, we have $A\blrt A^+$.
\end{theorem}
\begin{proof}
We prove the desired result by induction on $\mathfrak{o}(A)$. Suppose we
have the desired result for each  proposition $B$ with
$\mathfrak{o}(B)<\mathfrak{o}(A)$. We treat $A$
by the following cases.

\begin{enumerate}[leftmargin=*]
\item (A1) $A$ is atomic. Then $A^+=A$  by definition,
and the result holds trivially.
\item (A1-A4, B1) $ A=\Box B , A=B\wedge C, A=B\vee C$. All these cases hold
by induction hypothesis. In boxed case, we use induction
hypothesis and A4. In conjunction, we use  A1-A3 and in
disjunction we use A1, A2 and B1.
\item $ A=B \ra C $. There are several sub-cases. Similar to the
definition of the $\NNIL$-algorithm, an occurrence of a sub-formula $B$ of $A$ is
said to be an {\em outer occurrence} in $A$, if it is neither in
the scope of a $\Box$ nor in the scope of $\ra$.
\begin{enumerate}[leftmargin=*]
\item (A1-A3) $C$ contains an outer occurrence of a conjunction.
We can treat this case using induction hypothesis and
$\TNNIL$-algorithm.
\item (A1-A3) $B$ contains an outer occurrence of a disjunction.
We can treat this case by induction hypothesis and
$\TNNIL$-algorithm.
\item $B=\bigwedge X$ and $C=\bigvee Y$, where $X$ and  $Y$ are sets
of implications, atoms and boxed formulae. We have several
sub-cases:
\begin{enumerate}[leftmargin=*]
\item (A1, A2, B3)  $X$ contains atomic variables. Let $p$ be an atomic
variable in $X$. Set $D:=\bigwedge(X\setminus\{p\})$. Then $A^+ \equiv p\to (D\to C)^+$.
On the other hand, we have by induction hypothesis and A2 and B3, that
$p\ra (D\ra C)^+\blrt p\ra(D\ra C)$. Finally by A1 and A2, we have $A^+\blrt A$.
\item (A1, A2, B3) $X$ contains boxed formula. Similar
to the previous case.
\item (A1, A2) $X$ contains $\top$ or $\bot$. Trivial.
\item (A1-A3, B1-B3) $X$ contains only implications. This case needs the
axiom B2 and it seems to be the interesting case. we have: 
$$A^+\equiv \bigwedge\left\lbrace\left(B\!\downarrow\!
D\ra C\right)^+ \mid D\in
X\right\rbrace\wedge\bigvee\{([B]'E)^+:E\in Z\} \quad$$
By the argument in \cite{Visser02}, we have $\mathfrak{o}\left(B\!\downarrow\! D\ra C\right)<\mathfrak{o}(A)$ and 
${\mathfrak{o}([B]'E)<\mathfrak{o}(A)}$ and hence one can apply induction hypothesis on $B\!\downarrow\! D\ra C$ and $[B]'E$.
Then by induction hypothesis, A1-A3, B1 and B3, we have:
\begin{align*}
A^+&\blrt \bigwedge\left\lbrace B\!\downarrow\!
D\ra C \mid D\in
X\right\rbrace\wedge [B]'Z
\end{align*}

First we show that for each $E\in Z$,
\begin{equation}\label{Eq-100}
{\sf iK4}\vdash\left(\bigwedge\{(B\!\downarrow\! D)\ra C\mid D\in X\}\wedge [B]'E\right)\ra A
\end{equation} 
Since $[B]E$ and $[B]'E$ are $\IPC_\Box$-equivalent (\Cref{Remark0}), it's enough to show that 
\begin{equation}\label{Eq-101}
{\sf iK4}\vdash\left(\bigwedge\{(B\!\downarrow\! D)\ra C\mid D\in X\}\wedge [B]E\right)\ra A
\end{equation} 
If $E=C$, we are done by $\IPC_\Box\vdash[B]C\ra(B\ra C)$. So let 
$E$ be the antecedent of  some $E\ra F\in X$. We reason in ${\sf iK4}$. Assume
$\bigwedge\{(B\!\downarrow\! D\ra C\mid D\in X\}$, $[B]E$ and $B$ as the hypothesis. We
want to derive $C$.  From $B$ and $[B]E$, we
derive $E$. Also from $B$, we derive $E\ra F$, and so $F$. Hence
we have $\bigwedge(X\setminus\{E\ra F\})\wedge F$, which implies
$C$, as desired.

Now  \cref{Eq-101} by use of A1 and A2 implies $A^+\brt A$.  

To show the other way around, i.e.  $A\brt A^+$, 
we first show
\begin{equation}\label{Eq-101}
 A\brt \bigwedge\left\lbrace
B\!\downarrow\! D\ra C \mid D\in X\right\rbrace\wedge[B]'Z
\end{equation}
 and then  by use of induction hypothesis and A2, we can deduce 
$A\brt A^+$, as desired. So it remains to show that \cref{Eq-101} holds. 
We have $\IPC_\bo\vdash A\ra \bigwedge\left\lbrace B\!\downarrow\!
D\ra C \mid D\in X\right\rbrace$, and hence by A1, 
$A\brt \bigwedge\left\lbrace B\!\downarrow\! D\ra C \mid D\in X\right\rbrace $. 
On the other hand, by \Cref{Remark1}, we have  $A\brt [B]'Z$. Now A3 implies \cref{Eq-101}, as desired.
\end{enumerate}
\end{enumerate}
\end{enumerate}
\end{proof}

\begin{corollary}\label{corollar HA-NNIL approximation is propositionally equivalent}
$\lles\vdash A^-\lr A$.
\end{corollary}
\begin{proof}
Let $A=B(\Box C_1,\Box C_2,\ldots,\Box C_n)$
where $ B(p_1,\ldots,p_n)$ is a non-modal proposition. 
 By definition of
$A^-$, we have $A^-=B(\Box
C_1^+,\ldots,\Box C_n^+)$. 
Then 
\Cref{Theorem HA-NNIL approximation is propositionally equivalent}
implies that 
${\lles\vdash \Box (C_i)^+\lr \Box C_i}$. Hence $\lles\vdash A^-\lr A$, as desired.
\end{proof}


\subsection{$\TNNIL$-Conservativity of $\LC$ over $\lle$}\label{sec.tnnil.conservativity}
It is clearly the case that $\LC\supseteq\lle$. One can use
Kripke models (from the next section) to show
$\neg\neg\Box\bot\in\LC\setminus\lle$. This implies that the
inclusion is strict.  As we will see later in this section, $\LC$
and $\lle$ have the same $\TNNIL$-theorems. To prove this, we need
some lemmas.

\begin{lemma}
${\sf iK4}+{\sf Le}^+\vdash {\sf Le}$.
\end{lemma}
\begin{proof}
Assume some axiom instance of ${\sf Le}$, $\Box(B\vee
C)\ra\Box(\Box B\vee C)$. Let $A:=B\vee C$. By axiom schema
${\sf Le}^+$, we have $\Box A\ra\Box A^l$, which is
$\Box(B\vee C)\ra\Box(\Boxdot B\vee \Boxdot C)$. This
implies (inside ${\sf iK4}$) $\Box(B\vee C)\ra\Box(\Box B\vee
C)$ .
\end{proof}

\begin{lemma}\label{Lemma-leivant's translation properties}
For each modal proposition $A$,
\begin{enumerate}
\item \label{1Lemma-leivant's translation properties}If $A\in{\sf NOI}$, then ${\sf iK4}+{\sf CP}_{\sf a}\vdash A\ra\Box A$.
\item \label{2Lemma-leivant's translation properties}${\sf iK4}\vdash A^l\ra A$.
\item \label{3Lemma-leivant's translation properties}If $A\in {\sf NOI}$, then ${\sf iK4}+\CP_{\sf a}\vdash A^l\lr A$.
\item \label{4Lemma-leivant's translation properties}${\sf LLe}^+\vdash \Box A^l \lr\Box A$.
\end{enumerate}
\end{lemma}
\begin{proof}
Proofs of  \cref{1Lemma-leivant's translation properties,2Lemma-leivant's translation properties,3Lemma-leivant's translation properties} are routine by induction on $A$. \Cref{4Lemma-leivant's translation properties}
is deduced from \cref{2Lemma-leivant's translation properties}, i.e. we have $\Box A^l\ra\Box A$, by
\cref{2Lemma-leivant's translation properties} and $\Box A^l\leftarrow\Box A$ is exactly ${\sf
Le}^+$.
\end{proof}

\begin{lemma}\label{Lemma-leivant's translation vs box translation}
For any $\TNNIL$ formula $A$, we have
\begin{enumerate}
\item \label{1Lemma-leivant's translation vs box translation}$\lle\vdash\Boxdot A^l\lr\Boxdot A^\Box$,
\item \label{2Lemma-leivant's translation vs box translation}$\lle\vdash A^\Box\ra A$,
\item \label{3Lemma-leivant's translation vs box translation}If $A\in {\sf NOI}$, then $\lle\vdash A^\Box\lr A$,
\item \label{4Lemma-leivant's translation vs box translation}$\lle\vdash \Box A\lr\Box A^\Box$.
\end{enumerate}
\end{lemma}
\begin{proof}
We prove all items by induction on the complexity of $A$,
simultaneously. In the middle of proof, when we are using
induction hypothesis of item $i, 1\leq i\leq 4$, we mention the
number in parenthesis that number and also when we deduce some
item of lemma, we also mention the number of that part in
parentheses as well.

\vspace{.1in}

\noindent{\em Atomic:} For atomic $A$, we have $A^l=A$ and
$A^\Box=\Boxdot A $, hence by properties of $\Boxdot$ 
(\Cref{Lemma-Properties of boxdot} \cref{1Lemma-Properties of boxdot}), 
${\sf iK4}\vdash \Boxdot A^l\lr\Boxdot A^\Box$ (\cref{1Lemma-leivant's translation vs box translation}) and  
$\lle\vdash A^\Box \lr A$ (\cref{2Lemma-leivant's translation vs box translation,3Lemma-leivant's translation vs box translation}),
 which by necessitation, that implies
$\lle\vdash\Box A^\bo\lr\bo A$ (\cref{4Lemma-leivant's translation vs box translation}).

\vspace{.1in}

\noindent{\em Boxed:} Let $A=\Box B$. Then by definition,
$A^l=A$ and $A^\Box=\Box B^\Box$. Hence by induction
hypothesis (\cref{1Lemma-leivant's translation vs box translation}), 
$\lle\vdash \Boxdot B^l\lr\Boxdot
B^\Box$. Then $\lle\vdash \Boxdot\Box B^l\lr\Boxdot\Box
B^\Box$. On the other hand, by  \Cref{Lemma-leivant's translation properties} 
\cref{4Lemma-leivant's translation properties}, $\lle\vdash \Box B^l\lr\Box B$.
Hence $\lle\vdash \Boxdot\Box B\lr\Boxdot\Box B^\Box$
(\cref{1Lemma-leivant's translation vs box translation}). Also by induction hypothesis 
(\cref{4Lemma-leivant's translation vs box translation}), $\lle\vdash
\Box B\lr\Box B^\Box$. Hence $\lle\vdash A\lr A^\Box$
(\cref{2Lemma-leivant's translation vs box translation,3Lemma-leivant's translation vs box translation}). That, by necessitation, implies $\lle\vdash \bo
A\lr \bo A^\Box$ (\cref{4Lemma-leivant's translation vs box translation}).

\vspace{.1in}

\noindent{\em Conjunction: } This case is trivial.

\vspace{.1in}

\noindent{\em Disjunction: } Assume $A=B\vee C$. If $A\in {\sf
NOI}$, then $B,C\in {\sf NOI}$ and hence induction hypothesis for
$B$ and $C$  (\cref{3Lemma-leivant's translation vs box translation}) 
implies $\lle\vdash A^\Box\lr A$ (\cref{3Lemma-leivant's translation vs box translation}).
 For the other parts, we have, by definition, $A^l=\Boxdot
B^l\vee\Boxdot C^l$. Hence by induction hypothesis (\cref{1Lemma-leivant's translation vs box translation}),
$\lle\vdash A^l\lr(\Boxdot B^\Box\vee\Boxdot C^\Box)$.
Hence, by  \Cref{Lemma-Properties of boxdot} \cref{2Lemma-Properties of boxdot}, we derive
 $\lle\vdash \Boxdot
A^l\lr\Boxdot A^\Box$ (\cref{1Lemma-leivant's translation vs box translation}). 
To prove the \cref{2Lemma-leivant's translation vs box translation}, 
note that, by induction hypothesis (\cref{2Lemma-leivant's translation vs box translation}), $\lle\vdash
B^\Box\ra B$ and $\lle\vdash C^\bo\ra C$. Hence $\lle\vdash
(B\vee C)^\bo\ra (B\vee C)$ (\cref{2Lemma-leivant's translation vs box translation}). 
To prove \cref{4Lemma-leivant's translation vs box translation}, we note
that, by \cref{1Lemma-leivant's translation vs box translation} for $A$, we have  $\lle\vdash \Boxdot
A^l\lr\Boxdot A^\Box$. Hence $\lle\vdash \bo\Boxdot
A^l\lr\bo\Boxdot A^\Box$, which implies $\lle\vdash \bo
A^l\lr\bo A^\Box$ (by  \Cref{Lemma-Properties of boxdot} \cref{1Lemma-Properties of boxdot}). 
Now  \Cref{Lemma-leivant's translation properties} \cref{4Lemma-leivant's translation properties}
 implies $\lle\vdash \bo A\lr\bo A^\Box$ (\cref{4Lemma-leivant's translation vs box translation}).

\vspace{.1in}

\noindent{\em Implication: } Assume $A=B\ra C$. Clearly $A\not\in
{\sf NOI}$ and $B\in {\sf NOI}$. We only show induction claim for
\cref{1Lemma-leivant's translation vs box translation}. 
The other items can be shown easily. By induction
hypothesis (\cref{3Lemma-leivant's translation vs box translation}), $\lle\vdash B^\Box\lr B$, and also by
 \Cref{Lemma-leivant's translation properties} \cref{1Lemma-leivant's translation properties}, $\lle\vdash
\Boxdot B\lr B$. Note that $\Boxdot A^l = \Boxdot(B\ra C^l)$ and
hence $\lle\vdash \Boxdot A^l\lr\Boxdot(\Boxdot B\ra C^l)$. By
 \Cref{Lemma-leivant's translation properties}  \cref{3Lemma-leivant's translation properties}, $\lle\vdash
B^l\lr B$, hence $\lle\vdash \Boxdot A^l\lr\Boxdot(\Boxdot B^l\ra
C^l)$. Now properties of $\Boxdot$ implies that $\lle\vdash\Boxdot
A^l\lr\Boxdot(\Boxdot B^l\ra\Boxdot C^l) $, and induction
hypothesis (\cref{1Lemma-leivant's translation vs box translation}), implies $\lle\vdash \Boxdot
A^l\lr\Boxdot(\Boxdot B^\Box\ra\Boxdot C^\Box)$. This
implies, again by properties of $\Boxdot $, the desired result,
$\lle\vdash \Boxdot A^l\lr\Boxdot A^\Box$ (\cref{1Lemma-leivant's translation vs box translation}).
\end{proof}

\begin{lemma}\label{Lemma-LLe+-versus-iGL}
If $\LC\vdash A$, then ${\sf iGL}\vdash A^\Box$.
\end{lemma}
\begin{proof}
From $\LC\vdash A$ we have 
${\sf iGL}\vdash\bigwedge_i\Boxdot(B_i\ra\Box B_i)\ra A$. Hence
by \Cref{Proposition-propositional properties of Box translation}, we have 
${\sf iGL}\vdash[\bigwedge_i\Boxdot(B_i\ra\Box B_i)\ra A]^\Box$.  This implies 
${\sf iGL}\vdash\bigwedge_i\Boxdot(B_i^\Box\ra \Box B_i^\Box)\ra A^\Box$.
 Now \Cref{Lemma-Properties of boxdot} \cref{2Lemma-Properties of boxdot} implies ${\sf iGL}\vdash A^\Box$, as desired.
\end{proof}

\begin{theorem}\label{Theorem-TNNIL Conservativity of LC over LLe+}
$\LC$ is $\TNNIL$-conservative over $\lle$.
\end{theorem}
\begin{proof}
Let $\LC\vdash A$. 
 From \Cref{Lemma-LLe+-versus-iGL} we have  
${\sf iGL}\vdash A^\Box $
and hence $\lle\vdash A^\Box$. Now  
\Cref{Lemma-leivant's translation vs box translation} \cref{2Lemma-leivant's translation vs box translation} 
implies that $\lle\vdash A$.
\end{proof}

\subsection{Kripke semantics for \LC}\label{Sec-PropModKripke}

Let us first review results and notations from \cite{IemhoffT}
which will be used here. Assume two binary relations $R$ and $S$
on a set. Define $\alpha(R\circ S)\gamma$ iff there exists some $\beta$ such that 
$\alpha R \beta $ and $\beta S \gamma$.

A Kripke model $\mathcal{K}$, for intuitionistic modal logic, is
a quadruple $(K,<,\R,V)$, such that $K$ is a set (we call
its elements as nodes), $(K,<)$ is a partial ordering, $\R$ is a
binary relation on $K$ such that $(\leq\circ\, \R)\subseteq\;\R$, and
$V$ is a binary relation between nodes and atomic
variables such that $\alpha V p$ and $\alpha \leq \beta$ implies $\beta V
p$. Then we can extend $V $ to the modal language with $\R$
corresponding to $\Box$ and $\leq $ for intuitionistic $\ra$.
More precisely, we define $\Vdash$ inductively as an extension of $V$ as follows:
\begin{itemize}
 \item $\kcal,\alpha\Vdash p$ iff $\alpha Vp$, for atomic variable $p$,
 \item $\kcal,\alpha\Vdash A\vee B$ iff $\kcal,\alpha\Vdash A$ or $\kcal,\alpha\Vdash B$,
 \item $\kcal,\alpha\Vdash A\wedge B$ iff $\kcal,\alpha\Vdash A$ and $\kcal,\alpha\Vdash B$,
 \item $\kcal,\alpha\nVdash\bot$ and $\kcal,\alpha\Vdash\top$,
 \item $\kcal,\alpha\Vdash A\ra B$ iff for all $\beta\geq\alpha$, $\kcal,\beta\Vdash A$ implies $\kcal,\beta\Vdash B$,
 \item $\kcal,\alpha\Vdash \bo A$ iff for all $\beta$ with $\alpha\,\R\, \beta$, we have $\kcal,\beta\Vdash A$.
\end{itemize}
If also we assume that $\R$ is empty and restrict our attention to non-modal language, we have the usual 
Kripke models for intuitionistic (non-modal) logic.
In the rest of paper, we may simply  write $\alpha\Vdash A$ instead of $\kcal,\alpha\Vdash A$,
 if no confusion is likely.
By an induction on the complexity of $A$, one can observe that $\alpha\Vdash A$
implies $\beta\Vdash A$ for all $A$ and $\alpha\leq \beta$. We define the
following notions.
\begin{itemize}
\item If $\alpha\leq\beta$, $\beta$ is called to be above $\alpha$ and
$\alpha$ is beneath $\beta$. If  $\alpha\;\R\;\beta$, $\beta $ is
called to be a successor of $\alpha$. We define $\R(\alpha)$ to be
the set of all successors of $\alpha$.
\item A Kripke model is finite if its set of nodes is finite.
A Kripke model is tree-frame if its set
of nodes with ordering $\leq$ is a tree.
\item A Kripke model $\mathcal{K}=(K,<,\R,V)$ is
reverse well-founded iff $K$ is well-founded with the ordering
$\R^{-1}$.
\item $\mathcal{K}$ is called {\em neat} iff  $\alpha\,\R\,\gamma$ and $\alpha\leq\beta\leq\gamma$ implies $\alpha\,\R\,\beta$
or $\beta\,\R\,\gamma$.
\item  $\mathcal{K}$ is called {\em brilliant} iff $(\R\;\circ \leq)\subseteq\R$. (\cite{IemhoffT})
\item  $\mathcal{K}$ is called {\em perfect} iff it is
brilliant, reverse well-founded and $R\subseteq\,<$. 
\item Suppose $X$ is a set of propositions that is closed under
sub-formulae (we call such $X$ to be adequate). An $X$-saturated set of
propositions $\Gamma$ with respect to some theory $T$ is a subset
of $X$ that
\begin{itemize}
 \item For each $A\vee B\in X $, $T+\Gamma\vdash A\vee B$ implies $A\in\Gamma$ or $B\in \Gamma$.
 \item For each $A\in X$, $T+\Gamma\vdash A$ implies $A\in\Gamma$.
\end{itemize}
\end{itemize}
\begin{lemma}\label{Lemma-saturation}
Let $T\nvdash A$ and let $X$ be an adequate set. Then there is an $X$-saturated  set $\Gamma$ such that 
$T\cap X\subseteq\Gamma\nvdash A$.
\end{lemma}
\begin{proof}
See \cite{IemhoffT}.
\end{proof}


\begin{theorem}\label{Theorem-Propositional Completeness LC}
$\LC$ is sound and complete for finite neat perfect Kripke
models with tree frames.
\end{theorem}
\begin{proof}
Soundness part can easily be proved by induction on the complexity
of formulae. For the completeness, we first find some finite
perfect Kripke counter-model for each $A$ with $\LC\nvdash A$,
and then convert it to a perfect Kripke model with finite 
tree frame.

Assume $\LC\nvdash A$. Let ${\sf Sub}(A)$ be the set of sub-formulae of
$A$. Then define  $$X:=\{B,\Box B \ | \ B\in {\sf Sub}(A)\}$$ It is
obvious that $X$ is a finite adequate set. We define
$\mathcal{K}=(K,<,\R,V)$ as follows. Take $K$ as the set of
all $X$-saturated sets with respect to $\LC$, and $\leq$ is the
subset relation over $K$. Define $\alpha\,\R\, \beta$ iff for all $\Box
B\in X$, $ \Box B\in \alpha$ implies  $B\in \beta$, and also there
exists some $\Box C\in \beta\setminus \alpha$. Finally define $\alpha V
p$ iff $p\in \alpha$,  for atomic $p$. \\
It only remains to show that $\mathcal{K}$ is a finite perfect
Kripke model that refutes $A$. To do this, we first show by
induction on $B\in X$ that $B\in \alpha$ iff $\alpha\Vdash B$, for each
$\alpha\in K$. The only non-trivial case is $B=\Box C$. Let
$\Box C\not\in \alpha$. We must show $\alpha\nVdash \Box C$. The other
direction is easier to prove and we leave it to reader. Let
$\beta_0:=\{D\in X\ |\ \alpha\vdash\Box D\}$. If $\beta_0,\Box C\vdash
C$, then, by definition of $\beta_0$, we have $\alpha\vdash\Box \beta_0$ and
hence by L\"{o}b's axiom,  $\alpha\vdash \Box C$, contradicting
$\Box C\not\in \alpha$. Hence $\beta_0,\Box C\nvdash C$ and so there
exists some $X$-saturated set $\beta$ such that $\beta\nvdash C$,
$\beta\supseteq \beta_0\cup\{\Box C\}$. Hence $\beta\in K$ and $\alpha\,\R\, \beta$.
Then by induction hypothesis, $\beta\nVdash C$ and hence $\alpha\nVdash
\Box C$.

Since $\LC\nvdash A$, by \Cref{Lemma-saturation}, there exists some
$X$-saturated set $\alpha\in K$ such that $\alpha\nvdash A$, and hence by
the above argument we have $\alpha\nVdash A$.
\\
$\mathcal{K}$ trivially satisfies all the properties of perfect
Kripke model. As a sample, we show that why $\R\subseteq\, <$ holds.
Assume $\alpha\,\R\, \beta$ and let $B\in \alpha$. If $B$ is boxed formula, like
$C$, then by definition, $C\in \beta$ and hence $\beta\vdash B$ and we are
done. So assume  $B$ is not a boxed formula. Then by definition
of $X$, we have $\Box B\in X$ and  by the completeness axiom in
$\LC$, we have  $\alpha\vdash\Box B$ and hence by definition of
$\R$, it is the case that $B\in \beta$. This shows $\alpha\subseteq \beta$ and
hence $\alpha\leq \beta$. But $\alpha$ is not equal to $\beta$, because $\alpha\,\R\, \beta$
implies existence of some $\Box C\in \beta\setminus \alpha$. Hence $\alpha<
\beta$, as desired.

Now we explain how to convert $\mathcal{K}$ to a 
 Kripke model $\mathcal{T}:=(T,<_t,\R_{t},V_t)\nVdash A$
 with a neat tree frame. Let $T$ be the set
of all finite (excluding empty sequence) sequences $\langle
\alpha_1, \ldots,\alpha_n\rangle$ such that ${\alpha_1 <\ldots<\alpha_n}$. Let
$\leq_t$ be the initial segment relation. 
Then  define $\langle \alpha_1, \ldots,\alpha_n\rangle\;\R_{t}\;\langle\alpha_1, \ldots,\alpha_{n+k}\rangle$
 iff $\alpha_{n+i} \,\R\, \alpha_{n+i+1}$ for some $0\leq i<k$.
Finally, define $\langle \alpha_1, \ldots,\alpha_n\rangle\; V_t \; p$, for atomic $p$, iff 
$\alpha_n\;V\;p$. Now one can prove by induction on
$B$, that for any $\alpha=\langle \alpha_1, \ldots,\alpha_n\rangle\in T$, $\mathcal{T},\alpha\Vdash B$ iff
$\kcal,\alpha_n\Vdash B$. Hence $\mathcal{T}\nVdash A$.
\end{proof}
\noindent Since $\LC$ has finite model property, as it is expected, we can easily deduce the decidability of $\LC$:
\begin{corollary}\label{Corollary-decidability of LC}
$\LC$ is decidable.
\end{corollary}
\begin{proof}
Let $A$ be given. Assume that $n$ is the number of elements of
$X$ defined in the above proof. It shows us that we should only
check if for all  Kripke models $\mathcal{K}$ with $2^n$
nodes (only over atomic variables that appear in $A$), we have
$\mathcal{K}\Vdash A$. If that was the case, we say ``yes" to
$\LC\vdash A$?, otherwise the answer is ``no"  to $\LC\vdash A$?.
\end{proof}

\subsubsection*{Relation to intuitionistic non-modal Kripke models}

The usual intuitionistic non-modal Kripke models are the same
Kripke models as is defined above, without the additional
relation $\R$. Extending it to all non-modal propositions is the
same as the  one for modal language. It is well-known that $\IPC$ is sound and
complete for non-modal Kripke models. We have the following
conservativity result.

\begin{theorem}\label{Theorem-nonmodal conservativity of LC over IPC}
$\LC+\bo\bot$ is conservative over $\IPC$ in non-modal language,
i.e. for any non-modal proposition $A$, if 
$\LC+\bo\bot\vdash A$, then $\IPC\vdash A$.
\end{theorem}
\begin{proof}
We reason contrapositively. Assume that $\IPC\nvdash A$. Then by
completeness, there exists some non-modal Kripke model
$\mathcal{K}=(K,\leq,V)\nVdash A$. Let $\R:=\emptyset$ and
$\mathcal{K}':=(K,\leq,\R,V)$. It is easy to observe that
$\mathcal{K}'$ is a Kripke model and for all non-modal
proposition $B$ and $u\in K$, we have $\mathcal{K},u\Vdash B$ iff
$\mathcal{K}',u\Vdash B$. Hence we have $\mathcal{K}'\nVdash A$.
Now soundness theorem (\Cref{Theorem-Propositional Completeness LC})
implies $\LC+\bo\bot\nvdash A$.
\end{proof}

\section{Transforming Kripke models}\label{sec-transforming}
Smor\'ynski (\cite{Smorynski-Thesis}) showed that one could simulate the behaviour of a propositional non-modal Kripke model by  a first-order Kripke model of $\HA$.  
Also Solovay (\cite{Solovay}) showed that one could simulate the behaviour of 
a Kripke model of classical modal logic inside $\PA$. 
However, the combination of these two ideas could be assumed  as major
obstacle towards the  characterization of the provability logic of $\HA$. 
In this section, we will show that one could simulate the behaviour of perfect Kripke models  by 
first-order Kripke models of $\HA$.  This would lead us to the characterization of the 
$\Sigma_1$-provability logic of $\HA$.
More precisely, we will prove the following theorem. 

\begin{theorem}\label{Theorem-Main tool}
Let  $\kcal_0=(K_0,\R_{0},<_0,V_0)$ be a finite neat perfect Kripke model with tree frame
 and  $\Gamma\subseteq\TNNIL^-$ be a finite set. Then  there exists some arithmetical $\Sigma_1$-substitution 
$\sigma$ and a Kripke model $\kcal_1=(K_0,<_0,\mathfrak{M})$ such that for all $A\in \Gamma$ and $\alpha\in K_0$ we have  $\kcal_0,\alpha\Vdash A$ iff $\kcal_1,\alpha\Vdash \sigma_{_{\sf HA}}(A)$.
\end{theorem}

Before we continue with the rather long proof of \Cref{Theorem-Main tool}, that will take up all of this section,
let us explain the outline of the proof. 
 
First we define a recursive function $F$ (the Solovay function) with the domain of natural numbers. 
$F(0)$ is defined to be some fresh node $\alpha_0$. The function $F$, always climbs over the frame ${(K_0,\R_{0},<_0)}$, 
but it is reluctant to do so. It only goes to some node $\beta$ at some stage $n+1$ (i.e. $F(n+1)=\beta$), if $n+1$ is a witness  (in some sense which would be clarified in this section) for this statement 
\begin{center}
``$F$ is not going to stay 
in $\beta$ forever or  $\neg\sigma_{_{\sf HA}}(\Box\varphi_{_\beta})$".
\end{center}
In this  definition, $\varphi_{_\beta}$ is the conjunction of all sentences $B$ such that $ \Box B\in{\sf Sub}(\Gamma)$
and  $\beta\Vdash\Box B$. Here, ${\sf Sub}(\Gamma)$ is the set of sub-formulae 
 of  some formula in $\Gamma$. The most interesting (and difficult part to prove as well) property of the function $F$
 is that this function actually (in the standard model of arithmetic $\mathbb{N}$) does not climb over tree at all, i.e. 
 the function $F$ is constant,  $\mathbb{N}\models\forall{x}F(x)=\alpha_0$. In contrast with the classical case, proving this fact for the intuitionistic case is rather complicated.
 
  Let $L\succcurlyeq\alpha$ denote ``the function $F$ would go above $\alpha$ or remain equal to $\alpha$".
 Then we define the substitution $\sigma(p):=\bigvee_{\alpha\Vdash p}L\succcurlyeq\alpha$.  
 Then we define the $I$-frame $\mathcal{I}=(K_0,<_0,T)$, where $T_\alpha$ is defined to be $\PA$  plus the following statement:
 ``The limit of the function $F$ is $\alpha$".  Finally, with the aid of 
 \Cref{Theorem-Smorynski's general method of Kripke model construction} 
 we find the desired Kripke model $\kcal_1$, by assigning an appropriate classical model of $T_\alpha$ to the node $\alpha$.
 We will show that $T_\alpha\vdash \sigma_{_{\sf HA}}( \Box\varphi_{_\alpha})$
  (\Cref{corollar-4st&2st}) and also $T_\alpha\vdash \sigma_{_{\sf HA}}(\neg\Box B)$ for any 
  $\Box B\in {\sf Sub}(\Gamma)$ and $\alpha\nVdash\Box B$ 
  (\Cref{Lemma-2st Properties of Solovay's Function}). In this way, we can simulate the role of modal operator $\Box$  
  in the first-order Kripke model $\kcal_1$. 
 
 \vspace{.3cm}
 \noindent {\bf Notation.} In the rest of this section, we fix the Kripke model $\kcal_0=(K_0,\R_{0},<_0,V_0)$ 
 and the set $S:=\{B\in{\sf Sub}(\Gamma) \mid B\in\TNNIL\}$. We also assume that $\alpha_0\not\in K_0$
 and define 
 $$
 \R:=\R_{0} \cup \{ (\alpha_0,\alpha)\mid\alpha\in K_0\}
 \quad   \quad  
 <\; :=\; <_0\cup\; \{ (\alpha_0,\alpha)\mid\alpha\in K_0\}
 \quad  \quad  
 K:=K_0\cup \{\alpha_0\}
$$
In other words, we add $\alpha_0$ in  beneath of all nodes of $\kcal_0$. Finally we define $\kcal:=(K,\R,<,V_0)$.

\subsection{Definition of the Solovay function} \label{subsection-solovay}
Solovay used some  special recursive function (here we call it the {\em Solovay function}) to prove the completeness 
of $\GL$ (The G\"odel-L\"ob logic) for arithmetical interpretations in $\PA$ (See \cite{Solovay}). 
The Solovay function in \cite{Solovay},  is a function $G:\mathbb{N}\longrightarrow X$, in which $X$ is a finite partially ordered set ordered by $\preccurlyeq$. The recursive definition of $G$ is such that $G$ climbs over $X$, i.e. $G(x)\preccurlyeq G(x+1)$ and moreover, it goes to some new node iff there exists a witness  that $G$ would not remain there. More precisely, $G(x+1)\neq G(x)$ iff $x+1$ is the code of a proof (in $\PA$) for the fact that the limit of the function $G$ is not $G(x+1)$.
Although it is true (in the standard model) that $G$ will not climb over $X$ (i.e. $G$ is a constant function), $\PA$ 
can't prove this fact.
In this subsection, we define  a similar recursive function  (we call it  $F$) for the proof of our main theorem
 (\Cref{Theorem-Main tool}), and state and prove some of its properties.

For technical reasons, we first define the set of all codes of sequences $z=\langle F(0),\ldots,F(x)\rangle$ by an arithmetical formula 
$\theta(z)$, and then define  $\phi_\theta(x,y):=\exists{z}({\sf lth}(z)=x+1\wedge \theta(z)\wedge \hat{z}=y)$ 
as the graph of a function $F_\theta$ and finally, let $F:=F_\theta$.
It is clear that we can also define $\theta_F(z)$ from the function $F$ in the following way. 

$$\theta_F(z):=\exists{x}(z=\langle F(0),\ldots,F(x)\rangle) \text{ or equivalently }
\theta_F(z):=\forall{\,x \!<\! {\sf lth}(z)}(F(x)=(z)_{x})$$

To be able to speak about $\mathcal{K}$ inside $\HA$, we need
some conventions. Suppose that
$K=\{\alpha_0,\alpha_1,\ldots,\alpha_k\}$. Hence for each
$\alpha\in K$, there exists a unique  index $0\leq i\leq k$ such that
$\alpha=\alpha_i$. We define $\overline{\alpha}$ to be $\bar{i}$
($\bar{n}$ is  $n$-th numeral in the language of arithmetic, i.e.
$\bar{i}:=S^i(0)$). We may simply use $\alpha$ instead of
$\overline{\alpha}$, if no confusion is likely. The following notations for arbitrary terms $t$ and $s$ in the language of
arithmetic will be used later.
\begin{itemize}
\item $\overline{K}(t):=\bigvee_{\alpha\in K}(t=\overline{\alpha})$,
\item $t\prec s:=\bigvee_{\alpha\lneqq\beta}(t=\overline{\alpha}\wedge
s=\overline{\beta})$, $t\preccurlyeq
s:=\bigvee_{\alpha\leq\beta}(t=\overline{\alpha}\wedge
s=\overline{\beta})$,
\item $t\;\Rbar\; s:=\bigvee_{\alpha\R\beta}(t=\alphabar\wedge s=\betabar)$,
\item $\varphi_{_\alpha}:=\bigwedge_{B\in {\sf Sub}(\Gamma), B\in{\sf TNNIL},\alpha\Vdash \Box B}B$. 
\end{itemize}
In the following definition, $L_\theta=y$ as the arithmetical formula equivalent to  ``The limit of the function $F_\theta$ 
is equal to $y$". Similarly, define $\alpha\prec L_\theta$ and so on.
\begin{definition}
Let $\theta(z)$ be a  $\Sigma_1$-formula  in the language of arithmetic. Then
\begin{itemize}
\item $L_\theta=y$ is a shorthand for $\exists{u}\forall{z}(\theta(u*z)\to \hat{z}=y)$ in which $\hat{z}$ is the 
final element of the sequence with the code $z$,
\item For each $\alpha\in K$, $\alpha\preccurlyeq L_\theta$, $\alpha\prec
L_\theta$ and $\alpha\,\R\, L_\theta$ are shorthands for
$\bigvee_{\alpha\leq\beta}\exists{x}(\theta(x)\wedge \hat{x}=\beta)$,
$\bigvee_{\alpha\lneqq\beta}\exists{x}(\theta(x)\wedge \hat{x}=\beta)$ and
$\bigvee_{\alpha\R\beta}\exists{x}(\theta(x)\wedge \hat{x}=\beta)$, respectively,
\item The arithmetical substitution $\sigma$ is defined on propositional
variable $p$ by
$$\sigma(p):=\bigvee_{\beta\Vdash{p}}\beta\preccurlyeq L_\theta.$$
and finally, we extend $\sigma$ to all propositions by
interpreting $\Box$ as provability in \HA, i.e.,  
$\sigma_\theta:=\sigma_{_{\sf HA}}$, in which $\sigma_{_{\sf  HA}}$ is defined from $\sigma$ as 
in \Cref{Definition-Arithmetical substitutions},
\item Let $g$ be a recursive function with $\theta_g(z)$ as the formula $\exists{x}(z=\langle g(0),\ldots,g(x)\rangle)$. 
We define $L_g=y$, $L_g\succ \alpha$, $L_g\succeq \alpha$, $\alpha\,\R\, L_g$
and $\sigma_\tinysub{g}$ to be $L_{\theta_g}=y$, $L_{\theta_g}\succ \alpha$,
$L_{\theta_g}\succeq \alpha$, $\alpha\,\R\, L_{\theta_g}$ and $\sigma_{_{\theta_g}}$, respectively.
\end{itemize}
\end{definition}

Following Berarducci (\cite{Berarducci}), we define a primitive
recursive function as follows:
$$r_{_{\theta}}(\bar{\alpha},x)={\sf min}\left(\{ k \mid \exists{u\leq x}{\sf Proof}_\tinysub{{\sf PA}_k}
(u,\ulcorner\neg(L=\alpha\wedge\Box\sigma_{_\theta}(\varphi_{_\alpha}))\urcorner)\}\cup\{x+1\}\right)$$
Note that $r_{_\theta}(x,y)$ depends also on a $\Sigma_1$-formula which is 
appeared in the subscript of $\sigma$. We also should note that 
$L=\alpha$ is defined in reference to $\theta$ as well.
 We may omit subscripts of the  interpretation $\sigma_{_\theta}$ and the 
function $r_{_\theta}$ when no confusion is likely.

A variant of this function  was first appeared in
\cite{Berarducci}, to define Solovay functions for characterizing
interpretability logic of $\PA$. It is easy to observe that
$r(\alpha,x)$ is always equal or less than $x+1$, and
$r(\alpha,x)\leq x$ iff
$$\exists{y\leq x}{\sf Proof}_\tinysub{\sf PA}(y,\ulcorner\neg(L=\alpha\wedge\Box\sigma_{_{F}}(\varphi_{_\alpha}))\urcorner)$$

Now we are in a position to define the Solovay-like function for
$\mathcal{K}$. Informally speaking, ${F:\mathbb{N}\ra K}$ is
defined in such a way that fulfils  the following conditions.
$F(0):=\alpha_0$, and
\begin{equation}\label{SolovayFunction}
F(x+1):=
\begin{cases}
\beta& \text{  if}\  (x+1)_0=\langle1,\beta\rangle, F(x)R\beta \ \text{and} \
r_{_F}(\beta,x+1)\leq x+1, \\
\gamma & \text{  if } (x+1)_0=\langle2,\gamma\rangle, \neg F(x)R\gamma \text{ and } F(x)\leq\gamma \text{ and }\\ &\ \  \ \  r_{_F}(\gamma,x+1)<r_{_F}(F(x),x+1)\ \text{and}
 \quad F(r_{_F}(\gamma,x+1))\,\R\,\gamma,\\
F(x)& \text{  otherwise.}
\end{cases}
\end{equation}
As it is clear from the definition, $F$ is used in its own definition, i.e. 
we are in a loop. This will be overcome by the 
Diagonalization lemma. To be able to define $F$, we first define
 $\theta(z)$ and then define $F(x)=y$ (the graph of the function $F$) as 
 $$\exists{z}({\sf lth}(z)=x+1\wedge\theta(z)\wedge (z)_{x}=y)$$

By Diagonalization lemma  (\Cref{Lemma-diagonalization lemma}), we
find a $\Delta_0$ formula $\theta(y)$ such that

\begin{equation}\label{Eq1}
\HA_0\vdash\theta(y)\lr ({\sf lth}(y)\geq 1\wedge (y)_0=\overline{\alpha_0}\wedge\forall{x< {\sf lth}(y)}(x\neq0\ra \chi(x,y)))
\end{equation}

in which $\chi(x,y)$ is defined as disjunction of the following
three formulae:
\begin{align}
\nonumber &\chi_1:= \bigvee_{\beta\in K} [(x)_0=\langle1,\bar{\beta}\rangle
\wedge (y)_x=\bar{\beta}\wedge (y)_{x\dot{-}1}\, \bar{\R}\,\bar{\beta} \wedge {r}(\bar{\beta},x)\leq x]\\
\nonumber & \chi_2:= \bigvee_{\beta\in K} [(x)_0=\langle2,\bar{\beta}\rangle\wedge (y)_x=\bar{\beta}\wedge  \neg (y)_{x\dot{-}1}\,\bar{\R}\,\bar{\beta}\wedge (y)_{x\dot{-}1}\preccurlyeq \bar{\beta} \wedge {r}((y)_x,x)< {r}((y)_{x\dot{-}1},x))\\ \nonumber & \ \ \ \ \ \wedge (y)_{{r}((y)_x,x)}\,\bar{\R}\, (y)_x)]\\
\nonumber &\chi_3:= [(y)_x=(y)_{x\dot{-}1}]
\wedge
\bigwedge_{\beta\in K}\neg[ (x)_0=\langle1,\bar{\beta}\rangle\wedge  
(y)_{x\dot{-}1}\,\bar{\R}\, \bar{\beta} 
\wedge {r}(\bar{\beta},x)\leq x] \wedge\\
\nonumber & \bigwedge_{\beta\in K}\neg[(y)_{x\dot{-}1}\preccurlyeq \bar{\beta}
\wedge \neg (y)_{x\dot{-}1}\,\bar{\R}\,\bar{\beta} 
\wedge  (x)_0=\langle2,\bar{\beta}\rangle \wedge {r}(\bar{\beta},x)<{r}((y)_{x\dot{-}1},x)
\wedge (y)_{{r}(\bar{\beta},x)}\,\bar{\R}\,\bar{\beta})]
\end{align}
In the above formulae, $r(x,y)$ is $r_{_\theta}(x,y)$. 
Now we show that a provably total recursive function $F$ can be defined from $\theta(y)$.

\begin{lemma}\label{Lemma-1st Solovay}
The formula $\theta$ is $\Delta_0$ and
\begin{enumerate}
\item \label{1Lemma-1st Solovay}$\HA_0\vdash ({\sf lth}(y_1)\neq 0\wedge \theta(y_1*y_2))\to\theta(y_1)$,
\item \label{2Lemma-1st Solovay}$\HA_0\vdash (\theta(y_1)\wedge\theta(y_2)\wedge{\sf lth}(y_1)={\sf lth}(y_2))\ra y_1=y_2 $,
\item \label{3Lemma-1st Solovay}$\HA_0\vdash\forall{x}\exists{y}({\sf lth}(y)=x+1 \wedge \theta(y))$.
\end{enumerate}
\end{lemma}
\begin{proof}
It is not difficult to observe that the first item holds by definition of $\theta$ in \cref{Eq1}.
To prove the other items, it is enough to show
$\HA_0\vdash\forall{x}\exists!{y}({\sf lth}(y)=x+1\wedge\theta(y))$, in which $!\exists$, as usual, is  
the {\em uniqueness} existential quantifier. This
can be simply  done by induction on $x$. 
\end{proof}

Now, let us define 
$\phi(x,y):=\exists{z}(\theta(z)\wedge{\sf lth}(z)=x+1\wedge \hat{z}=y)$. 
Note that $\phi(x,y)$ is actually a 
$\Delta_0$ formula. The reason is the following.  
we can bound existential quantifier by the primitive recursive
function $h(z)$ with the following primitive recursive definition:
\begin{itemize}
 \item $h(0):=\langle k\rangle$, in which $k$ is the number of nodes of Kripke model,
 \item $h(z+1):=h(z)*\langle k\rangle$.
\end{itemize}
Hence $\HA_0\vdash\phi(x,y)\lr \exists{z\leq h(z)}[{\sf lth}(z)=x+1\wedge (z)_{x+1}=y\wedge\phi(z)]$.

\begin{notation}
{\em The above lemma (\Cref{Lemma-1st Solovay}) says that $\phi(x,y)$ is
the graph of a  $\Delta_0$- function $F$. In the rest of the
paper, we use $F$ as a function symbol with the graph $\phi(x,y)$.
We use $\sigma$ and $L$ instead of $\sigma_\tinysub{\theta}$ and $L_\theta$,
respectively. For simplicity of notations, when we work in the first-order language of arithmetic, instead of $\sigma_{_{\sf HA}}(B)$,
we may use the notation $B$. For instance assume that $p$ is an atomic variable in the propositional language.
When we write down the formula $\HA\vdash \Box (\Box p\to p)\to\Box p$, 
we actually mean $\HA\vdash \sigma_{_{\sf HA}}(\Box (\Box p\to p)\to\Box p)$. This abuse of notations, 
wipes out many unimportant symbols from the rest of \Cref{sec-transforming}.}
\end{notation}

\noindent One can observe that the function $F$ fulfils the recursive conditions of \cref{SolovayFunction}.

\subsection{Elementary properties of the Solovay function}
In this part, we will see some elementary properties of the function $F$. 

\begin{lemma}\label{Lemma-Properties of Solovay's Function}
The function $F$ has the following properties:
\begin{enumerate}
\item \label{1Lemma-Properties of Solovay's Function}$\HA_0\vdash\forall{x,y}({F}(x)\preccurlyeq {F}(x+y))$,
\item \label{2Lemma-Properties of Solovay's Function}For any $\alpha\in K$,
$\PA\vdash\exists{x}{F}(x)=\alpha\ra\bigvee_{\alpha\leq\beta}L=\beta$,
\item \label{3Lemma-Properties of Solovay's Function}For any $\alpha\in K$, $\PA\vdash \alpha\prec
L\leftrightarrow\bigvee_{\alpha< \beta}L=\beta$ and $\PA\vdash
\alpha\,\bar{\R}\,L\leftrightarrow\bigvee_{\alpha\R\beta}L=\beta$.
\end{enumerate}
\end{lemma}
\begin{proof}
\begin{enumerate}[leftmargin=*]
\item By recursive definition of $F$, 
$\HA_0\vdash{F}(x)\preccurlyeq {F}(x+1)$. Let
${A(y):=\bar{F}(x)\preccurlyeq \bar{F}(x+y)}$ and use induction on $y$ in $A(y)$.
\item We prove this fact by induction (in meta-language) on
the tree $(K,\lneqq)$ with reverse order. Suppose that for all
$\beta\gneqq\alpha$, we have
$\PA\vdash\exists{x}{F}(x)=\beta\ra\bigvee_{\beta\leq \gamma}L=\gamma$.
Then 
$$\PA\vdash{F}(x)=\alpha\ra(\forall{y\geq
x}{F}(y)=\alpha\vee\exists{y\geq x}{F}(y)\neq\alpha)$$
 By
part 1 and definition of $L=\alpha$, we get $\PA\vdash
{F}(x)=\alpha\ra(L=\alpha\vee\exists{y\geq
x}(\alpha\prec {F}(y)))$. Now  induction hypothesis implies
$\PA\vdash{F}(x)=\alpha\ra\bigvee_{\beta\geq\alpha}L=\beta$.
\item Proof of this part is an immediate consequence of part 2 and perfectness of $\mathcal{K}$.
\end{enumerate}
\end{proof}

\begin{lemma}\label{Lemma-Properties of Solovay's Function.4}
 For any $\alpha, \beta\in K$ with $\alpha\,\R\,\beta$, $\HA_0\vdash L=\alpha\ra\neg\Box^+
\neg(L=\beta\wedge\Box\varphi_{_\beta}\emptycommand)$.
\end{lemma}
\begin{proof}
We argue inside $\HA_0$. Assume $L=\alpha$ and ${\sf Proof}_\tinysub{\sf PA}
(x,\ulcorner\neg(L=\beta\wedge\Box\varphi_{_\beta}\emptycommand)\urcorner)$.
Let $y>x$ such that $(y+1)_0=\langle2,\beta\rangle$. Then because
$L=\alpha$, we have $F(y)=\alpha$. On the other hand, by
recursive definition of $F$, $F(y+1):=\beta$, a contradiction.
\end{proof}

\begin{lemma}\label{Lemma-1.5st Properties of Solovay's Function}
For any $\delta, \alpha, \beta\in K$ with $\delta\,\R\,\alpha\leq
\beta $,
$\HA_0+L=\delta\vdash(L=\alpha\wedge\Box\varphi\emptycommand_\alpha)\rhd
(L=\beta\wedge\Box\varphi\emptycommand_\beta)$.
\end{lemma}
\begin{proof}
If $\alpha\,\R\,\beta$, by  \Cref{Lemma-Properties of Solovay's
Function.4}, $\HA_0\vdash
L=\alpha\ra\neg\Box^+\neg(L=\beta\wedge\Box\varphi_{_\beta}\emptycommand)$.
So  
$${\mathbb{N}\models
\Box^+(L=\alpha\ra\neg\Box^+\neg(L=\beta\wedge\Box\varphi_{_\beta}\emptycommand))}$$
and hence by  \Cref{Lemma-bounded Sigma completeness} ($\Sigma_1$-completeness of $\HA_0$), we can
deduce $\HA_0\vdash
L=\alpha\rhd(L=\beta\wedge\Box\varphi_{_\beta}\emptycommand)$. So
assume $\alpha\;\NR\;\beta$ and $\alpha\neq\beta$. By definition of $A\rhd B$, we must show
$$\HA_0+L=\delta\vdash
\forall{x}\Box^+[(L=\alpha\wedge\Box\varphi\emptycommand_\alpha)\ra
\neg\Box^+_x\neg(L=\beta\wedge\Box\varphi\emptycommand_\beta)].$$
We work inside $\HA_0$. Assume $L=\delta$ and fix some large enough $x$ such that $F(x)=\delta$. Then
for each $u\leq x$, we have $F(u)\,\R\,\beta$. Now work in the scope
of $\Box^+$. By $\Sigma$-completeness of $\PA$, we have
$\forall{u}\leq x F(u)\,\R\,\beta$. Assume $L=\alpha$,
$\Box\varphi_{_\alpha}\emptycommand$ and 
$\Box^+_x\neg(L=\beta\wedge\Box\varphi_{_\beta}\emptycommand)$. We
should deduce $\bot$. By
$\Box^+_x\neg(L=\beta\wedge\Box\varphi_{_\beta}\emptycommand)$,
for sufficiently large $y$ (larger than
$\langle2,\beta\rangle*z$, in which $z$ is a proof code in
$\PA_x$ for $\neg(L=\beta\wedge\Box\varphi_{_\beta}\emptycommand)$),
we have $r(\beta,y)\leq x$. If $r(\alpha,y)\leq r(\beta,y)$, then
$\Box^+_x\neg(L=\alpha\wedge\Box\varphi_{_\alpha}\emptycommand)$,
and hence by  \Cref{Lemma-Reflection}, we have
$\neg(L=\alpha\wedge\Box\varphi_{_\alpha}\emptycommand)$, a
contradiction. If $r(\alpha,y)> r(\beta,y)$,  since
$r(\beta,y)\leq x$, then $F(r(\beta,y))\,\R\,\beta$. So by
recursive definition of $F$, there exists some $z\geq y$ such
that $F(z)=\beta$, contradicting $L=\alpha$.
\end{proof}

\subsection{Deciding the boxed formulas}\label{subsection-proofofmaintheorem}
In this subsection, we will show that  $\HA+L=\alpha+\Box\varphi_{_\alpha}\emptycommand$ can decide boxed propositions in 
${\sf Sub}(\Gamma)$. More precisely,  for all $\Box B\in {\sf Sub}(\Gamma)$ and $\alpha\in K$, 
$$
\begin{cases}
\HA+L=\alpha+\Box\varphi_{_\alpha}\emptycommand\vdash \Box B\emptycommand \ \ \ & \text{ if } \alpha\Vdash \Box B \\
\HA+L=\alpha+\Box\varphi_{_\alpha}\emptycommand\vdash \neg\Box B\emptycommand \ \ \ & \text{ if } \alpha\nVdash \Box B
\end{cases}
$$
Note that by definition of $\varphi_{_\alpha}$, if $\alpha\Vdash \Box B$,  
then $B$ is a conjunct of $\varphi_{_\alpha}$. Hence in case $\alpha\Vdash \Box B$, 
we obviously have $\HA+\Box\varphi_{_\alpha}\emptycommand\vdash\Box B\emptycommand$. 
Moreover we will show in  \Cref{sec-54} (\Cref{corollar-4st&2st}) that  
 $\HA\vdash L=\alpha \to \Box \varphi_{_\alpha}\emptycommand$ for $\alpha\in K_0$, 
and then the following improvement of the above
equation  holds: 
$$
\begin{cases}
\HA+L=\alpha \vdash \Box B\emptycommand \ \ \ & \text{ if } \alpha\Vdash \Box B \\
\HA+L=\alpha \vdash \neg\Box B\emptycommand \ \ \ & \text{ if } \alpha\nVdash \Box B
\end{cases}
$$
\begin{lemma}\label{Lemma-1.6st Properties of Solovay Function}
Let $B\in {\sf Sub}(\Gamma)$ be such that all occurrences of  $\ra$ in $B$ are in the scope of some $\Box$
{\em ($B\in{\sf NOI}$)}, $\alpha\in K$ and $\alpha\Vdash B$. Then 
$\HA_0\vdash(L=\alpha\wedge\Box\varphi_{_\alpha}\emptycommand)\ra B\emptycommand$. Moreover, this argument is formalizable and 
provable in $\HA_0$, i.e. $\HA_0\vdash\Box_0((L=\alpha\wedge\Box\varphi_{_\alpha}\emptycommand)\ra B\emptycommand)$.
\end{lemma}
\begin{proof}
One can prove $\HA_0\vdash
(L=\alpha\wedge\Box\varphi_{_\alpha}\emptycommand)\ra B\emptycommand$, by
induction on the complexity of $B$. Then by 
\Cref{Lemma-bounded Sigma completeness}, we derive its formalized form in $\HA_0$.
\end{proof}

\begin{notation}
We say that $\alpha\mrefute A$ if $\alpha\nVdash A$ and for all $\beta\gneqq\alpha$ we have  $\beta\Vdash A$.  
\end{notation}
\noindent We have the following observations:
\begin{itemize}
\item $\alpha\mrefute B\to C$ iff ``$\alpha\Vdash B$ and $\alpha\mrefute C$'',
\item $\alpha\mrefute B\vee C$ implies ``$\alpha\nVdash B$ and $\alpha\nVdash C$'',
\item $\alpha\mrefute B\wedge C$ iff  ``$\alpha\mrefute B$ or $\alpha\mrefute C$''.
\end{itemize}

\noindent Let $A$ be a $\TNNIL$-formula such that $\alpha\mrefute A$.
In \Cref{Lemma-1.7st Properties of Solovay Function} and \Cref{Lemma-1.84st Properties of Solovay Function}, we need to put $\Box_x$ before all occurrences of subformulas 
$B$  in the right of $\ra$, when it is not the case that $\alpha\mrefute B$. This is the content of the following definition.
\begin{definition}
\em 
Let $A$ be a modal proposition, $\alpha\in K$ and $x$ be a variable.
 We define the first-order sentence $d(A,\alpha,x)$, by induction on $A$.
If this is not the case that $\alpha\mrefute A$, then we define
$d(A,\alpha,x):=\Box_x \sigma_{_{\sf HA}}(A)$, and if $\alpha\mrefute A$,
we define the formula $d(A,\alpha,x)$ by cases:
  \begin{itemize}
  \item $A$ is atomic or boxed. $d(A,\alpha,x):=\sigma_{_{\sf HA}}(A)$,
  \item $A=B\to C$. Define $d(A,\alpha,x)$ by cases. If $B\not\in{\sf NOI}$, then let $d(A,\alpha,x):=\sigma_{_{\sf HA}}(A)$,
   otherwise let   $d(A,\alpha,x):=\sigma_{_{\sf HA}}(B)\to d(C,\alpha,x)$,
  \item $A=B\wedge C$. If $\alpha\mrefute B$ then $d(A,\alpha,x):=d(B,\alpha,x)$, else $d(A,\alpha,x):=d(C,\alpha,x)$,
  \item $A=B\vee C$. $d(A,\alpha,x):=d(B,\alpha,x)\vee d(C,\alpha,x)$.
\end{itemize}   
\end{definition}
\noindent In the following lemma, we use definition of $\sigma_{_l}(A,x)$ from \Cref{Sec-ExLePr}:
\begin{lemma}\label{Lemma-1.69st Properties of Solovay Function}
 Let $A$ be a modal proposition, $\alpha\in K$ such that $\alpha\mrefute A$. Then 
 $$\HA_0\vdash  \sigma_{_l}(A,x)\to d(A,\alpha,x)$$
\end{lemma} 
\begin{proof}
Use induction on $A$. 
\end{proof}
\begin{lemma}\label{Lemma-1.7st Properties of Solovay Function}
Let $A$ be a modal proposition. Then there exists some provably {\em (}in
$\HA${\em)} total recursive function $g_{_A}$ such that   for any $\alpha\in K$ with $\alpha\mrefute A$  we have
$$\HA\vdash\Box_x A\emptycommand \ra\Box_{g_{_A}(x)}d(A,\alpha,g_{_A}(x))$$
\end{lemma}
\begin{proof}
Use   \Cref{Lemma-sigma_l translation} and \Cref{Lemma-1.69st Properties of Solovay Function}.
\end{proof}

\noindent 
Let $\HA\vdash A$, for arbitrary $A$ in the language of arithmetic. Then  by the compactness theorem, one could deduce 
that $\HA_n\vdash A$ for some $n\in \omega$. In the following definition of the $n_1$ and $n_2$, we make use of this fact.
\noindent Define $m\in\omega$ as the maximum of the following $n_i$'s:
\begin{itemize}
\item $n_1$. By  \Cref{Lemma-1.7st Properties of Solovay Function} and the compactness theorem,
we can find some $n_1$ such that for each $B\in {\sf Sub}(\Gamma)$, $g_{_B}$ is provably total in
$\HA_{n_1}$.
\item $n_2$. For each $\alpha\in K$ and $B\in {\sf Sub}(\Gamma)$ such that $\alpha\mrefute
B$, by  \Cref{Lemma-1.7st Properties of Solovay Function} and the compactness theorem,
there exists some $n$ such that $\HA_n$ proves the desired
sentence of the Lemma. Let $n_2$ be the maximum of such $n$.
\item $n_3$. By  \Cref{Lemma-Reflection refinement}, for each
$\alpha\in K$, there exists some provably (in $\HA$) total
function $h_\alpha$, such that $h_\alpha(x)\geq x$ and $\HA\vdash
\Box_{h_\alpha(x)}(\Box_x\neg(L=\alpha\wedge\Box\varphi_{_\alpha}\emptycommand)\ra
\neg(L=\alpha\wedge\Box\varphi_{_\alpha}\emptycommand))$. Hence by the compactness theorem, there exists
some $n_\alpha\in\omega$ such that $h_\alpha$ is provably total
in $\HA_{n_\alpha}$ and 
$${\HA_{n_\alpha}\vdash
\Box_{h_\alpha(x)}(\Box_x\neg(L=\alpha\wedge\Box\varphi_{_\alpha}\emptycommand)\ra
\neg(L=\alpha\wedge\Box\varphi_{_\alpha}\emptycommand))}$$ Let
$n_3:=\text{max}\{n_\alpha | \alpha\in K\}$.
\end{itemize}
Then define $\hat{g}_{_B}(x)$ as the maximum of $g_{_B}(x)$, $m$ and $x$. Assume some $B\in {\sf Sub}(\Gamma)$. We define the provably (in $\HA_m$)
total recursive function $f_{_B}$,  by induction on
the complexity of $B$:
$$f_{_B}(x):=
\begin{cases}
\text{max}(X)&  \text{\ \ if }X=\{h_\alpha(f_{_C}(\hat{g}_{_B}(x)))\ |\  
C\in {\sf Sub}(B) , C\neq B, \alpha\in K \}\neq\emptyset\\
\hat{g}_{_B}(x) &  \text{\ \ else}
\end{cases}
$$
where $h_\alpha$ is as we stated in definition of
$n_3$. From the above definition, one can observe that for each
atomic $C\in {\sf Sub}(\Gamma)$,  the set $X$ is empty. Hence we have $f_{_C}(x)=x$. 
Since each non-atomic
formula $B$ has some atomic sub-formula $C$, one can deduce that
$f_{_B}(x)\geq \hat{g}_{_B}(x)\geq x,m$.
Moreover,
all of the above functions are provably total in $\HA_m$.

\begin{lemma}\label{Lemma-1.84st Properties of Solovay Function}
Let $B,E\in {\sf Sub}(\Gamma)\cap\TNNIL$ and  $\beta\in K$,  such that $\beta\mrefute B$, $\beta\mrefute E$ and 
${B\in{\sf Sub}(E)}$. Then

\begin{equation}\label{Eq-7}
\HA_m\vdash 
[
F(f_{_E}(x))\,\R\,\beta\wedge\Box_x E\emptycommand
]
\to \Box_{f_{_E}(x)}
\left(  
(L=\beta\wedge\Box\varphi_{_\beta}\emptycommand)\to\neg d(B,\beta,\hat{g}_{_E}(x))
\right)
\end{equation}
\end{lemma}
\begin{proof}
We prove \cref{Eq-7} by induction on the complexity of $B$. As induction hypothesis, assume that for any 
sub-formula $C$ of $B$ ($C\neq B$) and any $E'\in {\sf Sub}(\Gamma)\cap \TNNIL$ and $\gamma\in K$, such that $C\in{\sf Sub}(E')$ and  $\gamma\mrefute C,E'$,
 we have
\begin{equation*}
\HA_m\vdash 
\left(
F(f_{_{E'}}(x))\,\R\,\gamma\wedge\Box_x E'\emptycommand
\right) 
\to \Box_{f_{_{E'}}(x)}
\left(  
(L=\gamma\wedge\Box\varphi_{_\gamma}\emptycommand)\to\neg d(C,\gamma,\hat{g}_{_{E'}}(x))
\right)
\end{equation*}
We consider different cases.
\begin{itemize}[leftmargin=*]
\item $B$ is atomic. Then  $d(B,\beta,\hat{g}_{_B}(x))=\sigma(B)$
 and the desired result holds by definition of the substitution $\sigma$ and $\beta\nVdash B$ 
 and  also by 
 \Cref{Lemma-Properties of Solovay's Function}  \cref{1Lemma-Properties of Solovay's Function}, 
 \item $B=\Box C$. Then $d(B,\beta,\hat{g}_{_B}(x))=\sigma_{_{\sf HA}}(B)=\Box \sigma_{_{\sf HA}}(C)$. 
Since $\beta\mrefute\Box C$, there exists some 
$\gamma$ such that $\beta\,\R\,\gamma\mrefute C$. Then, by induction hypothesis,  
$$\HA_m\vdash\left(F\left(f_{_C}\left(x\right)\right)\,\R\,\gamma \wedge\Box_x C\emptycommand\right)\ra 
\Box_{f_{_C}(x)} \left( (L=\gamma\wedge\Box\varphi_{_\gamma}\emptycommand)
\to \neg d(C,\gamma,\hat{g}_{_C}(x))
 \right)$$
 By \Cref{Lemma-1.7st Properties of Solovay Function}, we have 
 $\HA_m\vdash \Box_x C\emptycommand\to\Box_{f_{_C}(x)}d(C,\gamma,\hat{g}_{_C}(x))$. 
Hence 
$$\HA_m\vdash{\left(L=\beta \wedge\Box C\emptycommand\right)}\ra
 {\Box \neg\left(L=\gamma\wedge\Box\varphi_{_\gamma}\emptycommand\right)}$$
 By \Cref{Lemma-Properties of Solovay's Function.4}, we have 
 $\HA_m\vdash \neg \left(L=\beta \wedge\Box C\emptycommand\right)$. 
Hence by \Cref{Lemma-bounded Sigma completeness},  
$\HA_0 \vdash \Box_m \neg \left(L=\beta \wedge\Box C\emptycommand\right)$. Since $f_{_B}(x)\geq m$, we have 
$\HA_m\vdash \Box_{f_{_B}(x)}  
\left((L=\beta \wedge\Box \varphi_{_\beta}\emptycommand)\to \neg d(B,\beta,\hat{g}_{_B}(x))\right)$, 
which implies \cref{Eq-7}.
\item $B=C\to D$. In this case  $\beta\Vdash C\in{\sf NOI}$,  $\beta\mrefute D$ 
 and 
$d(B,\beta,\hat{g}_{_B}(x))=\sigma_{_{\sf HA}}(C)\to d(D,\beta,\hat{g}_{_B}(x))$. 
Hence, by induction hypothesis,  
\begin{equation*}
\HA_m\vdash\left(F\left(f_{_E}\left(x\right)\right)\,\R\,\beta \wedge\Box_x E\emptycommand\right)\ra  
 \Box_{f_{_E}(x)}  	\left( (L=\beta\wedge\Box\varphi_{_\beta}\emptycommand) \to 
\neg d(D,\beta,\hat{g}_{_E}(x)) 
 \right)
\end{equation*}
Then  by \Cref{Lemma-1.6st Properties of Solovay Function},  
 $$ \HA_m\vdash\left(F\left(f_{_E}\left(x\right)\right)\,\R\,\beta \wedge\Box_x E\emptycommand\right)
 \ra 
 \Box_{f_{_E}(x)}\left( 
 (L=\beta\wedge\Box\varphi_{_\beta}\emptycommand)\to\neg(C\emptycommand\to d(D,\beta,\hat{g}_{_E}(x)) \right)$$
  \item $B=C\wedge D$. Since $\beta\mrefute B$, either $\beta\mrefute C$ or $\beta\mrefute D$ holds. 
 We only treat the case that ${\beta\mrefute C}$. The other case is similar. Assume that $\beta\mrefute C$.
Then by definition, $d(B,\beta,y)=d(C,\beta,y)$. 
 Now the induction hypothesis for $C$, directly implies the desired result, i.e. \cref{Eq-7}.
 \item $B=C\vee D$. This case is the interesting one. We have 4 sub-cases:  (1) $\beta\mrefute C$ and $\beta\mrefute D$, 
 (2)\nolinebreak \ not $\beta\mrefute C$ and $\beta\mrefute D$, (3) $\beta\mrefute C$ and not $\beta\mrefute D$, 
 (4) not $\beta\mrefute C$ and not $\beta\mrefute D$. We only treat the case (3) here. Other cases can be treated 
 similarly. Assume that the case (3) occurs. By definition,  
 ${d(B,\beta,\hat{g}_{_E}(x))}={d(C,\beta,\hat{g}_{_E}(x))\vee\Box_{\hat{g}_{_E}(x)}D\emptycommand}$.
From the induction hypothesis for $C$,  
\begin{equation}\label{Eq-17}
\HA_m\vdash \left( F(f_{_E}(x))\,\R\,\beta\wedge\Box_x E\emptycommand\right) \to 
 \Box_{f_{_E}(x)} \left((L=\beta\wedge\Box\varphi_{_\beta}\emptycommand)\to\neg d(C,\beta,\hat{g}_{_E}(x)) \right)
\end{equation}
So it is enough to show that 
\begin{equation}\label{Eq-18}
\HA_m\vdash \left( F(f_{_E}(x))\,\R\,\beta\wedge\Box_x E\emptycommand\right) \to 
 \Box_{f_{_E}(x)} \left((L=\beta\wedge\Box\varphi_{_\beta}\emptycommand)\to\neg \Box_{\hat{g}_{_E}(x)} D\emptycommand \right)
\end{equation}
Since $\beta\nVdash D$ and not $\beta\mrefute D$, there exists some $\gamma\gneqq\beta$ such that $\gamma\mrefute D$.
If $\beta\,\R\,\gamma$, then we can repeat the reasoning as in the  case $B=\Box C$. So assume that 
$\beta\NR\gamma$. By the induction hypothesis for $D$ and $\gamma$, we have 
\begin{equation*}
\HA_m\vdash \left( F(f_{_D}(x))\,\R\,\gamma\wedge\Box_x D\emptycommand\right) \to 
 \Box_{f_{_D}(x)} \left((L=\gamma\wedge\Box \varphi_{_\gamma}\emptycommand)\to\neg d(D,\gamma,\hat{g}_{_D}(x)) \right)
\end{equation*}
On the other hand, by \Cref{Lemma-1.7st Properties of Solovay Function},  we have 
\begin{equation*}
\HA_m\vdash \Box_x D\emptycommand\to \Box_{f_{_D}(x)}d(D,\gamma,\hat{g}_{_D}(x))
\end{equation*}
Hence
\begin{equation}\label{Eq-19}
\HA_m\vdash \left( F(f_{_D}(x))\,\R\,\gamma\wedge\Box_x D\emptycommand\right) \to 
 \Box_{f_{_D}(x)} \neg(L=\gamma\wedge\Box\varphi_{_\gamma}\emptycommand)
\end{equation}
We argue inside $\HA_m$. Assume $F(f_{_E}(x))\,\R\,\beta$ and $\Box_x E\emptycommand$.
Since $f_{_E}(x)\geq f_{_D}(\hat{g}_{_E}(x))$, by the
assumption of $F(f_{_E}(x))\,\R\,\beta$, we have $F(f_{_D}(\hat{g}_{_E}(x)))\,\R\,\gamma$, and by 
\Cref{Lemma-bounded Sigma completeness}, 
we get $\Box_m(F(f_{_D}(\hat{g}_{_E}(x)))\,\R\,\gamma)$. Hence 
if we replace $\hat{g}_{_E}(x)$ for $x$ in \cref{Eq-19},  we may deduce 
\begin{equation}\label{Eq-21}
\Box_m\left(\Box_{\hat{g}_{_E}(x)}D\emptycommand\ra
\Box_{f_{_D}(\hat{g}_{_E}(x))}\neg(L=\gamma\wedge\Box\varphi_{_\gamma}\emptycommand)\right)
\end{equation}
Now we work inside $\Box_{f_{_E}(x)}$. We
have $F(f_{_E}(x))\,\R\,\beta$. Assume
$\Box_{\hat{g}_{_E}(x)}D\emptycommand$ and $L=\beta$ and
$\Box\varphi_{_\beta}\emptycommand$. We should deduce $\bot$. From 
$\Box_{\hat{g}_{_E}(x)}D\emptycommand$ and \cref{Eq-21}, we have 
$\Box_{t(x)}\neg(L=\gamma\wedge\Box\varphi_{_\gamma}\emptycommand)$, in
which $t(x):=f_{_D}(\hat{g}_{_E}(x))$. So there exists some
$y_1$ such that 
${\sf Proof}_{{\sf HA}_{t(x)}}(y_1,\ulcorner \neg(L=\gamma\wedge\varphi_{_\gamma}\emptycommand) \urcorner)$. 
Also by $L=\beta$, there exists some $y_2\geq y_1$ such that
$\forall{z\geq y_2}F(z)=\beta$. Let some $y$ greater than 
$\langle2,\gamma\rangle^*y_2$ and $t(x)$. If $r(\beta,y+1)\leq t(x)$, then  
$\Box_{t(x)}\neg(L=\beta\wedge\Box\varphi_{_\beta}\emptycommand)$. Now,
since $ f_{_E}(x)\geq h_\beta(t(x))$ and we are working in
$\Box_{f_{_E}(x)}$, by  \Cref{Lemma-Reflection refinement}, we have
$\neg(L=\beta\wedge\Box\varphi_{_\beta}\emptycommand)$ and hence 
$\bot$. If $t(x)<r(\beta,y+1)$, since $r(\gamma,y+1)\leq t(x)$,  by recursive definition of $F$,
then  $F(y+1)=\gamma$, which contradicts with $L=\beta$.
\end{itemize}
\end{proof}
\begin{corollary}\label{Lemma-1.9st Properties of Solovay Function}
For each $B\in {\sf Sub}(\Gamma)\cap\TNNIL$ and $\beta\in K$ such that
$\beta\mrefute B$,
$$\HA_m\vdash\left(F\left(f_{_B}\left(x\right)\right)\,\R\,\beta \wedge\Box_x B\emptycommand\right)\ra \Box_{f_{_B}(x)}
\neg\left(L=\beta\wedge\Box\varphi_{_\beta}\emptycommand\right)$$
\end{corollary}
\begin{proof}
Use \Cref{Lemma-1.84st Properties of Solovay Function} for $E=B$. Then by \cref{Eq-7}
and  \Cref{Lemma-1.7st Properties of Solovay Function}, one can deduce the  desired result.
\end{proof}

\begin{theorem}\label{Lemma-2st Properties of Solovay's Function}
For each $B\in {\sf Sub}(\Gamma)\cap\TNNIL$ and $\alpha\in K$ such that $\alpha\nVdash
\Box B$,
$$\HA\vdash L\!=\!\alpha\ra\neg\Box B\emptycommand $$
\end{theorem}
\begin{proof}
From $\alpha\nVdash\Box B$, we conclude that there exists some $\beta\in K$
such that $\alpha\,\R\,\beta$ and $\beta\mrefute B$. Now 
\Cref{Lemma-1.9st Properties of Solovay Function} implies
$\HA\vdash (L=\alpha\wedge\Box B\emptycommand)\ra
\Box\neg(L=\beta\wedge\Box\varphi_{_\beta}\emptycommand)$. 
On the other hand,  by \Cref{Lemma-Properties of Solovay's Function.4},  
$\HA\vdash L=\alpha\to \neg \Box\neg(L=\beta\wedge\Box\varphi_{_\beta}\emptycommand)$.
Hence $\HA\vdash (L=\alpha\wedge\Box B\emptycommand)\ra\bot$, as desired.
\end{proof}

\subsection{The Solovay function is a constant function}\label{sec-54}
In this subsection,  we will show that $L=\alpha_0$ is a true statement in the standard model (\Cref{Lemma-limit is root}).
This fact is necessary for showing that for any $\alpha\in K$, the theory $L=\alpha+\PA$ 
is  consistent. 
\begin{lemma}\label{Lemma-3th Properties of Solovay's Function}
For each $\alpha\lneqq\beta\in K$ with $\alpha\NR\beta$,
$$\HA\vdash\exists{x}F(x)=\alpha\ra\Box^+\neg(L=\beta\wedge\Box\varphi_{_\beta}\emptycommand)$$
\end{lemma}
\begin{proof}
By $\Pi_2$ conservativity of $\PA$ over $\HA$, it is enough to
prove the above assertion in $\PA$ instead of $\HA$. We work
inside $\PA$. Fix some $x$ such that $F(x)=\alpha$. Then for each
$y\leq x$, we have $F(y)\preccurlyeq\alpha$. Now, work inside $\Box^+$.
Assume $L=\beta$ and $\Box\varphi_{_\beta}\emptycommand$. Then there
exists some minimum $z$ such that $F(z+1)=\beta$. So there exists
some $\delta$ such that $F(z)=\delta$. Since $F(x)=\alpha$, we
have  $\beta\gneqq\delta\geq\alpha$. Hence $\delta\NR\beta$.
So by recursive definition of $F$, $r(\beta,z+1)<r(\delta,z+1)$
and $F(r(\beta,z+1))\,\R\,\beta$. Since $\alpha\NR\beta$, we have
 $F(r(\beta,z+1))\precneqq F(x)=\alpha$,
  which implies $r(\beta,z+1)< x$. Since $x\leq z$, we have $r(\beta,z+1)<z$   and hence
$\Box_x\neg(L=\beta\wedge\Box\varphi_{_\beta}\emptycommand)$. Thus
by   \Cref{Lemma-Reflection},
$\neg(L=\beta\wedge\Box\varphi_{_\beta}\emptycommand)$, that is a
contradiction.
\end{proof}
\begin{lemma}\label{Lemma-1st}
For any $\beta\in K$ and $B\in {\sf Sub}(\Gamma)\cap\TNNIL$,
\begin{itemize}
\item if $\beta\Vdash B$, then $\HA\vdash (L=\beta\wedge\Box\varphi_{_\beta}\emptycommand)\ra
B\emptycommand$,
\item if $\beta\nVdash B$ and any occurrence of $\ra$ in $B$ is in the scope of some $\Box$
{\em($B\in{\sf NOI}$)}, then ${\HA\vdash(L=\beta
\wedge\Box\varphi_{_\beta}\emptycommand)\ra\neg B\emptycommand}$.
\end{itemize}
\end{lemma}
\begin{proof}
We prove both items by induction on the complexity of $B$.
\begin{itemize}
\item $B$ is atomic. Then, by definition of the substitution $\sigma$,
$\HA\vdash B\emptycommand\lr\bigvee_{\gamma\Vdash
B}\exists{x}{F}(x)=\gamma$. If $\beta\Vdash B$, then
$\HA\vdash L=\beta\ra B\emptycommand$. If $\beta\nVdash B$, then for
each $\gamma\Vdash B$, we have $\gamma\nleq\beta$, and hence by
 \Cref{Lemma-Properties of Solovay's Function} \cref{1Lemma-Properties of Solovay's Function}, $\HA\vdash
L=\beta\ra\neg\exists{x}{F}(x)=\gamma$ . Hence $\HA\vdash L=\beta\ra\neg B\emptycommand$.
\item $B$ is a conjunction or disjunction. We have the desired conclusions by the induction hypotheses.
\item $B=\Box C$. First assume $\beta\Vdash\Box C$. Then, by definition of
$\varphi_{_\beta}\emptycommand$, \ $C\emptycommand$ is a conjunct of
$\varphi_{_\beta}\emptycommand$, and then
$\HA\vdash(L=\beta\wedge\Box\varphi_{_\beta}\emptycommand)\ra
B\emptycommand$. For the other side, assume $\beta\nVdash\Box C$.
Then  \Cref{Lemma-2st Properties of Solovay's
Function} implies  ${\HA\vdash  L=\beta\ra\neg\Box C\emptycommand}$.
\item $B=C\ra D$. Since $B$ is $\TNNIL$, we have $C\in{\sf NOI}$. 
First assume that $\beta\Vdash C\ra D$.
If $\beta\Vdash C$, then $\beta\Vdash D$, and
hence by the induction hypothesis,
$$\HA\vdash(L=\beta\wedge\Box\varphi_{_\beta}\emptycommand)\ra(C\emptycommand\ra
D\emptycommand).$$ If $\beta\nVdash C$, then again by the induction
hypothesis,
$\HA\vdash(L=\beta\wedge\Box\varphi_{_\beta}\emptycommand)\ra\neg
C\emptycommand$, and hence
$\HA\vdash(L=\beta\wedge\Box\varphi_{_\beta}\emptycommand)\ra(C\emptycommand\ra
D\emptycommand)$.
\end{itemize}
\end{proof}

\begin{lemma}\label{Lemma-1.5st}
Let $\alpha \in K$ and for each $\beta\geq\alpha$, we have
$\HA\vdash\beta\,\R\, L\ra \varphi_{_\beta}\emptycommand$. Then for each
$\beta\geq\alpha$ and $\gamma\gneqq\beta$ such that $\beta\NR\gamma$, we have
$$\HA\vdash\exists{x}F(x)=\beta\ra\Box^+L\not=\gamma$$
\end{lemma}
\begin{proof}
Fix some $\beta\geq\alpha$. We use induction on $\gamma$.
Suppose that for each $ \gamma_0\gneqq\gamma\gneqq\beta$ with
$\beta\NR\gamma_0$, we have
$\HA\vdash{\exists{x}F(x)=\beta}\ra{\Box^+L\not=\gamma_0}$. 
Then 
\begin{align*}
\HA &\vdash \exists{x}F(x)=\beta \ra
\Box^+\neg(L=\gamma\wedge\Box\varphi_{_\gamma}\emptycommand)
 &\text{\Cref{Lemma-3th Properties of Solovay's Function}}&\\
\HA &\vdash {\exists{x}F(x)=\beta} \ra
{\Box^+((\exists{x}F(x)=\gamma\wedge\Box\varphi_{_\gamma}\emptycommand)\ra
L\neq\gamma)}& &\\
\HA&\vdash
\exists{x}F(x)=\beta\ra
\Box^+((\exists{x}F(x)=\gamma\wedge\Box\varphi_{_\gamma}\emptycommand)\ra
\gamma\,\R\, L) & \text{induction hypothesis and neatness}&\\
\HA &\vdash
\exists{x}F(x)=\beta\ra
\Box(\exists{x}F(x)=\gamma\ra(\Box\varphi_{_\gamma}\emptycommand\ra
\varphi_{_\gamma}\emptycommand)) & \text{hypothesis of lemma and 
\Cref{Lemma-Conservativity of HA}}&\\
\HA &\vdash
\exists{x}F(x)=\beta\ra \Box(\exists{x}F(x)=\gamma\ra
\Box\varphi_{_\gamma}\emptycommand) & \text{L\"{o}b's axiom,
$\Sigma_1$-completeness of $\HA$}&
\end{align*}
This in combination with $\HA\vdash
\exists{x}F(x)=\beta\ra
\Box^+\neg(L=\gamma\wedge\Box\varphi_{_\gamma}\emptycommand)$
implies $$\HA\vdash {\exists{x}F(x)=\beta\ra\Box^+L\neq\gamma}$$
\end{proof}

\begin{lemma}\label{Lemma-1.7st}
For any $\gamma\in K_0$, 
$\PA\vdash \exists{x}F(x)=\gamma\ra\Box^+\neg(L=\gamma\wedge\Box\varphi_{_\gamma}\emptycommand)$.
\end{lemma}
\begin{proof}
We work inside $\PA$. Assume $\exists{x}{F}(x)=\gamma$. There
exists a minimum $x_0$ such that $F(x_0)=\gamma$. Then by
recursive definition of $F$, we have $F(x)\prec F(x_0)$ for all
$x< x_0$, and $F(x_0\dot{-}1)=\beta$, and one of the following
cases holds:
\begin{enumerate}
\item $\beta\,\R\,\gamma$ and $r(\gamma,x_0)\leq x_0$, by
definition of $r$, we can deduce 
$$\exists{x\leq x_0}\,{\sf Proof}_{_{\sf PA}}(x,\ulcorner
\neg(L=\gamma\wedge\varphi_{_\gamma}\emptycommand)\urcorner)$$
and then 
$\Box^+\neg(L=\gamma\wedge\Box\varphi_{_\gamma}\emptycommand)$. 

\item $\beta\NR\gamma, \beta\prec\gamma$ and
$r(\gamma,x_0)<r(\beta,x_0)$. Because $r(\beta,x_0)\leq x_0+1$, we
can deduce $r(\gamma,x_0)\leq x_0$. By repeating the above
argument, we get
$\Box^+\neg(L=\gamma\wedge\Box\varphi_{_\gamma}\emptycommand)$.
\end{enumerate}
\end{proof}

\begin{lemma}\label{Lemma-2st}
Let $\beta \in K_0$ and for each $\gamma\geq \beta$, we have
$\HA\vdash\gamma\,\R\, L\ra \varphi_{_\gamma}\emptycommand$. Then for
each $\gamma\geq \beta$, we have
$\HA\vdash\exists{x}{F}(x)=\gamma\ra\Box\varphi_{_\gamma}\emptycommand$.
\end{lemma}
\begin{proof}
By  \Cref{Lemma-1.7st}, $\PA\vdash\exists{x}{F}(x)=\gamma\ra \Box^+\neg(L=\gamma\wedge\Box\varphi_{_\gamma}\emptycommand)$. 
Then   \Cref{Lemma-Conservativity of HA} implies 
$$
\HA\vdash\exists{x}{F}(x)=\gamma\ra \Box^+\neg(L=\gamma\wedge\Box\varphi_{_\gamma}\emptycommand)
 $$
Hence
\begin{align*}
\HA\vdash\exists{x}{F}(x)=\gamma
&
\ra\Box^+(\Box\varphi_{_\gamma}\emptycommand\ra L\not=\gamma)\\
& \ra\Box^+(\Box\varphi_{_\gamma}\emptycommand\ra L\succ \gamma) \tag{by $\Sigma_1$-completeness and
 \Cref{Lemma-Properties of Solovay's Function}} \\
& \ra\Box^+(\Box\varphi_{_\gamma}\emptycommand\ra \gamma\,\R\, L)
\tag{by \Cref{Lemma-1.5st} and \Cref{Lemma-Properties of Solovay's Function}}\\
& \ra\Box(\Box\varphi_{_\gamma}\emptycommand\ra \gamma\,\R\, L)
\tag{by $\Pi_2$-conservativity of $\PA$ over $\HA$}\\
& \ra\Box(\Box\varphi_{_\gamma}\emptycommand\ra\varphi_{_\gamma}\emptycommand) \tag{by hypothesis} \\
& \ra\Box\varphi_{_\gamma}\emptycommand \tag{L\"{o}b's axiom}
\end{align*}
\end{proof}

\begin{lemma}\label{Lemma-3st}
Let $\beta \in K_0$ and let for each
$\gamma\geq \beta$, we have $\HA\vdash\gamma\,\R\, L
\ra\varphi_{_\gamma}\emptycommand$. Then for each $\gamma\geq\beta$ 
 and  $B\in {\sf Sub}(\Gamma)\cap\TNNIL$,  $\gamma\Vdash B$ implies
$\HA\vdash\exists{x}{F}(x)\!=\!\gamma\ra B\emptycommand$.
\end{lemma}

\begin{proof}
We prove this by induction on the frame $(K,\lneqq)$ with reverse order. 
Let some $\gamma\geq\beta$ and 
as  the (first) induction hypothesis, assume  that for
each $\gamma_0\gneqq \gamma$ and $B\in {\sf Sub}(\Gamma)\cap\TNNIL$, if $\gamma_0\Vdash B$, then
${\HA\vdash\exists{x}F(x)=\gamma_0\ra B\emptycommand}$. We will show
that for each $B\in {\sf Sub}(\Gamma)\cap\TNNIL$, if $\gamma\Vdash B$, then
${\HA\vdash\exists{x}F(x)=\gamma\ra B\emptycommand}$. We prove this by a (second)
induction on the complexity of $B\in {\sf Sub}(\Gamma)\cap\TNNIL$. Let some  $B\in {\sf Sub}(\Gamma)\cap\TNNIL$
 and $\gamma\Vdash B$ and 
as the (second) induction
hypothesis,  assume that  for each
$C\in {\sf Sub}(\Gamma)\cap\TNNIL$ with lower complexity than \nolinebreak$B$ (i.e. $C$ is a strict sub-formula of $B$) 
such that $\gamma\Vdash
C$, we have $\HA\vdash\exists{x}F(x)=\gamma\ra C\emptycommand$. We
will show  $\HA\vdash\exists{x}F(x)=\gamma\ra B\emptycommand$. We
have following cases.
\begin{itemize}[leftmargin=*]
\item $B$ is atomic. It is trivial by definition of $B\emptycommand$.
\item $B$ is conjunction or disjunction. The result follows easily by (second) induction hypothesis.
\item $B=\Box C$. Suppose that $\gamma\Vdash\Box C$. Then by
definition of $\varphi_{_\gamma}$, we have
$\HA\vdash\Box\varphi_{_\gamma}\emptycommand\ra\Box C\emptycommand$.
Now the result is a consequence of  \Cref{Lemma-2st}.
\item $B=C\ra D\in {\sf Sub}(\Gamma)\cap\TNNIL$ and $C$ does not have an occurrence of
implication which is not in the scope of any box. Suppose that
$\gamma\Vdash C\ra D$. There are two sub-cases.
\begin{enumerate}[leftmargin=*]
\item $\gamma\Vdash C$. Then, by (second) induction hypothesis, we can derive
$\HA\vdash\exists{x}{F}(x)=\gamma\ra D\emptycommand$, and hence
$\HA\vdash\exists{x}{F}(x)=\gamma\ra (C\emptycommand\ra
D\emptycommand)$.
\item $\gamma\nVdash C$. Then  \Cref{Lemma-1st} implies that $\HA\vdash
(L=\gamma\wedge\Box\varphi_{_\gamma}\emptycommand)\ra\neg C\emptycommand$,
and by  \Cref{Lemma-2st}, ${\HA\vdash L=\gamma\ra\neg
C\emptycommand}$. Now we have $\HA\vdash C\emptycommand\ra L\not=\gamma$
and hence 
$${\HA\vdash(\exists{x}{F}(x)=\gamma\wedge
C\emptycommand)\ra L\not=\gamma}$$ 
So ${\PA\vdash(\exists{x}{F}(x)=\gamma\wedge C\emptycommand)\ra L\not=\gamma}$. 
Then by  \Cref{Lemma-Properties of Solovay's Function} \cref{3Lemma-Properties of Solovay's Function} 
$${\PA\vdash(\exists{x}{F}(x)=\gamma\wedge
C\emptycommand)\ra L\succ \gamma}$$
 On the other hand, $C$ is
implication-free and $C\emptycommand$ is $\Sigma_1$, so by 
\Cref{Lemma-Conservativity of HA}, we have
$\HA\vdash(\exists{x}{F}(x)=\gamma\wedge C\emptycommand)\ra
L\succ \gamma$. For arbitrary $\gamma_0\gneqq \gamma$, we have
$\gamma_0\Vdash C\ra D$. So by the first induction hypothesis, we
can derive $\HA\vdash\exists{x}F(x)=\gamma_0\ra (C\ra
D)\emptycommand$. By definition of $\gamma\prec  L$, we have
$\HA\vdash\gamma\prec  L\ra (C\ra D)\emptycommand$. Hence
$\HA\vdash(\exists{x}{F}(x)=\gamma\wedge C\emptycommand)\ra
(C\emptycommand\ra D\emptycommand)$, which implies
$\HA\vdash\exists{x}{F}(x)=\gamma\ra (C\emptycommand\ra
D\emptycommand)$.
\end{enumerate}
\end{itemize}
\end{proof}

\begin{lemma}\label{Lemma-4st}
For any $\alpha\in K$, \ $\HA\vdash
\alpha\,\R\, L \ra\varphi_{_\alpha}\emptycommand$.
\end{lemma}
\begin{proof}
Our proof is by reverse induction on  the frame $(K,\lneqq)$. As
the induction hypothesis, assume that for each $\beta \gneqq \alpha$, we
have $\HA\vdash\beta\,\R\, L\ra\varphi_{_\beta}\emptycommand$. For each
$\beta$ with $\alpha\,\R\,\beta$, by definition of $\varphi_{_\alpha}$,
we have $\beta\Vdash\varphi_{_\alpha}$. Hence by the induction
hypothesis and  \Cref{Lemma-3st},
$\HA\vdash\exists{x}{F}(x)=\beta\ra\varphi_{_\alpha}\emptycommand$.
Then
$\HA\vdash\bigvee_{\alpha\,\R\,\beta}\exists{x}{F}(x)=\beta\ra\varphi_{_\alpha}\emptycommand$,
which implies $\HA\vdash \alpha\,\R\, L\ra\varphi_{_\alpha}\emptycommand$.
\end{proof}

As a direct consequence of \Cref{Lemma-2st} and \Cref{Lemma-4st} we have the following result.

\begin{corollary}\label{corollar-4st&2st}
For any $\alpha\in K_0$,  $\HA\vdash \exists{x}F(x)=\alpha\to\Box\varphi_{_\alpha}\emptycommand$.
\end{corollary}


\begin{lemma}\label{Lemma-4.5st}
Let $\alpha\in K$ and $ \mathbb{N}\models L=\alpha$. Then $\HA\vdash\alpha\prec  L\ra \alpha\,\R\, L$.
\end{lemma}
\begin{proof}
Assume $\alpha\lneqq \beta$ and $\alpha\NR\beta$. 
\Cref{Lemma-4st} implies that for each $\gamma\geq \alpha$,
$\HA\vdash\gamma\,\R\, L\ra\varphi_{_\gamma}\emptycommand$. So  
\Cref{Lemma-1.5st} implies  $\HA\vdash L=\alpha\ra \Box^+
L\neq\beta$. Then  $\mathbb{N}\models L=\alpha\ra
\Box^+L\neq\beta$, which implies $\PA\vdash L\neq\beta$. Hence
$\PA\vdash \alpha\prec  L\ra\alpha\,\R\, L$. Now by
$\Pi_2$-conservativity of $\PA$ over $\HA$, $\HA\vdash
\alpha\prec  L\ra\alpha\,\R\, L$.
\end{proof}

\begin{lemma}\label{Lemma-5st}
Let $\alpha\in K$ and $\mathbb{N}\models L\!=\!\alpha$. Then
$\PA\cup\{L\!=\!\alpha,\Box\varphi_{_\alpha}\emptycommand\}$ is consistent.
\end{lemma}
\begin{proof}
Suppose not, i.e., 
$\PA\vdash(L=\alpha\wedge\Box\varphi_{_\alpha}\emptycommand)\ra\bot$.
Then $\PA\vdash\Box\varphi_{_\alpha}\emptycommand\ra L\not=\alpha$. By
$\Sigma_1$-completeness of $\PA$, we have 
$\PA\vdash\exists{x}{F}(x)=\alpha$ and then by 
\Cref{Lemma-Properties of Solovay's Function},
$\PA\vdash\Box\varphi_{_\alpha}\emptycommand\ra L\succ \alpha$. By
$\Pi_2$-conservativity of $\PA$ over $\HA$,
$\HA\vdash\Box\varphi_{_\alpha}\emptycommand\ra L\succ \alpha$. 
\Cref{Lemma-4.5st} implies
$\HA\vdash\Box\varphi_{_\alpha}\emptycommand\ra \alpha\,\R\, L$. Then by
 \Cref{Lemma-4st},
$\HA\vdash\Box\varphi\emptycommand_\alpha\ra\varphi\emptycommand_\alpha$.
Then  by L\"{o}b's theorem,
$\HA\vdash\varphi_{\alpha}\emptycommand$. Hence
$\HA\vdash\Box\varphi_{_\alpha}\emptycommand$. That implies $PA\vdash
L\not=\alpha$, a contradiction with $\mathbb{N}\models L=\alpha$.
\end{proof}

\begin{theorem}\label{Lemma-limit is root}
$\mathbb{N}\models L=\alpha_0$.
\end{theorem}
\begin{proof}
Suppose not, i.e., $\mathbb{N}\models L\not=\alpha_0$. Then by
 \Cref{Lemma-Properties of Solovay's Function} 
 \cref{2Lemma-Properties of Solovay's Function,3Lemma-Properties of Solovay's Function},
$\mathbb{N}\models L=\alpha$, for some $\alpha>\alpha_0$. By
 \Cref{Lemma-1.7st},
$\mathbb{N}\models\exists{x}F(x)=\alpha\ra\Box^+
\neg(L=\alpha\wedge\Box\varphi_{_\alpha}\emptycommand)$. This implies
$\PA\vdash\neg(L=\alpha\wedge\Box\varphi_{_\alpha}\emptycommand)$, a
contradiction with  \Cref{Lemma-5st}.
\end{proof}

\begin{corollary}\label{corollar-limit is root}
For any $\alpha\leq\beta\in K$ we have $(L=\alpha\wedge \varphi_{_\alpha}\emptycommand) \rhd (L=\beta\wedge \varphi_{_\beta}\emptycommand)$.
\end{corollary}
\begin{proof}
We should show that $\mathbb{N}\models (L=\alpha\wedge \varphi_{_\alpha}\emptycommand) \rhd (L=\beta\wedge \varphi_{_\beta}\emptycommand)$.
This is a direct consequence of \Cref{Lemma-limit is root}, \Cref{Lemma-1.5st Properties of Solovay's Function}
and \Cref{Lemma-Properties of Solovay's Function.4}.
\end{proof}

\subsection{Proof of the main theorem}
In this subsection, we will prove \Cref{Theorem-Main tool}.

With the general method of constructing Kripke models for $\HA$, invented by Smory\'nski \cite{Smorynski-Troelstra},
 interpretability of theories containing $\PA$ plays an important role
in constructing Kripke models of \nolinebreak$\HA$.  
%
%

\begin{definition}\label{Definition-Iframe}
A triple $\mathcal{I}:=(K,<,T)$ is called an I-frame
iff it has the following properties:
\begin{itemize}
\item $(K,<)$ is a finite tree,
\item $T$ is a function from $K$ to arithmetical r.e. consistent
theories containing $\PA$,
\item if $\beta<\gamma$, then $T_\beta$ interprets $T_\gamma$  $(\, T_\beta\rhd T_\gamma\,)$.
\end{itemize}
\end{definition}

\begin{theorem}\label{Theorem-Smorynski's general method of Kripke model construction}
For every I-frame $\mathcal{I}:=(K,<,T)$ there exists
a first-order Kripke model  $\kcal={(K,<,\mathfrak{M})}$ such that  
$\kcal\Vdash \HA$ and moreover $\mathfrak{M}(\alpha)\models T_\alpha$,  for any $\alpha\in K$.
 Note that both of the I-frame and  Kripke model are sharing the same frame $(K,<)$.
\end{theorem}
\begin{proof} See \cite[page~372-7]{Smorynski-Troelstra}. For more detailed proof of a generalization of this theorem, see
\cite[Theorem~4.8]{ArMo14}.
\end{proof}

\begin{lemma}\label{Lemma-Kripke simulation}
For any  $\beta\in K_0$, let $T_\beta:=\PA\cup \{L\!=\!\beta\}$
and define $\mathcal{I}:=(K_0,<_0,T)$. Then
\begin{enumerate}
\item $\mathcal{I}$ is an I-frame.
\item There exists a first-order Kripke model
$\mathcal{K}_1:=(K_0,<_0,\mathfrak{M})$ of $\HA$ such that
for all $\beta\in K_0$ and  $B\in {\sf Sub}(\Gamma)$, $\kcal_0,\beta\Vdash B$ iff
$\kcal_1,\beta\Vdash \sigma_{_{\sf HA}}(B)$.
\end{enumerate}
\end{lemma}
\begin{proof}
\begin{enumerate}[leftmargin=*]
\item  \Cref{corollar-limit is root} implies that for each
$\alpha\leq\beta$, $T_\alpha\rhd T_\beta$. Moreover,  
\Cref{Lemma-5st} implies that $T_{\alpha_0}$ is consistent. These 
finish the requirements of $\mathcal{I}$ to be an $I$-frame.
\item By  \Cref{Theorem-Smorynski's general method of Kripke model
construction}, we can find a first-order Kripke model
$\mathcal{K}_1:=(K_0,<_0,\mathfrak{M})$, such that for all
$\beta\in K_0$,  $\mathfrak{M}(\beta)\models T_\beta$ and 
$\mathcal{K}_1\Vdash\HA$. Now we prove, by induction on the
complexity of $B\in {\sf Sub}(\Gamma)$, that for all $\beta\in K_0$,
$\kcal_0,\beta\Vdash B$ iff $\kcal_1,\beta\Vdash \sigma_{_{\sf HA}}(B)$.
\begin{itemize}[leftmargin=*]
\item $B$ is atomic. Then by definition,
$\sigma_{_{\sf HA}}(B):=\bigvee_{\gamma\Vdash B}L\succeq\gamma$. For left to
right direction, let $\beta\Vdash B$. By second part  of
\Cref{Lemma-Properties of Solovay's Function},
$T_\beta\vdash\bigvee_{\gamma\Vdash B}L\succeq\gamma$ and hence
$\mathfrak{M}(\beta)\models\bigvee_{\gamma\Vdash B}L\succeq\gamma$. Since
$\bigvee_{\gamma\Vdash B}L\succeq\gamma$ is $\Sigma_1$-formula,
by \Cref{Lemma-Sigma-local-global}, we  have
 $\beta\Vdash\bigvee_{\gamma\Vdash B}L\succeq\gamma$. 
For the other
way around, let $\beta\nVdash B$. Then 
\Cref{Lemma-Properties of Solovay's Function} \cref{2Lemma-Properties of Solovay's Function} implies
$T_\beta\vdash\neg L\succeq\gamma$, for all $\gamma\Vdash B$. 
This implies that $\mathfrak{M}(\beta)\not\models L\succeq\gamma$,
which by use of \Cref{Lemma-Sigma-local-global} implies  $\beta\nVdash L\succeq\gamma$. So
$\beta\nVdash\bigvee_{\gamma\Vdash B}L\succeq\gamma$.
\item $B$ is conjunction, disjunction or implication. Result follows easily by
induction hypothesis and inductive definition of $\Vdash$.
\item $B=\Box C$. Let $\beta\Vdash B$. Then by definition
of $\varphi_{_\beta}$ and \Cref{corollar-4st&2st},   $T_\beta\vdash \sigma_{_{\sf HA}}(\Box C)$,
and so $\mathfrak{M}(\beta)\models \sigma_{_{\sf HA}}(\Box C)$. By \Cref{Lemma-Sigma-local-global}, 
$\beta\Vdash \sigma_{_{\sf HA}}(\Box C)$. For the other way around, suppose
$\beta\nVdash\Box C$. Then  \Cref{Lemma-2st Properties of
Solovay's Function} implies $T_\beta\vdash\neg \sigma_{_{\sf HA}}(\Box C)$ and hence 
$\mathfrak{M}(\beta)\not\models \sigma_{_{\sf HA}}(\Box C)$. 
Then \Cref{Lemma-Sigma-local-global} implies  $\beta\nVdash \sigma_{_{\sf HA}}(\Box C)$.
\end{itemize}
\end{enumerate}
\end{proof}

\begin{proof}{(of \Cref{Theorem-Main tool})}
Let  $\sigma$ be the substitution  that we have defined at the beginning of this section and 
$\kcal_1$ be as we have by \Cref{Lemma-Kripke simulation}. Then the assertion of \Cref{Lemma-Kripke simulation} implies that 
for any $A\in {\sf Sub}(\Gamma)$ we have:
\begin{equation*}
\kcal_0,\alpha\Vdash A\quad  \text{ iff }\quad \kcal_1,\alpha\Vdash  \sigma_{_{\sf HA}}(A)
\end{equation*}
\end{proof}

\section{The $\Sigma_1$-provability logic of $\HA$}\label{Section-Sigma Provability}
In this section, we will show that  $\mathcal{PL}_\sigma(\HA)=\lles$ (see \Cref{Definition-Provability Logic}).
 Moreover, we will show that $\lles$ is decidable.  As a by-product of \Cref{Theorem-Main
tool}, we show that $\HA+\Box\bot$ has de Jongh property.
Before we continue with the statement and proof of soundness and completeness theorems, let us  
apply our techniques presented in this paper  to \Cref{example01,example02} in \cref{sec-introduction}.

\begin{example}\label{example1}\em 
Let $A=\Box (p\vee q)\to(\Box p\vee\Box q)$. 
We will refute the modal proposition $A$ from the provability logic (and $\Sigma_1$-provability logic) of $\HA$. 

First of all note that the Kripke model $\kcal_0$ from \Cref{example01} is a counter-model for $A$. 
Then by \Cref{Theorem-Main tool}, there exists some first-order Kripke model $\kcal_1\Vdash\HA$ and some 
$\Sigma_1$-substitution \nolinebreak$\sigma$ such that $\kcal_1\nVdash \sigma_{_{\sf HA}}(A)$. Hence we have 
$\HA\nvdash \sigma_{_{\sf HA}}(A)$, in other words, if we define $B:=\sigma(p)$ and $C:=\sigma(q)$, then 
we have $\HA\nvdash\Box(B\vee C)\to(\Box B\vee\Box C)$. \\
\end{example}

\begin{example}\label{example2}\em 
In this example, we show that how to refute $A=\neg\neg\Box(\neg\neg p\to  p)\to\Box(\neg\neg p\to p)$ from 
the provability logic of $\HA$  and also  from the $\Sigma_1$-provability logic of $\HA$. 
First we compute the $\TNNIL$-approximation of $A$, i.e. $A^+$. By $\TNNIL$-algorithm  
(\Cref{subsubsec-TNNIL^-algorith}),  $A^+\!=\!{\Box(p\vee\neg p)\vee\neg\Box(p\vee \neg p)}$.
Then $\kcal_0$ from \Cref{example02} is a countermodel for $A^+$.
Then by \Cref{Theorem-Main tool}, there exists some first-order Kripke model $\kcal_1\Vdash\HA$ and some 
$\Sigma_1$-substitution $\sigma$ such that $\kcal_1\nVdash \sigma_{_{\sf HA}}(A^+)$. Hence 
$\HA\nvdash \sigma_{_{\sf HA}}(A^+)$.  Since $\HA\vdash \sigma_{_{\sf HA}}(\Box A\lr\Box A^+)$
(\Cref{corollar-NNIL properties}), we can deduce 
$\HA\nvdash \sigma_{_{\sf HA}}(A)$.
\end{example}

\subsection*{Soundness}
\begin{theorem}\label{Theorem-Soundness}
$\lles$ is sound for $\Sigma_1$-arithmetical interpretations in
$\HA$, i.e. $\lles\subseteq\PLS(\HA)$.
\end{theorem}
\begin{proof}
We must show that $\Sigma_1$-interpretations of all axioms of $\lles$
hold in $\HA$. For a proof that $\Sigma_1$-interpretations of axioms
$\Box A\ra\Box B$ with $A\brt B$, hold in $\HA$, see
\cite{Visser02}, Theorem 10.2. The other axioms are well-known or
obvious, except for the Extended Leivant's principle, which holds by 
\Cref{Theorem-Soundness of HA for lle+}.
\end{proof}

The following corollary, shows that $\LC$   captures the $\TNNIL$ part of the 
theory $\lles$. In other words, as far as we are interested in $\TNNIL$ propositions, 
we can work with the rather simple theory $\LC$ instead of $\lles$.
\begin{corollary}\label{Corollary-TNNIL-equaipotency}
For any $\TNNIL$ modal proposition $A$, $\LC\vdash A$ if and only if $\lles\vdash A$.
\end{corollary}
\begin{proof}
The deduction from  left to right is  by 
\Cref{Theorem-TNNIL Conservativity of LC over LLe+} and the fact that $\lles\vdash \lle$, which holds by definition of $\lles$. 
For the  right to left direction, assume that $\LC\nvdash A$.  
Then by \Cref{Theorem-Propositional Completeness LC}, 
there exists some Kripke model $\kcal_0$ such that 
$\kcal_0\nVdash A$.
Then by   \Cref{Theorem-Main tool}, we can find
 some $\Sigma_1$-substitution $\sigma$ and some first-order Kripke model $\kcal_1$, such that  $\kcal_1\Vdash \HA$ 
 and $\kcal_1\nVdash\sigma_{_{\sf HA}}(A)$. Hence 
$\HA\nvdash \sigma_{_{\sf HA}}(A)$.  Now \Cref{Theorem-Soundness} implies that $\lles\nvdash A$, as desired.
\end{proof}

\subsection*{Completeness}
\begin{theorem}\label{Theorem HA-Completeness}
$\lles$ is complete for $\Sigma_1$-arithmetical interpretations in $\HA$, i.e.
$$\mathcal{PL}_\sigma(\HA)\subseteq\lles$$
\end{theorem}
\begin{proof}
Let $\lles\nvdash A$. Then by  \Cref{corollar HA-NNIL
approximation is propositionally equivalent}, $\lles\nvdash A^-$.
Then by  \Cref{Theorem-NNIL Crucial Properties} \cref{1Theorem-NNIL Crucial Properties},
$\lles\nvdash (A^-)^*$, which implies $\lles\nvdash A^+$. Hence
$\lle\nvdash A^+$ and by \Cref{Theorem-TNNIL Conservativity of LC over LLe+},  $\LC\nvdash A^+$.
Then by \Cref{Theorem-Propositional Completeness LC}, there exists some Kripke model $\kcal_0$ such that 
$\kcal_0\nVdash A^+$.
Then by   \Cref{Theorem-Main tool}, we can find
 some $\Sigma_1$-substitution $\sigma$ and some first-order Kripke model $\kcal_1$, such that  $\kcal\Vdash \HA$ 
 and $\kcal\nVdash\sigma_{_{\sf HA}}(A^+)$. Hence 
$\HA\nvdash \sigma_{_{\sf HA}}(A^+)$.  Now  \Cref{corollar-NNIL properties} \cref{1corollar-NNIL properties}
implies $\HA\nvdash \sigma_{_{\sf HA}}(A)$, as desired.
\end{proof}

Although the axioms of the theory $\lles$  sounds very complicated, however we have the following surpring result.

\begin{theorem}
The $\Sigma_1$-provability logic of $\HA$ {\em ($\lles$)}  is decidable.
\end{theorem}
\begin{proof}
Let $A$ be a given modal proposition. We explain how to decide
$\lles\vdash A$ or $\lles\nvdash A$. First by $\TNNIL$ algorithm,
compute $A^+$. Then by  \Cref{Corollary-decidability of LC},
we can decide  $\LC\vdash A^+$. If $\LC\vdash A^+$ , we say `yes"
to $\lles\vdash A$, and otherwise we say ``no" to $\lles\vdash A$.
\Cref{Corollary-TNNIL-equaipotency} guarantees validity of the algorithm.

\end{proof}
\begin{theorem}\label{Theorem-de jongh property for HA+boxbot}
$\HA+\bo\bot$ has the de Jongh property, i.e. for all non-modal
proposition $A$, $\IPC\vdash A$ iff for all arithmetical
substitution $\sigma$,  $\HA+\bo\bot\vdash\sigma(A)$.
\end{theorem}
\begin{proof}
If $\IPC\vdash A$, we apparently have $\HA\vdash \sigma(A)$, for
all $\sigma$, and hence $\HA+\bo\bot\vdash \sigma(A)$.

For the other way around, let $\IPC\nvdash A$. Hence by 
\Cref{Theorem-NNIL Crucial Properties} \cref{1Theorem-NNIL Crucial Properties}, $\IPC\nvdash A^*$.  Then
by  \Cref{Theorem-nonmodal conservativity of LC over IPC},
$\LC\nvdash\bo\bot\ra A^*$. This implies that
$\lle\nvdash\bo\bot\ra A^*$. Hence by  \Cref{Theorem-Main
tool}, there exists some substitution $\sigma$ such that
$\HA+\bo\bot\nvdash \sigma(A)$, as desired.
\end{proof}

\section*{Acknowledgement} 
We would like to thank Albert Visser for reading of the first draft of this paper and his 
very helpful comments, remarks and corrections.
Some results of  this paper is obtained  
during the second author's visit to the Department of Philosophy in Utrecht University in May-June 2011.
The second author is  thankful to Albert Visser who patiently 
answered to his questions and having very fruitful conversations. The second author is also very thankful to Jeroen Goudsmit
for his  helps, hospitality and conversations. 


\providecommand{\bysame}{\leavevmode\hbox to3em{\hrulefill}\thinspace}
\providecommand{\MR}{\relax\ifhmode\unskip\space\fi MR }
\providecommand{\MRhref}[2]{%
	\href{http://www.ams.org/mathscinet-getitem?mr=#1}{#2}
}
\providecommand{\href}[2]{#2}

\end{document}